\documentclass[11pt,reqno]{amsart}
\usepackage{amssymb,amsmath,amsfonts,amsthm,enumerate,enumitem,stmaryrd,tensor,mathtools,dsfont,upgreek,bbm,mathrsfs,nameref,xcolor,bm,tikz,stmaryrd,todonotes,graphicx,thmtools,microtype}
\usepackage[framemethod=TikZ]{mdframed}
\usetikzlibrary{arrows}
\usepackage[left=0.84 in, right=0.84 in, top=0.84 in, bottom=0.84 in]{geometry}

\usepackage{hyperref}
\usepackage{caption}
\usepackage{subcaption}

\numberwithin{equation}{subsection}
\setlist[enumerate]{leftmargin=0.2in}
\setlist[itemize]{leftmargin=0.2in}
\definecolor{popblue}{RGB}{55,115,255}
\definecolor{lightbl}{RGB}{155,205,255}
\definecolor{depthbl}{RGB}{145,215,255}
\definecolor{fancyre}{RGB}{225,55,115}
\definecolor{lightgr}{RGB}{230,255,230}
\definecolor{darkgre}{RGB}{25,105,25}
\definecolor{darkblu}{RGB}{15,75,185}
\definecolor{mellowy}{RGB}{225,225,35}

\renewcommand{\tilde}[1]{\widetilde{#1}}
\renewcommand{\Bar}{\overline}

\renewcommand{\S}{\mathbb{S}}
\newcommand{\R}{\mathbb{R}}
\newcommand{\N}{\mathbb{N}}
\newcommand{\Z}{\mathbb{Z}}

\newcommand{\C}{\mathbb{C}}

\newcommand{\T}{\mathbb{T}}

\newcommand{\imp}{\;\Rightarrow\;}
\newcommand{\m}{\mathrm}

\newcommand{\lv}{\lVert}
\newcommand{\rv}{\rVert}

\newcommand{\al}{\alpha}

\newcommand{\es}{\varnothing}

\newcommand{\ep}{\varepsilon}
\newcommand{\f}{\frac}
\newcommand{\sig}{\sigma}
\newcommand{\gam}{\gamma}
\newcommand{\del}{\delta}

\newcommand{\pd}{\partial}

\newcommand{\grad}{\nabla}
\newcommand{\bpm}{\begin{pmatrix}}
\newcommand{\epm}{\end{pmatrix}}

\renewcommand{\le}{\leqslant}
\renewcommand{\ge}{\geqslant}

\newcommand{\bnorm}[1]{\Big\lv#1\Big\rv}

\newcommand{\tnorm}[1]{\lv#1\rv}

\newcommand{\bp}[1]{\Big(#1\Big)}
\renewcommand{\sp}[1]{\big(#1\big)}
\newcommand{\tp}[1]{(#1)}

\newcommand{\babs}[1]{\Big|#1\Big|}

\newcommand{\tabs}[1]{|#1|}

\newcommand{\bsb}[1]{\Big[{#1}\Big]}

\newcommand{\tsb}[1]{[{#1}]}

\newcommand{\bcb}[1]{\Big\{{#1}\Big\}}
\newcommand{\tcb}[1]{\{{#1}\}}

\providecommand{\tbr}[1]{\langle #1 \rangle}

\renewcommand{\bf}[1]{\mathbf{#1}}
\newcommand{\ii}{\mathsf{i}}

\DeclareMathOperator{\supp}{supp}

\DeclareMathOperator{\sech}{sech}

\newtheorem{introthmcounter}{Theorem}

\declaretheoremstyle[
    headfont=\bfseries\rmfamily\color{popblue}, bodyfont=\normalfont, 
    mdframed={
        linewidth=0.8pt,
        linecolor=lightbl, backgroundcolor=lightbl!5,
        skipabove=4pt,
    }
]{blueStyle}

\declaretheoremstyle[
    headfont=\bfseries\rmfamily\color{popblue}, 
    bodyfont=\itshape,
]{boringStyle}

\declaretheorem[style=boringStyle, numberwithin=section, name=Proposition]{propC}
\declaretheorem[style=boringStyle, sibling=propC, name=Lemma]{lemC}
\declaretheorem[style=boringStyle, sibling=propC, name=Theorem]{thmC}
\declaretheorem[style=boringStyle, sibling=propC, name=Corollary]{coroC}
\declaretheorem[style=boringStyle, sibling=propC, name=Remark]{rmkC}

\declaretheoremstyle[
    headfont=\bfseries\rmfamily\color{darkgre}, bodyfont=\normalfont, 
    mdframed={
        linewidth=0.8pt,
        linecolor=darkgre!35, backgroundcolor=lightgr!7,
        skipabove=4pt,
    }
]{greenStyle}

\declaretheoremstyle[
    headfont=\bfseries\rmfamily\color{fancyre}, bodyfont=\normalfont, 
    mdframed={
        linewidth=0.8pt,
        linecolor=fancyre!35, backgroundcolor=mellowy!6,
        skipabove=4pt,
    }
]{redStyle}
\declaretheorem[style=boringStyle, sibling=propC, name=Definition]{defnC}
\declaretheorem[style=boringStyle, sibling=introthmcounter, name=Theorem]{introthm}

\DeclareFontFamily{U}{cbgreek}{}
\DeclareFontShape{U}{cbgreek}{m}{n}{
  <-6>    grmn0500
  <6-7>   grmn0600
  <7-8>   grmn0700
  <8-9>   grmn0800
  <9-10>  grmn0900
  <10-12> grmn1000
  <12-17> grmn1200
  <17->   grmn1728
}{}
\DeclareFontShape{U}{cbgreek}{bx}{n}{
  <-6>    grxn0500
  <6-7>   grxn0600
  <7-8>   grxn0700
  <8-9>   grxn0800
  <9-10>  grxn0900
  <10-12> grxn1000
  <12-17> grxn1200
  <17->   grxn1728
}{}

\DeclareRobustCommand{\Qoppa}{%
  \text{\usefont{U}{cbgreek}{\normalorbold}{n}\symbol{21}}%
}
\DeclareRobustCommand{\sampi}{%
  \text{\usefont{U}{cbgreek}{\normalorbold}{n}\symbol{27}}%
}
\DeclareRobustCommand{\Sampi}{%
  \text{\usefont{U}{cbgreek}{\normalorbold}{n}\symbol{23}}%
}

\makeatletter
\newcommand{\normalorbold}{%
  \ifnum\pdf@strcmp{\math@version}{bold}=\z@ bx\else m\fi
}
\makeatother

\author{R\u{a}zvan-Octavian Radu}
\address{
	Department of Mathematics\\
	Princeton University\\
	Princeton, NJ 08544, USA
}
\email[R.-O. Radu]{rradu@math.princeton.edu}
\thanks{R.-O. Radu was partially supported by a Charlotte Elizabeth Procter Fellowship and by an NSF Grant (DMS \#2350252)}
\author{Noah Stevenson}
\address{
	Department of Mathematics\\
	Princeton University\\
	Princeton, NJ 08544, USA
}
\email[N. Stevenson]{stevenson@princeton.edu}
\thanks{N. Stevenson was supported by an NSF Graduate Research Fellowship}

\title[Desingularization of steady rotating vortex patches]{Desingularization of nondegenerate rotating vortex patches}

\begin{document}
\begin{abstract}
    This paper analyzes the space of steady rotating solutions to the two-dimensional incompressible Euler equations nearby vortex patch solutions satisfying a natural nondegeneracy condition. We address the question of desingularization and prove that such vortex patch states are the limit of rotating Euler solutions that are smooth to infinite order, have compact vorticity support, and respect dihedral symmetry. Our nondegeneracy condition is proved to be satisfied by Kirchhoff ellipses and along the local bifurcation curves emanating from the Rankine vortex.

    The construction, which is based on a local stream function formulation in a tubular neighborhood of the patch boundary, is a synthesis of delicate analysis on thin domains, nonlinear a priori estimates, and a custom version of Newton's method. Our techniques are robust enough to additionally allow us to construct exotic families of singular rotating vortex patch-like solutions nearby a given nondegenerate state. To the best of the authors' knowledge, this work constitutes the first desingularization procedure applicable to general families of steady rotating vortex patches. 
\end{abstract}
\maketitle
\section{Introduction}

\subsection{The Euler equations in the plane and in the disk, steady rotating solutions, and elliptical vortices}\label{subsection on opening}

In this paper we study inviscid incompressible fluids whose motion is described by solutions to the two-dimensional Euler system of equations. Two fluid geometries are within the scope of our analysis: the planar case of a fluid defined on all of $\R^2$ and the disk case of a fluid confined to $\Bar{B(0,R)}$ for some $R>0$. To consolidate notation, we shall make the convention that $\Bar{B(0,R)} = \R^2$ and $\pd B(0,R) = \es$ when $R = \infty$.

In Eulerian variables, the system of Euler equations relates the velocity vector $u:\R^+\times\Bar{B(0,R)}\to\R^2$ and the scalar pressure $p:\R^+\times\Bar{B(0,R)}\to\R$ via the bulk equations
\begin{equation}\label{dynamical planar euler equations}
  \pd_t u + u\cdot\grad u + \grad p = 0\text{ and }\grad\cdot u = 0
\end{equation}
and, when $R<\infty$, the no-penetration boundary condition
\begin{equation}\label{dynamical no penetration condition}
    u(t,x)\cdot x/|x| = 0\text{ for all }x\in\pd B(0,R),\;t\in\R^+.
\end{equation}
As is well-known (see, e.g., \cite{MR861488,MR1867882,MR1217252,MR1245492,MR4475666}), the above equations~\eqref{dynamical planar euler equations} and~\eqref{dynamical no penetration condition} can be equivalently formulated in terms of the vorticity $\tilde{\omega}:\R^+\times\Bar{B(0,R)}\to\R$ and the stream function $\psi:\R^+\times\Bar{B(0,R)}\to\R$ via the bulk equations
\begin{equation}\label{dynamical planar Euler, vorticity stream function formulation}
  \pd_t\tilde{\omega} + \grad^\perp\psi\cdot\grad\tilde{\omega} = 0\text{ and }\Delta\psi = \tilde{\omega},
\end{equation}
where $\grad^\perp = \tp{-\pd_2,\pd_1}$ denotes the counterclockwise rotated gradient, and the boundary condition
\begin{equation}\label{vorticity stream no penetration condition}
    \grad^\perp\psi(t,x)\cdot x/|x| = 0\text{ for all }x\in\pd B(0,R),\;t\in\R.
\end{equation}

One recovers the velocity and pressure of~\eqref{dynamical planar euler equations} from the vorticity and stream function of~\eqref{dynamical planar Euler, vorticity stream function formulation} and~\eqref{vorticity stream no penetration condition} via the relationships
\begin{equation}\label{recovery identities}
  u = \grad^\perp\psi\text{ and }-\Delta p = \grad\cdot\grad\cdot\tp{u\otimes u}.
\end{equation}

We are interested in the class of \emph{steady rotating} solutions to equations~\eqref{dynamical planar Euler, vorticity stream function formulation} and~\eqref{vorticity stream no penetration condition}; these are time-independent when viewed in a coordinate system rotating at a constant angular velocity. Let us fix an arbitrary angular velocity $\Omega\in\R$ and make the ansatz that there exists $\omega,\Psi:\Bar{B(0,R)}\to\R$ such that for all $t\in\R^+$ and $x\in\Bar{B(0,R)}$ we have
\begin{equation}\label{the relationship to the steady rotating unknowns}
  \tilde{\omega}(t,x) = \omega(\bf{R}(\Omega t)x)\text{ and }\psi(t,x) = \Psi(\bf{R}(\Omega t)x) + \Omega|\bf{R}(\Omega t)x|^2/2
\end{equation}
where for $s\in\R^+$ we have defined the rotation matrix
\begin{equation}\label{the rotation matrix}
  \bf{R}(s) = \bpm\cos(s)&-\sin(s)\\\sin(s)&\cos(s)\epm.
\end{equation}
The corresponding bulk equations for the \emph{time-independent} scalars $\omega$ and $\Psi$, called the vorticity and the relative stream function, respectively, are then computed to be the following:
\begin{equation}\label{the steady rotating equations for vorticity and relative stream function}
  \grad^\perp\Psi\cdot\grad\omega = 0\text{ and }\Delta\Psi = \omega - 2\Omega.
\end{equation}
When $R<\infty$, we additionally have the boundary condition
\begin{equation}\label{boundary condition for relative stream function}
    \grad^\perp\Psi(x)\cdot x/|x| = 0\text{ for all }x\in\pd B(0,R).
\end{equation}
On the other hand, when $R = \infty$, we supplement the second equation of~\eqref{the steady rotating equations for vorticity and relative stream function} with the condition $\grad\tp{\Psi + \Omega|\cdot|^2/2}(x)\to0$ as $|x|\to\infty$.

While the equations~\eqref{the steady rotating equations for vorticity and relative stream function} and~\eqref{boundary condition for relative stream function} were derived from the parent Euler system~\eqref{dynamical planar euler equations} under a sufficient amount of assumed smoothness, to properly cast our analysis and results we must interpret the equations more broadly in a sense of weak solutions. This is made precise by the following definition.
\begin{defnC}[Steady rotating weak solutions to the Euler system]\label{defn od steady rotating weak solutions to the planar Euler system}
  Let $R\in(0,\infty]$. We say that $\Omega\in\R$, $\Psi\in C^1\tp{\Bar{B(0,R)}}$, and $\omega\in\tp{L^1\cap L^\infty}\tp{\Bar{B(0,R)}}$ are a weak solution to equations~\eqref{the steady rotating equations for vorticity and relative stream function} and~\eqref{boundary condition for relative stream function} if the following hold.
  \begin{enumerate}
    \item For almost every $x\in\Bar{B(0,R)}$ we have satisfied
    \begin{equation}\label{weak equation 1}
      \grad\Psi(x) = -\Omega x - \int_{B(0,R)}\bf{K}_R\tp{x,y}^\perp\omega(y)\;\m{d}y
    \end{equation}
    with $\bf{K}_R$ the Biot-Savart kernel, given for $R<\infty$ by
    \begin{equation}\label{bio savart in a domain}
        \bf{K}_R(x,y) = \f{1}{2\pi}\f{\tp{x - y}^\perp}{|x - y|^2} + \f{1}{2\pi}\f{\tp{R^2y - |y|^2x}^\perp}{R^4 - 2R^2x\cdot y + |x|^2|y|^2}.
    \end{equation}
    When $R = \infty$, $\bf{K}_\infty$ is simply the above right hand side's first term.
    \item For all $f\in C^1_{\m{c}}\tp{\Bar{B(0,R)}}$ we have satisfied
    \begin{equation}\label{weak equation 2}
      \int_{B(0,R)}\omega(y)\grad^\perp\Psi(y)\cdot\grad f(y)\;\m{d}y = 0.
    \end{equation}
    \item For all $x\in\pd B(0,R)$ we have satisfied $\grad^\perp\Psi(x)\cdot x/|x| = 0$.
  \end{enumerate}
\end{defnC}

Weak solutions to equation~\eqref{the steady rotating equations for vorticity and relative stream function} in the sense of Definition~\ref{defn od steady rotating weak solutions to the planar Euler system} give rise to so-called Yudovich weak solutions (see Section 3.1 of~\cite{MR158189}, Section 8.2 of~\cite{MR1867882}, or Section 2.3 in~\cite{MR1245492}) to system~\eqref{dynamical planar euler equations}. Explicitly if $R\in(0,\infty]$, $\Omega\in\R$, $\Psi\in C^1\tp{\Bar{B(0,R)}}$, and $\omega\in\tp{L^1\cap L^\infty}\tp{\Bar{B(0,R)}}$ are a steady rotating weak solution in the sense of Definition~\ref{defn od steady rotating weak solutions to the planar Euler system} and we define the dynamic scalars $\tilde{\omega}$ and $\psi$ as in equation~\eqref{the relationship to the steady rotating unknowns} and let $u = \grad^\perp\psi$, then the following hold.
    \begin{itemize}
        \item We have the inclusions $\tilde{\omega}\in L^\infty\tp{[0,\infty);\tp{L^1\cap L^\infty}\tp{B(0,R)}}$ and $u\in L^\infty\tp{[0,\infty);\m{LL}\tp{\Bar{B(0,R)}}}$ where $\m{LL}$ denotes the collection of logarithmic-Lipschitz functions (see Section~\ref{section on notational conventions}).
        \item We have the identities
        \begin{equation}
            u = \int_{B(0,R)}\bf{K}_R(\cdot,y)\tilde{\omega}(y)\;\m{d}y,\;\tp{\grad^\perp\cdot u} = \tilde{\omega}\;\text{ and }\grad\cdot u = 0.
        \end{equation}
        in the sense of distributions on $B(0,R)$ where $\bf{K}_R$ is as in~\eqref{bio savart in a domain}. We also have, in the case $R<\infty$, the satisfaction of the no penetration boundary condition 
        \begin{equation}
            u(x)\cdot x/|x| = 0\text{ for all }x\in\pd B(0,R).
        \end{equation}
        \item For all $T>0$ and all $f\in C^1\tp{[0,T];C^1_{\m{c}}\tp{B(0,R)}}$ it holds that
        \begin{equation}
            \int_{B(0,R)}f(T,\cdot)\tilde{\omega}(T,\cdot) - \int_{B(0,R)}f(0,\cdot)\tilde{\omega}(0,\cdot) = \int_0^T\int_{B(0,R)}\tp{\pd_tf + u\cdot\grad f}\tilde{\omega}
        \end{equation}
    \end{itemize}

An important and well-studied class of solutions to the two-dimensional steady rotating system Euler equations are the \emph{steady rotating vortex patches}, which are also known as \emph{V-states}, see, e.g., \cite{DEEM_ZABUSKY,MR646163,MR734586,MR858995,MR3054601,MR3482335,MR3462104,MR3570233,MR3543554,MR3545942,XJM} and references therein. These are weak solutions to the steady rotating Euler equations in the sense of Definition~\ref{defn od steady rotating weak solutions to the planar Euler system} in which the vorticity is singular, being the discontinuous characteristic function of an open, bounded, and simply connected set with boundary belonging to some regularity class. Among this class of steady rotating vortex patches, there is an \emph{explicit} subfamily known as the Kirchhoff elliptical vortices (see Kirchhoff~\cite{Kirchhoff} or Article 159 of Chapter VII in Lamb~\cite{MR1317348}). Modulo the symmetries of the Euler equations, these form a one parameter family, that we enumerate in a particular way. For $\xi\in\R^+$ we define
\begin{equation}\label{the kirchhoff ellispses as sets}
    E^\xi = \tcb{x\in\R^2\;:\;x_1^2 + \tp{x_2/\tanh\tp{\xi}}^2<1}\subset\R^2\text{ and }\Omega^\xi = \tanh(\xi)/(1+\tanh(\xi))^2.
\end{equation}
The $\xi$-Kirchhoff vortex is a steady rotating (with angular velocity $\Omega^\xi$) weak solution (in the sense of Definition~\ref{defn od steady rotating weak solutions to the planar Euler system} with $R = \infty$) to~\eqref{the steady rotating equations for vorticity and relative stream function} with vorticity $\omega = \mathds{1}_{E^\xi}$ and, for some constant $c^\xi\in\R$, relative stream function $\Psi = \Psi^\xi$ satisfying
\begin{equation}\label{KEAF}
    \Psi^\xi\tp{x} = \tp{\bf{N}\ast\mathds{1}_{E^\xi}}(x) - \Omega^\xi\tabs{x}^2/2 - c^\xi\text{ for }x\in\R^2\text{ and }\Psi^\xi\tp{\pd E^\xi}=\tcb{0}
\end{equation}
with
\begin{equation}\label{defn of the newtonian potential}
    \bf{N}(x) = \tp{1/2\pi}\log\tabs{x}\text{ for }x\in\R^2\setminus\tcb{0}
\end{equation}
the Newtonian potential. In the extreme case $\xi=\infty$ we have that $E^\infty = B(0,1)$ is simply the open unit disk and $\Omega^\infty = 1/4$. Due to radial symmetry, $\omega = \mathds{1}_{B(0,1)}$ is the vorticity of a weak solution to~\eqref{the steady rotating equations for vorticity and relative stream function} not just for $\Omega = 1/4$, but for all $\Omega\in\R$. This family of circular vortex solutions are known as the Rankine vortices~\cite{RANKINE}.

A motivation for this work's forthcoming analysis is the following natural question: Are the class of singular solutions given by the steady rotating vortex patches able to be realized as limits of classical $C^\infty$-smooth solutions having the same qualitative properties, such as steady rotation and compact vorticity support, while respecting discrete symmetries? A positive answer to the previous desingularization question would show that these singular solutions are further from being pathological and, in a sense, are more physically meaningful.

In this paper we provide a partial positive answer to the preceding query. We formulate a notion of nondegeneracy for steady rotating vortex patches that is satisfied when a particular related linear operator (having the form of a compact perturbation of the identity) has a trivial kernel. When this condition is satisfied by a patch, we prove, among several other things, that the patch can be desingularized in the sense that it is the limit of smooth solutions to equations~\eqref{the steady rotating equations for vorticity and relative stream function} and~\eqref{boundary condition for relative stream function} that respect certain desirable qualitative properties.

It is clearly important to verify that there exist steady rotating patches that satisfy our nondegeneracy condition. We provide two uncountable families of examples. The first are the Kirchhoff elliptical vortices; more precisely, we show that there exists a countable set $\mathcal{N}\subset\R^+$ such that for all $\xi\in\R^+\setminus\mathcal{N}$ the $\xi$-Kirchhoff vortex of equations~\eqref{the kirchhoff ellispses as sets} and~\eqref{KEAF} is nondegenerate and hence can be desingularized. The second family of examples is given by the $m$-fold symmetric V-states of Burbea~\cite{MR646163} that locally bifurcate from the Rankine vortex (these are discussed in detail in Section~\ref{subsection on perturbations of the rankine vortex}). The authors conjecture -- but make no attempt to prove in this work -- that our nondegeneracy condition, that is precisely stated in the fourth item of Definition~\ref{defn of admissible patches} below, is a generic property of steady rotating vortex patches.

\subsection{Survey of prior work}\label{subsection on literature review}

The incompressible Euler equations~\eqref{dynamical planar euler equations} are among the first PDEs to be written down and have been at the focal point of intense research for nearly a quarter of a millennium. As such, we make no attempt to give a comprehensive literature review; rather we focus just on the results for the two-dimensional equations that are most closely related to our own.

The local and global well-posedness of the initial value problem for system~\eqref{dynamical planar euler equations} in the class of classical solutions is due to Lichtenstein~\cite{MR1544733}, H\"{o}lder~\cite{MR1545431}, and Wolibner~\cite{MR1545430}; modern treatments of these facts are found in Kato and Ponce~\cite{MR864654}. The corresponding global well-posedness result in a class of weak solutions for the vorticity formulation was proved by Yudovich~\cite{MR158189}; see also Bardos~\cite{MR333488}. Extensions to larger classes of (more singular) vortex data have been made by Yudovich~\cite{MR1312975}, Serfati~\cite{SERFATI}, and Elgindi, Murray, and Said~\cite{ELGINDI_MURRAY_SAID}.

Among the most well-known explicit solutions to~\eqref{dynamical planar euler equations} are the Kirchhoff elliptical vortices. Regimes of spectral stability and instability for the Kirchhoff elliptical vortices were first established by Love~\cite{MR1576343}. Promotion to nonlinear stability came with the work of Tang~\cite{MR911087} and Wan~\cite{MR861881}. In the complementary regime, nonlinear instability was proved by Guo, Hallstrom, and Spirn~\cite{MR2039699}. Pulvirenti and Wan~\cite{MR795112} established nonlinear stability of the Rankine vortices.

The distinguishing feature of two-dimensional Euler flow governed by equation~\eqref{dynamical planar Euler, vorticity stream function formulation} is the transportation of the vorticity along the velocity vector field; hence, when the initial value problem is given an initial vorticity that is the characteristic function of a set, this property is propagated in time. These weak solutions are known as \emph{vortex patches}. Global well-posedness of vortex patches in spaces of sufficient boundary regularity was proved by Chemin~\cite{MR1131586,MR1235440}, Serfati~\cite{MR1270072}, and Bertozzi-Constantin~\cite{MR1207667}; see also Radu~\cite{MR4509831}. Analysis of more singular vortex patch solutions can be found in the work of Danchin~\cite{MR1452164,MR1809342}, Elgindi-Jeong~\cite{MR4537304,MR4155190}, and Elgindi-Jo~\cite{ELGINDIJO}.

The steady rotating vortex patch solutions discussed in Section~\ref{subsection on opening} are evidently a special type of vortex patch solution whose dynamics are described simply by constant rotation. The only known explicit and nonradial rotating vortex patches are the Kirchhoff ellipses. There has long been an industry, starting with the numerical work of Deem and Zabusky~\cite{DEEM_ZABUSKY}, devoted to the construction of other steady rotating vortex patch solutions (also known as V-states). The first approach to obtaining such states as local bifurcations from the Rankine vortices was introduced by Burbea~\cite{MR646163}. The method was then revisited and placed on more solid technical ground by Hmidi, Mateu, and Verdera~\cite{MR3054601}. Local bifurcations near Kirchhoff ellipses were constructed by Castro, C\'ordoba, and G\'omez-Serrano~\cite{MR3462104} and Hmidi and Mateu~\cite{MR3543554}. Analytic regularity of the bifurcating patch's boundary was proved by Castro, C\'ordoba, and G\'omez-Serrano~\cite{MR3462104}. Global bifurcations from the Rankine vortices were first obtained by Hassainia, Masmoudi, and Wheeler~\cite{MR4156612}. We also mention that: Hmidi and Mateu~\cite{MR3607460} constructed disconnected steady rotating vortex patches; Hoz, Hmidi, and Mateu~\cite{MR3507551} found doubly connected planar vortex states; and Hoz, Hassainia, Hmidi, and Mateu~\cite{MR3570233} uncovered nontrivial rotating vortex patches confined to the disk. 

The literature related to the specific task of steady vortex patch desingularization for the Euler equations is much sparser, consisting of, to the best of the authors' knowledge, two specific results. Castro, C\'{o}rdoba, and G\'{o}mez-Serrano~\cite{MR3900813} construct via a delicate local bifurcation argument starting from Rankine vortices uniformly rotating Euler solutions with compact and class $C^2$ vorticity and without radial symmetry. Elgindi and Huang~\cite{MR4841732}, employing the Schauder fixed point theorem, desingularize the spatially periodic Bahouri-Chemin patch.

Our work in this paper, therefore, stands as the first general desingularization procedure for steady rotating vortex patches; moreover, we desingularize to the class of $C_{\m{c}}^\infty$-smooth states. We address \emph{any} nondegenerate V-state and prove that two well-known families of V-states -- Kirchhoff's ellipses and Burbea's Rankine perturbations -- are generically nondegenerate. 

Finally, let us remark that our results belong to the general structural study of the space of stationary or rigidly evolving solutions to the two-dimensional system of Euler equations. We list a few representative results of this subfield: Choffrut and \v{S}ver\'{a}k~\cite{MR2899685}; Choffrut and Sz\'{e}kelyhidi~\cite{MR3505175}; Constantin, Drivas, and Ginsberg~\cite{MR4275792}; and Coti Zelati, Elgindi, and Widmayer~\cite{MR4544663}.

\subsection{Main results and discussion}\label{subsection on main results and discussion}

As we alluded to in Section~\ref{subsection on opening}, our main theorems apply for a class of \emph{admissible} steady rotating vortex patches obeying a certain nondegeneracy condition. We now make these precise in the next definition. Recall that $\bf{N}$ from~\eqref{defn of the newtonian potential} denotes the Newtonian potential.

\begin{defnC}[Admissible steady rotating vortex patch solutions]\label{defn of admissible patches}
  We say that $\Psi\in C^{1,1}\tp{\R^2}$, $\Omega,c\in\R$, $\es\neq\tilde{U}\subset\R^2$, $m\in\N^+$ form an admissible tuple if the following are satisfied.
  \begin{enumerate}
    \item The set $\tilde{U}$ is open, bounded, simply connected, and has $m$-fold dihedral symmetry in the sense that
    \begin{equation}
        x=(x_1,x_2)\in\tilde{U}\text{ if and only if }(x_1,-x_2)\in\tilde{U}\text{ if and only if }\bf{R}\tp{2\pi/m}x\in\tilde{U},
    \end{equation}
    where $\bf{R}$ is defined in~\eqref{the rotation matrix}. The boundary $\Sigma = \pd\tilde{U}$ is class $C^{1,1}$.
    \item For all $x\in\Sigma$ we have $\Psi(x) = 0$ and if $\nu:\Sigma\to\R^2$ denotes the outward pointing unit normal, then $\min_{\Sigma}\tp{\nu\cdot\grad\Psi}>0$.
    \item For all $x\in\R^2$ it holds that
    \begin{equation}\label{the stream function form of the equation}
      \Psi(x) = \tp{\bf{N}\ast\mathds{1}_{\tilde{U}}}(x) - \Omega|x|^2/2 - c.
    \end{equation}
    \item If $\phi\in C^0\tp{\Sigma}$ has $m$-fold dihedral symmetry in the sense that
    \begin{equation}
        \phi(x) = \phi(x_1,-x_2) = \phi(\bf{R}\tp{2\pi/m}x)
    \end{equation}
    for all $x = (x_1,x_2)\in\Sigma$, and satisfies for all $x\in\Sigma$ the identity
    \begin{equation}\label{the nondegeneracy condition}
      \phi(x) + \int_{\Sigma}\bf{N}\tp{x - y}\f{\phi(y)}{\nu(y)\cdot\grad\Psi(y)}\;\m{d}\mathcal{H}^1\tp{y} = 0,
    \end{equation}
    then $\phi = 0$. Here $\m{d}\mathcal{H}^1$ denotes integration with respect to the one-dimensional Hausdorff measure.
  \end{enumerate}
\end{defnC}

The first three items of Definition~\ref{defn of admissible patches} guarantee in particular that the triple $\Omega$, $\Psi$, and $\omega = \mathds{1}_{\tilde{U}}$ are a steady rotating weak solution in the sense of Definition~\ref{defn od steady rotating weak solutions to the planar Euler system} with $R = \infty$. It is worth noting that the positive normal derivative assertion of the second item of Definition~\ref{defn of admissible patches} holds automatically in the case that $\Omega<1/2$, as a simple consequence of the Hopf boundary point lemma.

The fourth item is the nondegeneracy condition. Roughly speaking, this condition is saying that the linearization at the particular patch solution of~\eqref{the stream function form of the equation} has no kernel. As we shall prove in Theorem~\ref{thm_main_4} below, the collection of admissible states in the sense of Definition~\ref{defn of admissible patches} is far from being empty.

At last, we are ready to state our main theorems, the proofs of which are found in Section~\ref{subsection on proofs of main theorems}. Our first result gives a positive answer to the question of classical realization for rotating patches.

\begin{introthm}[Desingularization of admissible states]\label{thm_main_1}
    Let $\Psi$, $\Omega$, $c$, $\tilde{U}$, and $m$ be an admissible steady rotating vortex patch in the sense of Definition~\ref{defn of admissible patches} and let $\gam\in C^\infty\tp{\R}$ be any function that satisfies $\gam(t) = 0$ for $t\le -1$ and $\gam(t) = 1$ for $t\ge 1$. There exists $\ep_\star\in(0,1]$, $0<C<\infty$, and an open and bounded neighborhood $U\supset\Sigma$ with the property that for all $0<\ep\le\ep_\star$ there exists $\omega_\ep\in C^\infty_{\m{c}}\tp{\R^2}$ and $\Psi_\ep\in C^\infty\tp{\R^2}$ such that 
        \begin{equation}\label{THM_1_LIMS}
            \forall\;\al\in\tp{0,1},\;\tnorm{\grad\tp{\Psi_\ep - \Psi}}_{C^\al_{\m{b}}\tp{\R^2}}\to0\text{ as }\ep\to0
        \end{equation}
        and the following hold.
    \begin{enumerate}
        \item For all $x\in\R^2$
        \begin{equation}
            \Psi_\ep = \tp{\bf{N}\ast\omega_\ep}\tp{x} - \tp{\Omega/2}|x|^2 - c\text{ and }\grad^\perp\Psi_\ep\tp{x}\cdot\grad\omega_\ep\tp{x} = 0.
        \end{equation}
        In particular, the vorticity-relative stream function formulation of the steady rotating planar Euler system~\eqref{the steady rotating equations for vorticity and relative stream function} is classically satisfied with angular velocity $\Omega$, vorticity $\omega_\ep$, and relative stream function $\Psi_\ep$.
        \item We have $C\tnorm{\omega_\ep}_{C^1\tp{\R^2}}\ge1/\ep$ and the deviation estimates
        \begin{equation}\label{THM_I_deviation}
            \tnorm{\omega_\ep - \mathds{1}_{\tilde{U}}}_{L^1\tp{\R^2}} + \tabs{\tnorm{\grad\omega_\ep}_{\m{TV}\tp{\R^2}} - \tnorm{\grad\mathds{1}_{\tilde{U}}}_{\m{TV}\tp{\R^2}}} + \tnorm{\grad\tp{\Psi_\ep - \Psi}}_{L^\infty\tp{\R^2}}\le C\ep\log\tp{1/\ep},
        \end{equation}
        where $\tnorm{\cdot}_{\m{TV}\tp{\R^2}}$ denotes the total variation norm on the space of vector valued measures.
        \item We have
        \begin{equation}
            \forall\;x\in U,\;\omega_\ep\tp{x} = \gam\bp{-\f{\Psi_\ep\tp{x}}{\ep}}\text{ and }\forall\;x\in\R^2\setminus U,\;\omega\tp{x}\in\tcb{0,1}.
        \end{equation}
        \item The functions $\Psi_\ep$ and $\omega_\ep$ have $m$-fold dihedral symmetry.
    \end{enumerate}
\end{introthm}

The machinery we develop to prove Theorem~\ref{thm_main_1}, which is the content of Section~\ref{section on analysis near admissible states}, is quite robust and finds applications beyond just the task of admissible patch desingularization. The two subsequent main theorems highlight some of what is possible in this direction.

Our next result, also addressing admissible states, shows that any neighborhood of the rotating vortex patch contains a whole zoo of singular solutions that look like linear combinations of nested patches. Thus one can think of the given nondegenerate patch solution as able to be split into sums of multiple patches in many different ways.

\begin{introthm}[Splitting of admissible states]\label{thm_main_2}
   Let $\Psi$, $\Omega$, $c$, $\tilde{U}$, and $m$ be an admissible steady rotating vortex patch in the sense of Definition~\ref{defn of admissible patches}, $M\in\N^+$, and $\tcb{\sig_k}_{k=-M}^M\subset\R$ satisfy $\sum_{k=-M}^M\sig_k = 1$. There exists $\rho_\star\in(0,1]$, $0<C<\infty$, and an open and bounded neighborhood $U\supset\Sigma$ with the property that for all $0<\rho\le\rho_\star$ there exists $\omega_\rho\in\tp{L^\infty\cap\m{BV}}_{\m{c}}\tp{\R^2}$ (with the subscript `$\m{c}$' indicating compact support) and $\Psi_\rho\in\bigcap_{0<\al<1}C^{1,\al}\tp{\R^2}$ such that 
       \begin{equation}\label{THM_2_LIMS}
           \forall\;\al\in\tp{0,1},\;\tnorm{\grad\tp{\Psi_\rho - \Psi}}_{C^\al_{\m{b}}\tp{\R^2}}\to0\text{ as }\rho\to0
       \end{equation}
       and the following hold.
   \begin{enumerate}
       \item The steady rotating planar Euler system is satisfied, in the weak sense of Definition~\ref{defn od steady rotating weak solutions to the planar Euler system} with $R = \infty$, by the angular velocity $\Omega$, vorticity $\omega_\rho$, and relative stream function $\Psi_\rho$.
       \item There exist open, bounded domains of class $\bigcap_{0<\al<1}C^{1,\al}$, denoted $\tcb{W_\rho^k}_{k=-M}^M$, satisfying the following.
       \begin{enumerate}
           \item The monotone inclusion relations
           \begin{equation}\label{the monotone inclusion relations}
               W^{M}_\rho\subset\cdots\subset W^0_\rho\subset\cdots\subset W^{-M}_\rho
           \end{equation}
           hold.
           \item The vorticity is given by the linear combination
            \begin{equation}\label{the vorticity cake}
                \omega_\rho = \sum_{k=-M}^M\sig_k\mathds{1}_{W^k_\rho}
            \end{equation}
       almost everywhere in $\R^2$.
       \item The distance estimates 
       \begin{equation}\label{the distance estimates for the set boundaries}
           C\cdot \m{dist}\tp{\pd W^k_\rho,\pd W^{k+1}_\rho}\ge\rho,\;k\in\tcb{-M,\dots,M-1}
       \end{equation}
       hold.
       \end{enumerate}
       \item We have the deviation estimates
       \begin{equation}\label{estimates of the fourth item}
           \tnorm{\omega_\rho - \mathds{1}_{\tilde{U}}}_{L^1\tp{\R^2}} + \tabs{\tnorm{\grad\omega_\rho}_{\m{TV}\tp{\R^2}} - \tnorm{\grad\mathds{1}_{\tilde{U}}}_{\m{TV}\tp{\R^2}}} + \tnorm{\grad\tp{\Psi_\rho - \Psi}}_{L^\infty\tp{\R^2}}\le C\rho\log\tp{1/\rho},
       \end{equation}
       where, again, $\tnorm{\cdot}_{\m{TV}\tp{\R^2}}$ denotes the total variation norm.
       \item We have
       \begin{equation}\label{other vorticity cake}
           \omega_\rho = \sum_{k=-M}^M\sig_k\mathds{1}_{(0,\infty)}\tp{-\Psi_\rho - k\rho}\text{ a.e. in }U
       \end{equation}
       and, in particular, for all $k\in\tcb{-M,\dots,M}$ it holds that $\Psi_\rho + k\rho = 0$ on $\pd W^k_\rho$.
       \item The functions $\Psi_\rho$ and $\omega_\rho$ have $m$-fold dihedral symmetry.
   \end{enumerate}
\end{introthm}

Our final construction around admissible states addresses what we call trapping. The vortex patches in the scope of Definition~\ref{defn of admissible patches} are solutions to the Euler equations over the entire plane $\R^2$. One may wonder if the solutions can be confined to a circular disk region of a large but finite radius at the cost of a small perturbation in the patch and in the stream function. The following result shows that this is indeed true.

\begin{introthm}[Trapping of admissible states]\label{thm_main_3}
  Let $\Psi$, $\Omega$, $c$, $\tilde{U}$, and $m$ be an admissible steady rotating vortex patch in the sense of Definition~\ref{defn of admissible patches}. There exists $1\le R_\star<\infty$, $0<C,C_1<\infty$ with $C_1\le R_{\star}$, and an open and bounded neighborhood $U\supset\Sigma$ with the property that for all $R_\star\le R<\infty$ there exists $\omega_R\in\tp{L^\infty\cap\m{BV}}_{\m{c}}\tp{B(0,R/2)}$ (again with the subscript `$\m{c}$' indicating compact support) and $\Psi_R\in\bigcap_{0<\al<1}C^{1,\al}\tp{\Bar{B(0,R)}}$ such that
  \begin{equation}\label{THM_3_LIMS}
      \forall\tilde{R}\in\R^+,\;\forall\;\al\in\tp{0,1},\;\tnorm{\grad\tp{\Psi_R - \Psi}}_{C^\al_{\m{b}}\tp{\Bar{B(0,\min\tcb{R,\tilde{R}})}}}\to0\text{ as }R\to\infty
  \end{equation}
  and the following hold.
  \begin{enumerate}
      \item The steady rotating planar Euler system is satisfied, in the weak sense of Definition~\ref{defn od steady rotating weak solutions to the planar Euler system}, in the domain $B(0,R)$ with angular velocity $\Omega$, vorticity $\omega_{R}$, and relative stream function $\Psi_{R}$.
      \item There exists an open, bounded, and class $\bigcap_{0<\al<1}C^{1,\al}$ domain $W_R\subset B(0,C_1)$ such that $\omega_{R} = \mathds{1}_{W_R}$ and 
      \begin{equation}
          \pd W_R = \tcb{x\in U\;:\;\Psi_R(x) = 0}.
      \end{equation}
      \item We have the deviation estimates
      \begin{equation}\label{the convergence rate}
          \tnorm{\omega_R - \mathds{1}_{\tilde{U}}}_{L^1\tp{B(0,R)}} + \tabs{\tnorm{\grad\omega_R}_{\m{TV}\tp{B(0,R)}} - \tnorm{\grad\mathds{1}_{\tilde{U}}}_{\m{TV}\tp{B(0,R)}}} + \tnorm{\grad\tp{\Psi_R - \Psi}}_{L^\infty\tp{B(0,R)}}\le C\log\tp{R}/R^2,
      \end{equation}
      where $\tnorm{\cdot}_{\m{TV}\tp{B(0,R)}}$ denotes the total variation norm.
      \item The functions $\Psi_R$ and $\omega_R$ have $m$-fold dihedral symmetry.
  \end{enumerate}
\end{introthm}

The last main theorem of the paper, whose proof is the content of Section~\ref{section on existence of admissible states}, gives two examples of families of rotating vortex patches satisfying the conditions of Definition~\ref{defn of admissible patches}.

\begin{introthm}[Existence of admissible states]\label{thm_main_4}
Definition~\ref{defn of admissible patches} is nonvacuous. In fact:
\begin{enumerate}
    \item There exists a countable set $\mathcal{N}\subset\R^+$ such that for all $\xi\in\R^+\setminus\mathcal{N}$ the $\xi$-Kirchhoff elliptical vortex, defined in~\eqref{the kirchhoff ellispses as sets} and~\eqref{KEAF}, is admissible in the sense of Definition~\ref{defn of admissible patches} with $\Psi = \Psi^\xi$, $\Omega = \Omega^\xi$, $c = c^\xi$, $\tilde{U} = E^\xi$, and $m=2$.
    \item See Theorem~\ref{thm on admissibility of rankine vortex perturbations} for a precise statement; informally: for every $\N\ni m\ge 2$, the $m$-fold symmetric states locally bifurcating from the circular Rankine vortex with angular velocity $\Omega_m= (m-1)/2m$ (in other words Burbea's $m$-fold symmetric V-states) are admissible in the sense of Definition~\ref{defn of admissible patches}.
\end{enumerate}
\end{introthm}

The remainder of this subsection is devoted to a brief and high-level discussion of the proof of our main results. Let $\Psi$, $\Omega$, $c$, $\tilde{U}$, and $m$ be an admissible steady rotating vortex patch solution in the sense of Definition~\ref{defn of admissible patches}. The starting point is the equation~\eqref{the stream function form of the equation} that relates the stream function $\Psi$ to the interior vorticity support $\tilde{U}$. For the purposes of this discussion, let us suppose that $\Psi<0$ in $\tilde{U}$. This condition is not built into Definition~\ref{defn of admissible patches}, but is satisfied thanks to a maximum principle argument, if we know additionally that $\Omega<1/2$. Let us also pretend that $\Psi>0$ in all of $\R^2\setminus\Bar{\tilde{U}}$ -- this is not really true, but it is convenient for the ensuing imprecise discussion. These assumptions allow us to imagine that
\begin{equation}
    \mathds{1}_{\tilde{U}} = \mathds{1}_{\tp{0,\infty}}\tp{-\Psi}.
\end{equation}

Then, we obtain a closed equation in terms of the stream function with a Heaviside \emph{vorticity function}; we write this as an abstract operator equation $\mathsf{F}[\Psi] = 0$ with
\begin{equation}\label{_one_operator_}
    \tp{\mathsf{F}\tsb{\Psi}}\tp{x} = \Psi(x) - \f{1}{2\pi}\int_{\R^2}\log|x - y|\mathds{1}_{\tp{0,\infty}}\tp{-\Psi(y)}\;\m{d}y + \Omega|x|^2/2 + c = 0\text{ for all } x\in\R^2.
\end{equation}
The goal then is to construct nearby Euler solutions $\tilde{\Psi}$ that solve a perturbed operator equation of the form $\mathsf{F}_{\mathsf{v}}[\tilde{\Psi}] = 0$ with
\begin{equation}\label{the peturbed equation to be solved}
    \tp{\mathsf{F}_{\mathsf{v}}[\tilde{\Psi}]}\tp{x} = \tilde{\Psi}\tp{x} - \f{1}{2\pi}\int_{\R^2}\log\tabs{x - y}\mathsf{v}\tp{-\tilde{\Psi}(y)}\;\m{d}y + \Omega|x|^2/2 + c \text{ for all }x\in\R^2.
\end{equation}
where $\mathsf{v}:\R\to\R$ is a vorticity function that is in an appropriate sense close to the Heaviside function $\mathds{1}_{\tp{0,\infty}}$. For example, in Theorem~\ref{thm_main_1} we take (with $\ep>0$ small)
\begin{equation}\label{smooth case}
    \mathsf{v}\tp{t} = \gam\tp{t/\ep}\text{ with }\gam\in C^\infty\tp{\R},\;\gam(t) = 0\text{ for }t\le-1\text{ and }\gam(t) = 1\text{ for }t \ge 1;
\end{equation}
on the other hand, for Theorem~\ref{thm_main_2}, we take (with $\rho>0$ small)
\begin{equation}\label{singular case}
    \mathsf{v}\tp{t} = \sum_{k=-M}^M\sig_k\mathds{1}_{\tp{0,\infty}}\tp{t - k\rho}\text{ with }\sum_{k=-M}^M\sig_k = 1.
\end{equation}
Note that taking $\sig_0 = 1$ and $\sig_k = 0$ for $k\neq 0$ in~\eqref{singular case} recovers the Heaviside vorticity function.

One might hope to address these perturbed equations~\eqref{the peturbed equation to be solved} through some sort of implicit function theorem (IFT) or Banach fixed point argument. Indeed, it is this hope that is the inspiration for the instantiation of our nondegeneracy condition given in the fourth item of Definition~\ref{defn of admissible patches}. By formally taking the Fr\'{e}chet derivative of the expression for $\mathsf{F}$ given in equation~\eqref{_one_operator_} at the admissible stream function $\Psi$ in an arbitrary direction $\phi$ and restricting the result to $\Sigma$, we get exactly the left hand side of~\eqref{the nondegeneracy condition}. Thus we see precisely that the nondegeneracy condition is equivalent to the invertibility of this formal functional derivative.

However, substantial technical obstructions arise that preclude a direct invocation of the IFT or contraction mapping principle. More specifically, the operator $\mathsf{F}$~\eqref{_one_operator_} appears not to be $C^1$ in the Fr\'{e}chet sense or even locally Lipschitz for any choice of suitable function spaces. The reason for this failure is precisely the discontinuous vorticity function $\mathsf{v} = \mathds{1}_{(0,\infty)}$. 

As a first step in solving the perturbed equations $\mathsf{F}_{\mathsf{v}}[\tilde{\Psi}] =  0$, we can replace the Heaviside vorticity function of~\eqref{_one_operator_} with a nearby one that is qualitatively smooth. Executing this swap restores Fr\'{e}chet smoothness and a local Lipschitz condition, at the expense of introducing a source term. We define the following regularized versions of~\eqref{singular case}:
\begin{equation}
    \mathsf{v}_{\ep,\rho}(t) = \sum_{k=-M}^M\sig_k\gam\bp{\f{t - k\rho}{\ep}}.
\end{equation}
The singular case of the vorticity function is ultimately recovered by taking the limit as $\ep\to0$ in the above -- of course, this requires carefully tracking various quantities.

We let $X = C^1_{\m{b}}(\R^2)$ be an ambient Banach space and set up the operator
\begin{equation}\label{e o fim}
    \mathsf{G}_{\ep,\rho}:B_X(0,r)\to X\text{ with action }\mathsf{G}_{\ep,\rho}[\phi] = \mathsf{F}_{\mathsf{v}_{\ep,\rho}}[\Psi + \phi] - \mathsf{F}_{\mathsf{v}_{\ep,\rho}}[\Psi]\text{ for }\phi\in B_X(0,r).
\end{equation}
We endeavor to solve, at least for sufficiently small $\ep$ and $\rho$, the equation, that is equivalent to $\mathsf{F}_{\mathsf{v}}[\tilde{\Psi}] = 0$,
\begin{equation}\label{peturbed operator equation}
    \mathsf{G}_{\ep,\rho}[\psi] = \mathsf{g}_{\ep,\rho},\text{ for the source term }\mathsf{g}_{\ep,\rho} = \mathsf{F}[\Psi] - \mathsf{F}_{\mathsf{v}_{\ep,\rho}}[\Psi]\in X
\end{equation}
with $\psi = \tilde{\Psi} - \Psi$ denoting the relative stream function perturbation. Note crucially that in the formulation above, the operator $\mathsf{G}_{\ep,\rho}$ is smooth; however, the norms of its derivatives are blowing up in the limit $\ep\to0$.

The following two facts are sufficient to guarantee the existence of solutions to the perturbed operator equation~\eqref{peturbed operator equation}.
\begin{enumerate}
    \item We have the limit: $\lim_{\ep,\rho\to0}\tnorm{\mathsf{g}_{\ep,\rho}}_{X} = 0$. This essentially follows from the coarea formula.
    \item Uniform local surjectivity. More precisely: there exists $r_1,\ep_1,\rho_1>0$ such that for all $0<\ep\le\ep_1$ and $0\le\rho<\rho_1$ we have $B_X(0,r_1)\subset\mathsf{G}_{\ep,\rho}\tsb{B_X(0,r)}$.
\end{enumerate}

The proof of the latter fact is a delicate matter due to the development of singularities in the limit as $\ep\to0$ and its competition with the requirement that the radius $r_1>0$ be taken \emph{independent} of small $\ep,\rho>0$. To wit: the nondegeneracy condition of Definition~\ref{defn of admissible patches} does ensure that $D\mathsf{G}_{\ep,\rho}[0]:X\to X$ is invertible and so the IFT indeed gives local invertibility of $\mathsf{G}_{\ep,\rho}$ near zero; the problem, however, is that the operator norms satisfy $\min\tcb{\tnorm{DG_{\ep,\rho}[0]}_{\mathcal{L}\tp{X}},\tnorm{DG_{\ep,\rho}[0]^{-1}}_{\mathcal{L}\tp{X}}}\to\infty$ as $\ep\to0$. Thus, the domains granted by the IFT are shrinking to nothing and we do not yet satisfy the second fact above.

Our construction of the local inverse that does satisfy the uniform local surjectivity condition instead proceeds via a version of Newton's method capable of handling this singular limit. We formulate this tool abstractly in Theorem~\ref{thm on abstract quant loc surj} below, but the main hypotheses that guarantee a local inverse are:
\begin{enumerate}
    \item A nonlinear a priori estimate. There exists $c_1,c_2$ with $c_2<r$ ($r$ being the domain radius from~\eqref{e o fim}) such that (uniformly for small $\ep$ and $\rho$) $\tnorm{\mathsf{G}_{\ep,\rho}[\phi]}_{X}\le c_1$ implies that $\tnorm{\phi}_X\le c_2$.
    \item Qualitative invertibility of the derivative in a uniform neighborhood. There exists $c_2<c_3\le r$ (again uniform for small $\ep$ and $\rho$) such that for all $\tnorm{\phi}_X\le c_3$ we have that $D\mathsf{G}_{\ep,\rho}[\phi]:X\to X$ is an invertible map.
\end{enumerate}
The key feature is that the second item allows for the norm of the derivative to become unboundedly large in the singular limit of small parameters. What is important is that the size of the neighborhood in which the derivative's inverse exists is not collapsing. The role of the nonlinear a priori estimate is to ensure that the iterates in Newton's method do not leave a bounded region, thus facilitating the convergence. It is a pleasant fact that the two hypotheses of our uniform local surjectivity theorem when applied to the maps $\mathsf{G}_{\ep,\rho}$~\eqref{e o fim} are satisfied as consequences of the nondegeneracy condition~\eqref{the nondegeneracy condition} from Definition~\ref{defn of admissible patches}.

After we have constructed the uniform local inverses of the maps $\mathsf{G}_{\ep,\rho}$, we, roughly speaking, obtain Theorem~\ref{thm_main_1} by setting $\rho = 0$. On the other hand, Theorem~\ref{thm_main_2} is obtained by taking the limit $\ep\to0$ for each fixed $\rho>0$. Theorem~\ref{thm_main_3} is ultimately proved by a strategy similar to the one described above, but we introduce another type of perturbation to the base operator equation~\eqref{_one_operator_} by varying the Green's function.

The next subsection is a notation interlude. The plan for the remainder of the paper is as follows. Section~\ref{section on existence of admissible states} is devoted to the proof of Theorem~\ref{thm_main_4}, with Sections~\ref{subsection on Kirchhoff elliptical vortices} and~\ref{subsection on perturbations of the rankine vortex} handling the first and second items of this main theorem, respectively. Next, in Section~\ref{section on analysis near admissible states}, we essentially prove Theorems~\ref{thm_main_1}, \ref{thm_main_2}, and \ref{thm_main_3} simultaneously. Section~\ref{subsection on preliminary estimates and identities} witnesses the derivation of a flurry of preliminary tools and estimates. Next, in Section~\ref{subsection on key consequences of nondegeneracy}, we establish the most salient consequences of the nondegeneracy condition of the fourth item of Definition~\ref{defn of admissible patches}. After that, in Section~\ref{subsection on quantitative local invertibility} we construct our abstract quantitative local surjectivity theorem based on Newton's method and then apply it to problem at hand. We conclude Section~\ref{section on analysis near admissible states} with the proof of our main theorems in Section~\ref{subsection on proofs of main theorems}. Finally, Appendix~\ref{appendix on tools from analysis} records some useful analysis supplements.

\subsection{Conventions of notation}\label{section on notational conventions}

The set of natural numbers is denoted by $\N = \tcb{0,1,2,3,\dots}$ and we also denote $\N^+ = \N\setminus\tcb{0}$. The set of positive real numbers is denoted by $\R^+ = \tp{0,\infty}$. The characteristic function of a set $E_0\subset E_1$ is $\mathds{1}_{E_0}:E_1\to\tcb{0,1}$. If $E$ is a subset of any topological space, the closure and interior are denoted $\Bar{E}$ and $\m{int}E$, respectively. If $A$ and $B$ are subsets of a topological space then $A\Subset B$ means that $\Bar{A}\subset\m{int}B$. Let $m\in\N^+$; we say that $E\subset\R^2$ has $m$-fold dihedral symmetry if 
\begin{equation}
    x = (x_1,x_2)\in E\text{ if and only if }(x_1,-x_2)\in E\text{ if and only if }\bf{R}\tp{2\pi/m}x\in E
\end{equation}
where $\bf{R}$ is from~\eqref{the rotation matrix}.

If $\es\neq U\subset\R^2$ is open, $k\in\N$, and $\al\in[0,1]$ we let $C^{k,\al}\tp{U}$ denote the vector space of $k$-times differentiable functions that, together with their first $k$ derivatives, are locally $\al$-H\"{o}lder continuous in $U$. If $U$ is bounded, we shall also consider the Banach subspace of functions which are $C^{k,\al}$ up to the boundary:
\begin{equation}
    C^{k,\al}\tp{\Bar{U}} = \tcb{f\in C^{k,\al}\tp{U}\;:\;\tnorm{f}_{C^{k,\al}\tp{\Bar{U}}}<\infty}
\end{equation}
where
\begin{equation}
    \tnorm{f}_{C^{k,\al}\tp{\Bar{U}}} = \sum_{j=0}^k\bp{\sup_{x\in U}\tabs{\grad^j f(x)} + \sup_{x\neq\tilde{x}}\f{\tabs{\grad^j f(x) -\grad^jf(\tilde{x})}}{|x - \tilde{x}|^{\al}}}
\end{equation}
When $\al = 0$, we write $C^k\tp{U}$ and $C^{k}\tp{\Bar{U}}$ in place of $C^{k,0}\tp{U}$ and $C^{k,0}\tp{\Bar{U}}$, respectively. 

For a bounded open set $U\subset\R^2$ the collection of logarithmic-Lipschitz functions is the Banach space
\begin{equation}\label{the space LL}
    \m{LL}\tp{\Bar{U}} = \tcb{f\in C^0\tp{\Bar{U}}\;:\;\tnorm{f}_{\m{LL}\tp{\Bar{U}}}<\infty},\;\tnorm{f}_{\m{LL}\tp{\Bar{U}}} = \tnorm{f}_{C^0\tp{\Bar{U}}} + \sup_{\substack{x,y\in U\\x\neq y}}\f{\tabs{f(x) - f(y)}}{|x - y|\tp{1 + \tabs{\log\tabs{x - y}}}}.
\end{equation}
We shall also define the related space
\begin{equation}\label{the space LL1}
    \m{LL}^1\tp{\Bar{U}} = \tcb{f\in C^1\tp{\Bar{U}}\;:\;\grad f\in\m{LL}\tp{\Bar{U};\R^2}}\text{ with the norm }\tnorm{f}_{\m{LL}^1\tp{\Bar{U}}} = \tnorm{f}_{C^1\tp{\Bar{U}}} + \tnorm{\grad f}_{\m{LL}\tp{\Bar{U}}}.
\end{equation}

If $U\subset\R^2$ has $m$-fold dihedral symmetry and $X(U)$ is a Banach space of functions defined on $U$, we let $X_{\tp{m}}\tp{U}$ denote the closed subspace of functions having $m$-fold dihedral symmetry, precisely:
\begin{equation}
    X_{\tp{m}}\tp{U} = \tcb{f\in X(U)\;:\;\forall\;x = (x_1,x_2)\in U,\;f(x) = f(x_1,-x_2) = f\tp{\bf{R}\tp{2\pi/m}x}}.
\end{equation}

\section{Existence of admissible states}\label{section on existence of admissible states}

In this subsection we establish the existence of admissible steady rotating vortex patches that satisfy Definition~\ref{defn od steady rotating weak solutions to the planar Euler system}. The lion's share of the work is, perhaps unsurprisingly, devoted to checking that the fourth item of the aforementioned definition -- the nondegeneracy condition -- holds. In Section~\ref{subsection on Kirchhoff elliptical vortices} we consider the Kirchhoff elliptical vortices and most of the verification is explicit computation. In contrast, the content of Section~\ref{subsection on perturbations of the rankine vortex}, which addresses $m$-fold symmetric perturbations of the circular vortex, is a more involved indirect argument exploiting subtle connections between various equivalent formulations of the steady rotating patch equations.

\subsection{Kirchhoff elliptical vortices}\label{subsection on Kirchhoff elliptical vortices}

Recall the one parameter families of ellipses $\tcb{E^\xi}_{\xi\in\R^+}$ and angular velocities $\tcb{\Omega^\xi}_{\xi\in\R^+}$ defined in~\eqref{the kirchhoff ellispses as sets}. The purpose of the following result is to be more precise about how these continua form the family of Kirchhoff vortex patch solutions and to set up notation for subsequent results.

\begin{lemC}[Stream functions for Kirchhoff vortices]\label{lem on stream functions for Kirchhoff vortices}
    For each $\xi\in\R^+$ there exists $c^\xi\in\R$ such that the function $\Psi^\xi\in C^{1,1}_{\tp{2}}\tp{\R^2}$ defined via
    \begin{equation}\label{stream function definition}
        \Psi^\xi\tp{x} = \tp{\bf{N}\ast\mathds{1}_{E^\xi}}\tp{x} - \Omega^\xi|x|^2/2 - c^\xi\text{ for }x\in\R^2
    \end{equation}
    obeys the succeeding items.
    \begin{enumerate}
        \item For all $x\in\pd E^\xi$ we have $\Psi^\xi\tp{x} = 0$.
        \item If $\tilde{\Psi}^\xi:\R^+\times\R\to\R$, defined via
        \begin{equation}\label{a stone}
            \tilde{\Psi}^{\xi}\tp{\zeta,\eta} = \Psi^\xi\tp{\sech\tp{\xi}\cosh(\zeta)\cos(\eta),\sech\tp{\xi}\sinh(\zeta)\sin(\eta)}\text{ for all }\tp{\zeta,\eta}\in\R^+\times\R,
        \end{equation}
        denotes $\Psi^\xi$ in elliptical coordinates, then for all $\zeta\in[\xi,\infty)$ and $\eta\in\R$ we have the formula
        \begin{equation}\label{a stick}
            \tilde{\Psi}^\xi\tp{\zeta,\eta} = \f{\tanh(\xi)}{4}\bp{2\tp{\zeta - \xi} - \f{\cosh(2\zeta) - \cosh(2\xi)}{e^{2\xi}}} + \f{\tanh\tp{\xi}}{4}\bp{\f{1}{e^{2\zeta}} - \f{1}{e^{2\xi}}}\cos(2\eta).
        \end{equation}
    \item For all $\eta\in\R$ we have 
    \begin{equation}\label{normal derivative in elliptical coordinates}
        \pd_1\tilde{\Psi}^\xi\tp{\xi,\eta} = \f{\tanh(\xi)}{2e^{2\xi}}\tp{\cosh\tp{2\xi} - \cos\tp{2\eta}}.
    \end{equation}
    and hence $\min_{\pd E^\xi}\tabs{\grad\Psi^{\xi}}>0$.
    \end{enumerate}
\end{lemC}
\begin{proof}
    First, let $\tilde{\Psi}_{\m{rot}}:\R^+\times\R\to\R$ denote the elliptical coordinate description of the stream function associated with uniform rotation by the angular velocity $\Omega^\xi$; more precisely
    \begin{equation}
        \tilde{\Psi}^\xi_{\m{rot}}\tp{\zeta,\eta} = -\f{\Omega^\xi}{2}\tp{\tp{\sech\tp{\xi}\cosh\tp{\zeta}\cos\tp{\eta}}^2 + \tp{\sech\tp{\xi}\sinh\tp{\zeta}\sin\tp{\eta}}^2}
        =-\f{\tanh\tp{\xi}}{4e^{2\xi}}\tp{\cosh\tp{2\zeta} + \cos\tp{2\eta}}
    \end{equation}
    for all $\zeta\in\R^+$ and $\eta\in\R$.

    Next, by repeating the computations in Article 159 of Chapter VII of Lamb~\cite{MR1317348} we obtain a particular $\tilde{\Psi}^\xi_{\m{ell}}:\R^+\times\R\to\R$ that satisfies for $\xi\le\zeta<\infty$ and $\eta\in\R$ the identity
    \begin{equation}
        \tilde{\Psi}_{\m{ell}}^\xi\tp{\zeta,\eta} = \f{\tanh\tp{\xi}}{4}\bp{2\zeta + \f{\cos(2\eta)}{e^{2\zeta}}}.
    \end{equation}
    Now we define $\tilde{\Psi}^\xi:\R^+\times\R\to\R$ via
    \begin{equation}\label{aguas de marco}
        \tilde{\Psi}^\xi\tp{\zeta,\eta} = \tp{\tilde{\Psi}^\xi_{\m{rot}} + \tilde{\Psi}^\xi_{\m{ell}}}\tp{\zeta,\eta} - \tp{\tilde{\Psi}^\xi_{\m{rot}} + \tilde{\Psi}^\xi_{\m{ell}}}\tp{\xi,\eta}.
    \end{equation}

    After routine computation, it is clear that~\eqref{aguas de marco} agrees with the right hand side of~\eqref{a stick} on the set $\zeta\ge\xi$, $\eta\in\R$. The aforementioned construction in Lamb~\cite{MR1317348} ensures that there exists $\Psi^\xi\in C^{1,1}_{\tp{2}}\tp{\R^2}$ uniquely determined by the property that
    \begin{equation}
        \Psi^\xi\tp{\sech(\xi)\cosh(\zeta)\cos(\eta),\sech(\xi)\sinh(\zeta)\sin(\eta)} = \tilde{\Psi}^\xi\tp{\zeta,\eta}\text{ for all }\zeta\in\R^+,\;\eta\in[0,2\pi).
    \end{equation}

    Let us now see that $\Psi^\xi$ is the desired relative stream function that satisfies~\eqref{stream function definition}. The previous construction following Lamb~\cite{MR1317348} guarantees that
    \begin{equation}
        \Delta\Psi^\xi = \mathds{1}_{E^\xi} - 2\Omega^\xi\text{ in }\R^2,\;\Psi^\xi = 0\text{ on }\pd E^\xi,\text{ and }\lim_{|x|\to\infty}\tp{\grad\Psi^\xi\tp{x} + \Omega^\xi x} = 0.
    \end{equation}
    Therefore, Liouville's theorem for harmonic functions (see, e.g., Theorem 2.16 in Folland~\cite{MR1357411}) assures us that
    \begin{equation}
        \grad\tp{\Psi^\xi + \Omega^\xi|\cdot|^2/2 - \bf{N}\ast\mathds{1}_{E^\xi}} = 0\text{ in }\R^2.
    \end{equation}
    The existence of the number $c^\xi$ seen in~\eqref{stream function definition} is now immediate.
\end{proof}

Armed with sufficiently explicit formulae for the Kirchhoff vortices' relative stream functions, we are in a position to unpack the nondegeneracy condition of the fourth item of Definition~\ref{defn of admissible patches}.

\begin{propC}[Nondegeneracy on Kirchhoff stream functions, I]\label{kirchhoff nondegeneracy 1}
    Fix $\xi\in\R^+$ and let $\Sigma^\xi = \pd E^\xi$. The following are equivalent for $\phi\in C^0_{\tp{2}}\tp{\Sigma^\xi}$.
    \begin{enumerate}
        \item For all $x\in\Sigma^\xi$ we have
        \begin{equation}\label{the spectral condition of Kirchhoff}
            \phi(x) + \int_{\Sigma^\xi}\bf{N}\tp{x - y}\f{\phi(y)}{\tabs{\grad\Psi^\xi\tp{y}}}\;\m{d}\mathcal{H}^1\tp{y} = 0
        \end{equation}
        where $\m{d}\mathcal{H}^1$ denotes integration with respect to the one-dimensional Hausdorff measure.
        \item If we let $\rho\in C^0\tp{2\pi\T}$ be defined via $\rho(\eta) = \phi\tp{\cos(\eta),\tanh\tp{\xi}\sin\tp{\eta}}$ for $\eta\in 2\pi\T$ then
        \begin{equation}\label{zero vanish cond}
            \tp{\sinh\tp{2\xi} - 2e^{2\xi}\log\tp{1 + e^{-2\xi}}}\mathscr{F}[\rho](0) =  0
        \end{equation}
        and for $k\in\Z\setminus\tcb{0}$
        \begin{equation}\label{nonzero vanish cond}
            \tp{|k|\sinh(2\xi) - e^{2\xi}\tp{1+e^{-2\xi|k|}}}\mathscr{F}[\rho]\tp{k} = 0.
        \end{equation}
        Here $\mathscr{F}[\rho]$ denotes the Fourier transform of $\rho$.
    \end{enumerate}
\end{propC}
\begin{proof}
    The first stage of the proof is to write the right hand integral operator of $\phi$ in~\eqref{the spectral condition of Kirchhoff} in terms of $\rho$. Let $\phi\in C^0_{\tp{2}}\tp{\Sigma^\xi}$ and $\upphi\in C^0\tp{\R^2}$ be such that $\phi = \upphi$ on $\Sigma^\xi$. Such a $\upphi$ exists according to the Tietze extension theorem; see, e.g. Theorem 4.16 in Folland~\cite{MR1681462}. By using the coarea formula (see, e.g., Theorem 3.2.11 in~\cite{MR257325} or Theorem 3.13 in~\cite{MR3409135}), or arguing as in the proof of Proposition~\ref{prop on limit identification} below, one sees that for every $x\in\Sigma^\xi$
    \begin{equation}\label{going up a dimension}
        \int_{\Sigma^\xi}\bf{N}\tp{x - y}\f{\phi(y)}{\tabs{\grad\Psi^\xi\tp{y}}}\;\m{d}\mathcal{H}^1\tp{y} = \lim_{\ep\to0}\f{1}{\ep}\int_{U\cap\tcb{|\Psi^\xi|<\ep}}\bf{N}\tp{x - y}\upphi(y)\;\m{d}y
    \end{equation}
    where $U\supset\Sigma^\xi$ is a bounded open set with the property that $\tcb{x\in U\;:\;\Psi^\xi\tp{x} = 0} = \Sigma^\xi$ and $\min_{\Bar{U}}|\grad\Psi^\xi|>0$ (see the third item of Lemma~\ref{lem on stream functions for Kirchhoff vortices}).

    Now for $\ep\in\R^+$ sufficiently small, we change variables in the right hand side integral~\eqref{going up a dimension} according to the elliptical coordinates map
    \begin{equation}
        \bf{e}:\R^+\times2\pi\T\to\R^2\text{ defined via }\bf{e}\tp{\zeta,\eta} = \tp{\sech\tp{\xi}\cosh\tp{\zeta}\cos\tp{\eta},\sech\tp{\xi}\sinh\tp{\zeta}\sin\tp{\eta}}
    \end{equation}
    which has Jacobian
    \begin{equation}\label{the jacobian identity}
        \bf{J}:\R^+\times 2\pi\T\to\R^+\text{ defined via }\bf{J}\tp{\zeta,\eta} = \sech^2\tp{\xi}\tp{\cosh\tp{2\zeta} - \cos\tp{2\eta}}/2.
    \end{equation}
    Thus, upon writing $x = \bf{e}\tp{\xi,\tilde{\eta}}$ for some $\tilde{\eta}\in2\pi\T$ we have
    \begin{multline}\label{the CoV identity}
        \int_{U\cap\tcb{|\Psi^\xi|<\ep}}\bf{N}\tp{x - y}\upphi(y)\;\m{d}y \\= \f{1}{\ep}\int_0^\infty\int_0^{2\pi}\bf{N}\tp{\bf{e}\tp{\xi,\tilde{\eta}} - \bf{e}\tp{\zeta,\eta}}\tp{\upphi\circ\bf{e}}\tp{\zeta,\eta}\tp{\mathds{1}_{\bf{e}\in U\cap\tcb{|\tilde{\Psi}^\xi|<\ep}}\bf{J}}\tp{\zeta,\eta}\;\m{d}\eta\;\m{d}\zeta,
    \end{multline}
    with $\tilde{\Psi}^\xi$ as defined in~\eqref{a stone}. Now we pass to the limit as $\ep\to0$ on the right hand side of~\eqref{the CoV identity} while using an argument similar to the justification of~\eqref{going up a dimension} to obtain
    \begin{equation}\label{galaxy tablet}
        \int_{\Sigma^\xi}\bf{N}\tp{x - y}\f{\phi(y)}{\tabs{\grad\Psi^\xi\tp{y}}}\;\m{d}\mathcal{H}^1\tp{y} = \int_0^{2\pi}\bf{N}\tp{\bf{e}\tp{\xi,\tilde{\eta}} - \bf{e}\tp{\xi,\eta}}\rho(\eta)\f{\bf{J}\tp{\xi,\eta}}{\pd_1\tilde{\Psi}^\xi\tp{\xi,\eta}}\;\m{d}\eta
    \end{equation}
    with $\rho = \phi\circ\bf{e}\tp{\xi,\cdot}$. Note that in deriving~\eqref{galaxy tablet}, we have tacitly used the identity $|\grad\tilde{\Psi}^\xi\tp{\xi,\eta}| = \pd_1\tilde{\Psi}^\xi\tp{\xi,\eta}$ on $\tcb{\tilde{\Psi}^\xi = 0}\cap U$, which is a consequence of Lemma~\ref{lem on stream functions for Kirchhoff vortices}. Now we have the following key equality, that follows from~\eqref{normal derivative in elliptical coordinates} and~\eqref{the jacobian identity}:
    \begin{equation}\label{the crucial identity}
        \f{\bf{J}\tp{\xi,\eta}}{\pd_1\tilde{\Psi}^\xi\tp{\xi,\eta}} = \f{e^{2\xi}\sech^2\tp{\xi}}{\tanh\tp{\xi}} = \f{2 e^{2\xi}}{\sinh\tp{2\xi}}.
    \end{equation}
    Notice that crucially the dependence on $\eta$ drops out of the right hand side.

    Before substituting~\eqref{the crucial identity} into~\eqref{galaxy tablet}, we turn our attention to simplification of the kernel expression involving $\bf{N}$. One computes that
    \begin{equation}
        \tabs{\bf{e}(\xi,\tilde{\eta}) - \bf{e}\tp{\xi,\eta}}^2 = 4\sech^2\tp{\xi}\sin^2(\tp{\eta - \tilde{\eta}}/2)\tp{\sinh^2\tp{\xi} + \sin^2\tp{\tp{\eta + \tilde{\eta}}/2}}
    \end{equation}
    and hence for the kernels
    \begin{equation}
        K_1(\upeta) = \log\sin^2\tp{\upeta/2}/4\pi\text{ and }K_2\tp{\upeta} = \log\tp{\sinh^2\tp{\xi} + \sin^2\tp{\upeta/2}}/4\pi
    \end{equation}
    we have 
    \begin{equation}\label{the kernel decomposition}
        \bf{N}\tp{\bf{e}\tp{\xi,\tilde{\eta}} - \bf{e}\tp{\xi,\eta}} = \log\tp{2\sech(\xi)}/2\pi + K_1\tp{\eta - \tilde{\eta}} + K_2\tp{\eta + \tilde{\eta}}.
    \end{equation}
    Synthesizing~\eqref{galaxy tablet}, \eqref{the crucial identity}, and~\eqref{the kernel decomposition} gives us our sought after relationship:
    \begin{equation}\label{sought after relationship}
        \int_{\Sigma^\xi}\bf{N}\tp{x - y}\f{\phi(y)}{\tabs{\grad\Psi^\xi\tp{y}}}\;\m{d}\mathcal{H}^1\tp{y} = \f{2e^{2\xi}}{\sinh\tp{2\xi}}\int_{0}^{2\pi}\bp{\f{1}{2\pi}\log\tp{2\sech\tp{\xi}} + K_1\tp{\eta - \tilde{\eta}} + K_2\tp{\eta + \tilde{\eta}}}\rho\tp{\eta}\;\m{d}\eta.
    \end{equation}
    The second stage of the proof is to analyze the action in Fourier space of the following integral operators:
    \begin{equation}
        \tp{O_1\uprho}\tp{\tilde{\eta}} = \int_{0}^{2\pi}K_1(\eta - \tilde{\eta})\uprho(\eta)\;\m{d}\eta\text{ and }\tp{O_2\uprho}\tp{\tilde{\eta}} = \int_{0}^{2\pi}K_2\tp{\eta + \tilde{\eta}}\uprho(\eta)\;\m{d}\eta,
    \end{equation}
    defined for $\uprho\in C^0\tp{2\pi\T}$ and $\tilde{\eta}\in 2\pi\T$. The following identities for $k\in\Z$ are consequences of Lemmas 3.3 and 3.4 in Castro, C\'ordoba, and G\'omez-Serrano~\cite{MR3462104}:
    \begin{equation}\label{special integral identities}
        \int_{0}^{2\pi}K_1(\eta)e^{-\ii k\eta}\;\m{d}\eta = -\f{1}{2}\begin{cases}
            \log 4&\text{if }k=0,\\
            |k|^{-1}&\text{if }k\neq 0,
        \end{cases}
        \text{ and }
        \int_0^{2\pi}K_2(\eta)e^{-\ii k\eta}\;\m{d}\eta = -\f12\begin{cases}
            \log 4 - 2\xi&\text{if }k=0,\\
            e^{-2|k|\xi}|k|^{-1}&\text{if }k\neq 0.
        \end{cases}
     \end{equation}
     Therefore, for any $\uprho\in C^0\tp{2\pi\T}$ and $k\in\Z$ we can compute 
     \begin{equation}\label{Fourier action}
         \mathscr{F}[O_1\uprho](k) = -\f{\mathscr{F}[\uprho]\tp{k}}{2}\begin{cases}
             \log 4&\text{if }k= 0,\\
             |k|^{-1}&\text{if }k\neq 0,
         \end{cases}
         \text{ and }
         \mathscr{F}[O_2\uprho]\tp{k} = -\f{\Bar{\mathscr{F}[\uprho]\tp{k}}}{2}\begin{cases}
             \log 4 - 2\xi&\text{if }k=0,\\
             e^{-2|k|\xi}|k|^{-1}&\text{if }k\neq0.
         \end{cases}
     \end{equation}

     We now have all of the tools we need to prove the stated equivalence. Identity~\eqref{the spectral condition of Kirchhoff} holds for a $\phi\in C^0\tp{\Sigma^\xi}$ if and only if for $\rho = \phi\circ\bf{e}\tp{\xi,\cdot}$ we have
     \begin{equation}\label{the Fourier identity}
         \rho + 2e^{2\xi}\tp{\log\tp{2\sech\tp{\xi}}\mathscr{F}[\rho](0) + O_1\rho + O_2\rho}/\sinh(2\xi) = 0.
     \end{equation}
     Then equality~\eqref{the Fourier identity} holds if and only if each Fourier coefficient of the left hand side vanishes. Since $\phi$ is $e_2$-symmetric, we deduce $\rho(\eta) = \rho(-\eta)$ for all $\eta\in 2\pi\T$ and $\mathscr{F}[\rho](k) = \m{Re}\mathscr{F}[\rho](|k|)$ for all $k\in\Z$. We then may use the previous computations to see that this latter condition is equivalent to the desired identities~\eqref{zero vanish cond} and~\eqref{nonzero vanish cond}.
\end{proof}

\begin{coroC}[Nondegeneracy on Kirchhoff stream functions, II]\label{coro on nondegen kirchoff II}
    There exists an at most countable set $\mathcal{N}\subset\R^+$ such that for all $\xi\in\R^+\setminus\mathcal{N}$ the only solution $\phi\in C^0_{\tp{2}}\tp{\Sigma_\xi}$ to equation~\eqref{the spectral condition of Kirchhoff} is $\phi = 0$.
\end{coroC}
\begin{proof}
    As suggested by the second item of Proposition~\ref{kirchhoff nondegeneracy 1}, we consider the sequence of functions $\tcb{\Sampi_k}_{k\in\N}$ defined for $\xi\in\R^+$ via
    \begin{equation}
        \Sampi_0\tp{\xi} = \sinh(2\xi) - 2 e^{2\xi}\log\tp{1 + e^{-2\xi}}\text{ and for }k>0,\;\Sampi_k\tp{\xi} = k\sinh(2\xi) - e^{2\xi}\tp{1 + e^{-2k\xi}}.
    \end{equation}

    We deduce from Proposition~\ref{kirchhoff nondegeneracy 1} that for a fixed $\xi\in\R^+$ there exists $\phi\in C^0_{\tp{2}}\tp{\Sigma^\xi}\setminus\tcb{0}$ satisfying~\eqref{the spectral condition of Kirchhoff} if and only if there exists $k\in\N$ with $\Sampi_k(\xi) = 0$. Thus the result will follow as soon as we show that for each $k\in\N$ there exists at most one $\xi\in\R^+$ with $\Sampi_k\tp{\xi} = 0$.

    For $k = 0$ we show that $\Sampi_0$ has a unique positive zero. This follows from the observations that 1) $\lim_{\xi\to0}\Sampi_0\tp{\xi}=-2\log 2$, 2) $\lim_{\xi\to\infty}\Sampi_0\tp{\xi} = \infty$, and 3) $\Sampi_0'(\xi)>0$.

    For $k\in\tcb{1,2}$, we show that $\Sampi_k$ has no zeros. For $k=1$, we have $\lim_{\xi\to0}\Sampi_1(\xi)=-2$ and $\Sampi_1'(\xi)<0$ for all $\xi\in\R^+$ and hence $\Sampi_1(\xi)< -2$ for all $\xi\in\R^+$. For $k=2$, we calculate that $\Sampi_2(\xi) <0$ for all $\xi\in\R^+$.

    For $\N\ni k\ge 3$, we show that $\Sampi_k$ has a unique positive zero. This follows from the observations: 1) $\lim_{\xi\to0}\Sampi_k(\xi) = -2$, 2) $\lim_{\xi\to\infty}\Sampi_k\tp{\xi} = \infty$, and 3) $\Sampi_k'(\xi)>0$ for all $\xi\in\R^+$.
\end{proof}

Synthesizing the previous calculations gives us the main result of this subsection.
\begin{thmC}[Admissibility of Kirchhoff elliptical vortices]\label{thm on admissibility of Kirchhoff}
    Let $\mathcal{N}\subset\R^+$ be the countable set from Corollary~\ref{coro on nondegen kirchoff II}. For all $\xi\in\R^+\setminus\mathcal{N}$ the Kirchhoff elliptical vortex defined in~\eqref{the kirchhoff ellispses as sets} and~\eqref{KEAF} is an admissible steady rotating vortex patch in the sense of Definition~\ref{defn of admissible patches}.
\end{thmC}
\begin{proof}
    To verify Definition~\ref{defn of admissible patches} we take $\xi\in\R^+\setminus\mathcal{N}$, $\Psi = \Psi^\xi$, $\Omega = \Omega^\xi$, $c = c^\xi$, and $\tilde{U} = E^\xi$. The satisfaction of the first item of the definition is immediate. The second and third items follow from Lemma~\ref{lem on stream functions for Kirchhoff vortices}. The final item is a consequence of Corollary~\ref{coro on nondegen kirchoff II}.
\end{proof}

\subsection{Perturbations of the Rankine vortex}\label{subsection on perturbations of the rankine vortex}

We now turn our focus to the verification of the admissibility of the steady rotating patches with discrete rotation symmetry, locally bifurcating from the Rankine vortex. We start by recalling the set-up of their construction, following Burbea~\cite{MR646163}, Hmidi and Mateu~\cite{MR3054601}, and Hassainia, Masmoudi, and Wheeler~\cite{MR4156612}. 

Let $D$ be a bounded, simply connected domain corresponding to a rotating vortex patch solution, and denote by $\mathbb{D} = B(0,1)$ the open unit-disk with boundary $\S^1 = \partial\mathbb{D}$. Assuming that the boundary $\partial D$ is of $C^{k,\alpha}$ regularity ($k\in\N^+, \al\in(0,1)$), we infer that a conformal mapping $\Phi:\C\setminus\Bar{\mathbb{D}}\to\C\setminus\Bar{D}$ extends to a $C^{k,\alpha}$ mapping $\Phi:\C\setminus\mathbb{D}\to\C\setminus D$, such that $\Phi:\S^1 \to \partial D$ is a $C^{k,\alpha}$ parametrization of the boundary of the patch (see, e.g., Theorems 3.5 and 3.6 in Pommerenke~\cite{MR1217706}). Interpreting the gradient of the relative stream function as a complex scalar under the usual identification $\nabla \Psi = \partial_1 \Psi + \ii \partial_2 \Psi$, equation~\eqref{weak equation 1} on the boundary $\partial D$ of the patch reads 
\begin{equation} \label{compnaphi}
    -\Bar{\nabla \Psi}(\Phi(w)) =  \Omega \overline {\Phi (w)} + \frac{1}{4\pi\ii} \int_{\S^1} \frac{\Bar{\Phi(\tau)} - \Bar{\Phi(w)}}{\Phi(\tau) - \Phi(w)} \Phi'(\tau)\;\m{d}\tau,
\end{equation}
for all $w \in \S^1$. On the other hand, since $\Psi(\Phi(w)) = 0$ for all $w \in \S^1$ and $\ii w \Phi'(w)$ is the tangent vector at $\Phi(w)$, we have that 
\begin{equation}
    \m{Re}\tcb{\overline{\nabla \Psi} (\Phi(w)) \ii w \Phi'(w)  } = 0,
\end{equation}
which when combined with~\eqref{compnaphi} implies 
\begin{equation} \label{Burbea1}
    \m{Im} \bcb{ \bp{ \Omega \overline {\Phi (w)} + \frac{1}{4\pi \ii} \int_{\S^1} \frac{\overline{\Phi(\tau)} - \overline{\Phi(w)}}{\Phi(\tau) - \Phi(w)} \Phi'(\tau)\;\m{d}\tau}w \Phi'(w)} = 0.
\end{equation}
This is a version of the Burbea map, first introduced in~\cite{MR646163}. Following Hassainia, Masmoudi, and Wheeler~\cite{MR4156612}, we condense the notation in~\eqref{Burbea1} in terms of the Cauchy integral operator 
\begin{equation}
    \mathcal{C}(\Phi): f \mapsto \frac{1}{2\pi\ii} \int_{\S^1} \frac{f(\tau) - f(w)}{\Phi(\tau) - \Phi(w)} \Phi'(\tau) \;\m{d}\tau.
\end{equation}
Equation~\eqref{Burbea1}, then, becomes
\begin{equation}
    \m{Im}\tcb{\tp{\Omega \overline{\Phi} + \mathcal{C}(\Phi)\Bar{\Phi}/2 }w \Phi'} = 0.
\end{equation}
It is not difficult to see that, for all $\Omega \in \R$, $\Phi(z) = z$ solves the equation above on $\S^1$.

In the following definition, we instantiate the Burbea map on perturbations around the Rankine vortex that satisfy the reflection symmetry $\overline{\phi(w)} = \phi(\overline w)$.

\begin{defnC}[Burbea map around the Rankine vortex]\label{defn of burbea map and spaces}
  Let $k \ge 1$ and $\alpha \in (0,1)$. Consider the Banach spaces
  \begin{equation}
      X^{k,\alpha} = \bcb{\phi \in C^{k,\alpha}(\S^1; \C)\;:\; \phi(w) = \sum_{n=0}^\infty a_n \Bar{w}^n; a_n \in \R}
  \end{equation}
  and
  \begin{equation}
      Y^{k-1,\alpha} = \tcb{f \in C^{k-1, \alpha}(\S^1; \R)\;:\; f(\overline w) = - f(w) },
  \end{equation}
  and the open subset $\es\neq U^{k,\alpha} \subset X^{k,\alpha}$ defined by
  \begin{equation}
      U^{k,\alpha} = \bcb{\phi \in X^{k,\alpha}\;:\; \inf_{\tau \neq w} \babs{1+\frac{ \phi(\tau) - \phi(w)}{\tau - w} } > 0 }.
  \end{equation}
  We define the Burbea map around the Rankine vortex $\mathcal F:U^{k,\alpha} \times \R \rightarrow Y^{k-1,\alpha}$ by 
  \begin{equation} \label{BurbeaAroundRankine}
      \mathcal{F}(\phi,\Omega) = \m{Im}\tcb{\tp{\Omega\tp{\overline w + \overline {\phi}}+ \mathcal{C}(w + \phi)\tp{\overline w + \overline \phi}/2 }w\tp{1+\phi'}}.
  \end{equation}
\end{defnC}

In the previous definition, the space $X^{k,\alpha}$ does not encode any discrete rotation symmetry for the boundary of the patch. We remark, however, that the spaces can be modified in a way that keeps track of this information.

\begin{rmkC}\label{remark on smoothness}
  The Burbea map considered in Definition~\ref{defn of burbea map and spaces} is well-defined and analytic thanks to the analyticity of the Cauchy integral operator on $w + U^{k,\alpha}$ (see~\cite{MR1719083}; also Lemma 2.6 of~\cite{MR4156612} for a very similar setting to the one we consider here). Moreover, given $m \in \N^+,$ the spaces $X^{k,\alpha}$ and $Y^{k-1,\alpha}$, and the open set $U^{k,\alpha}$ can be replaced by 
  \begin{equation}
      X_m^{k,\alpha} = \bcb{\phi \in X^{k,\alpha}\;:\; \phi(w) = \sum_{n=1}^\infty a_n \overline{w}^{nm-1}},\;Y^{k-1,\alpha}_m = \tcb{f \in Y^{k-1,\alpha}\;:\; f(e^{2\pi\ii/m}w) = f(w) }
  \end{equation}
  and $U^{k,\alpha}_m = U^{k,\alpha} \cap X^{k,\alpha}_m$, respectively (see Section 2.2 in~\cite{MR4156612} or Section 3.7 of~\cite{MR3054601}).
\end{rmkC}

We now state the main existence result for nontrivial steady rotating vortex patches with dihedral symmetry. Most of the content of the following theorem is already proven in the literature in various forms, the closest to the formulation below being that of~\cite{MR4156612}. The novelty consists solely in the triviality of the kernel of $D_1\mathcal{F}$ on the bifurcating curve, locally near the bifurcation point.

\begin{thmC}[Bifurcations from the Rankine vortex]\label{thm on bifurcations from the rankine vortex}
  For $\N\ni m\ge 2$, let $\Omega_m = \tp{m-1}/2m$. There exists $\epsilon,\tilde{\epsilon} > 0$ with $\tilde{\epsilon}\le\epsilon$, an open set $(0,\Omega_m)\in V\subset U^{k,\alpha}\times \R$, and analytic functions $(\phi,\lambda):(-\epsilon, \epsilon) \rightarrow V$ such that the following hold.
  \begin{enumerate}
      \item For the map $\mathcal{F}$ the curve $(\phi, \lambda)$ is the unique bifurcating branch of solutions in $V$: 
      \begin{equation}
          \tp{V \cap \mathcal F^{-1}\{0\}} \setminus\tp{ \{0\}\times\R} = \left\{(\phi(s),\lambda(s))\;:\;0<|s|<\epsilon \right\}.
      \end{equation}
      \item The branch bifurcates from the Rankine vortex (i.e. $\phi(0) = 0$) through the $m$-fold symmetric class (i.e. $\phi(s) \in X^{k,\alpha}_m$ for all $s \in (-\epsilon, \epsilon)$) and satisfies $\phi'(0) = \overline w^{m-1}$.
      \item The bifurcation is of subcritical pitchfork type; in other words $\lambda(0) = \Omega_m$, $\lambda'(0) = 0$, $\lambda''(0) = -\tp{m-1}/2 < 0$. Consequently, for all $0<|s|\le\tilde{\epsilon}$,  $D_1\mathcal{F}(\phi(s), \lambda(s)):X^{k,\alpha} \to Y^{k-1,\alpha}$ has trivial kernel and $\lambda(s)<\Omega_m<1/2$.
  \end{enumerate}
\end{thmC}
\begin{proof}
The proof is based on an application of the Crandall-Rabinowitz Theorem~\cite{MR288640}. With the aim of verifying its assumptions, we begin by computing the Fr\'echet derivative of $\mathcal F$: 
\begin{equation} \label{Frechet1}
    D_1\mathcal{F}(\phi, \Omega)[\xi] = \m{Im} \tcb{\tp{\Omega \overline \xi + D\mathcal{C}(\Phi)[\xi]\overline{\Phi}/2 + \mathcal{C}(\Phi) \overline \xi/2}w\Phi' + \tp{\Omega \overline \Phi + \mathcal{C}(\Phi) \overline \Phi/2}w \xi '},
\end{equation}
where we have used the notation $\Phi(w) = w + \phi(w)$ and 
\begin{equation} \label{Cauchy_der_1}
    D \mathcal{C}(\Phi)[\xi]f(w) = \frac{1}{2\pi\ii} \int_{\S^1} \frac{f(\tau) - f(w)}{\Phi(\tau) - \Phi(w)} \xi'(\tau)\;\m{d}\tau - \frac{1}{2\pi\ii} \int_{\S^1} \frac{(f(\tau) - f(w))(\xi(\tau) - \xi(w))}{(\Phi(\tau) - \Phi(w))^2}\Phi'(\tau)\;\m{d}\tau.
\end{equation}
When we consider $\phi = 0$, simple computations (see, for instance, the proof of Lemma 3.2\footnote{Note that the statement of Lemma 3.2. in~\cite{MR4156612} does not precisely match the form we claim in~\eqref{linearizz}. This is due to a typo in equation (3.1) in the statement of the cited lemma. Following their proof, however, yields the claimed result.} of~\cite{MR4156612}) show that $D \mathcal C(w)[\xi]\overline w = \mathcal C(w) \overline \xi = 0$, and $\mathcal C(w) \overline w = - \overline w$. It follows, then, that
\begin{equation} \label{linearizz}
    D_1 \mathcal{F}(0, \Omega) [\xi] = \m{Im}\tcb{(\Omega - 1/2) \xi' + \Omega w \overline \xi }.
\end{equation}
By using the series expansion of $\xi \in X^{k,\alpha}$, $\xi(w) = \sum_{n=0}^\infty a_n \overline w^n$, we find
\begin{equation}
    D_1\mathcal{F}(0, \Omega)[\xi] =  \frac{1}{2\ii} \sum_{n=0}^\infty a_n \bp{\Omega + n\bp{\Omega - \f12}} \tp{w^{n+1} - \overline w^{n+1}}.
\end{equation}
When $\Omega = \Omega_m$, it is clear that 
\begin{equation}
    \m{ker} D_1\mathcal{F}(0,\Omega_m) = \m{span}\tcb{\overline w^{m-1}} \in X_m^{k,\alpha}.
\end{equation}
Moreover, the range of $D_1\mathcal{F}(0,\Omega_m)$ is of codimension one both when viewed as a mapping from $X^{k,\alpha}\to Y^{k-1,\alpha}$, and when viewed as a mapping from $X^{k,\alpha}_m \to Y_m^{k-1,\alpha}$. Finally, the following transversality condition holds: 
\begin{equation} \label{transvers}
    D_1D_2\mathcal F(0,\Omega_m)[\overline w^{m-1}] =m\tp{w^m - \overline w^m}/2\ii \notin \m{ran} D_1\mathcal{F}(0,\Omega_m).
\end{equation}
Thus the assumptions of Theorem in 1.7 in Crandall and Rabinowitz~\cite{MR288640} are satisfied; however, here we employ the analytic version given by Theorem 8.3.1 in Buffoni and Toland~\cite{MR1956130}. The latter result implies the existence of $\epsilon > 0$ and the analytic functions $(\phi,\lambda)$ claimed in the statement of the theorem such that the first item above holds, $\phi(0) = 0$, $\phi'(0) = \overline w^{m-1}$, and $\lambda(0) = \Omega_m$. Moreover, the fact that $\phi(s) \in X_{m}^{k,\alpha}$ follows from the application of the Crandall-Rabinowitz Theorem both to $\mathcal F:U^{k,\alpha} \times \R \to Y^{k-1,\alpha}$ and to its restriction $\mathcal F:U_m^{k,\alpha} \times \R \to Y_m^{k-1,\alpha}$, together with the uniqueness of the bifurcating branch. It remains, then, to verify the third item. 

The fact that $\lambda'(0) = 0$ (i.e. that the bifurcation is pitchfork) can be readily obtained from symmetry considerations (see the proof of Theorem 3.3. in~\cite{MR4156612}). Nevertheless, we obtain this result here as well, through direct computation, in the process of proving the subcriticality condition $\lambda''(0) < 0$. We claim that 
\begin{equation}\label{pitchfork}
    0 = D_1^2 \mathcal{F}(0, \Omega_m)[\phi'(0),\phi'(0)] + 2 D_1D_2 \mathcal{F}(0,\Omega_m)[\phi'(0)]\lambda'(0) + D_1 \mathcal{F}(0, \Omega_m)[\phi''(0)],
\end{equation}
and, assuming that $\lambda'(0) = 0$, it holds that 
\begin{multline}\label{subcrit}
    0 = D_1^3 \mathcal{F}(0,\Omega_m)[\phi'(0),\phi'(0),\phi'(0)] + 3 D_1^2 \mathcal{F}(0,\Omega_m)[\phi'(0),\phi''(0)]\\ + 3 D_1D_2 \mathcal F(0, \Omega_m)[\phi'(0)] \lambda''(0) + D_1\mathcal{F}(0,\Omega_m)[\phi'''(0)]. 
\end{multline}
By differentiating with respect to $s$ in $\mathcal{F}(\phi(s), \lambda(s)) = 0$, we obtain 
\begin{equation}
    D_1 \mathcal F(\phi(s), \lambda(s))[\phi'(s)] + D_2\mathcal{F}(\phi(s), \lambda(s))\lambda'(s) = 0.
\end{equation}
Differentiating once more in $s$ we obtain 
\begin{multline}
    0 = D_1^2\mathcal{F}(\phi(s), \lambda(s))[\phi'(s), \phi'(s)] + 2D_1D_2\mathcal{F}(\phi(s),\lambda(s))[\phi'(s)] \lambda'(s) + D_1 \mathcal{F}(\phi(s), \lambda(s))[\phi''(s)] \\ + D_2^2 \mathcal F(\phi(s), \lambda(s)) (\lambda'(s))^2 + D_2 \mathcal F(\phi(s), \lambda(s)) \lambda''(s). 
\end{multline}
On the other hand, $D_2^k \mathcal F(0,\Omega) = 0$ for all $\Omega \in \R$ and $k \in \N$. Consequently, evaluating the previous expression at $s = 0$ yields equation~\eqref{pitchfork}. To obtain~\eqref{subcrit}, we differentiate the above once more in $s$ and then evaluate at $s = 0$, while using the assumption $\lambda'(0)=0$.

We now claim that
\begin{equation} \label{Frechet2}
    D^2_1 \mathcal{F}(0,\Omega_m)[\overline w^{m-1},\overline w^{m-1}] = 0
\end{equation}
and that 
\begin{equation} \label{Frechet3}
    D^3_1 \mathcal{F}(0, \Omega_m)[\overline w^{m-1},\overline w^{m-1},\overline w^{m-1}] =3(m-1)m\tp{w^{m}-\overline w^{m}}/4\ii \notin \m{ran}D_1\mathcal{F}(0,\Omega_m)
\end{equation}
Assuming equations~\eqref{Frechet2} and~\eqref{Frechet3} hold for the moment, note that~\eqref{pitchfork} becomes 
\begin{equation}
    2 D_1D_2 \mathcal{F}(0, \Omega_m)[\overline w^{m-1}] \lambda'(0) + D_1\mathcal{F}(0,\Omega_m)[\phi''(0)] = 0.
\end{equation}
By virtue of the transversality condition \eqref{transvers}, we obtain $\lambda'(0) = 0$ and, consequently, that $\phi''(0) \in \m{ker} D_1 \mathcal{F}(0,\Omega_m)$, which implies $\phi''(0) = a \overline w^{m-1}$ for some $a \in \R$. Knowing now that $\lambda'(0) = 0$, equation~\eqref{subcrit} holds unconditionally, and, under the assumptions~\eqref{Frechet2} and~\eqref{Frechet3}, this equation becomes 
\begin{equation}
    D_1^3 \mathcal{F}(0,\Omega_m)[\overline w^{m-1},\overline w^{m-1},\overline w^{m-1}] + 3 D_1D_2 \mathcal{F}(0,\Omega_m) \lambda''(0) + D_1\mathcal{F}(0,\Omega_m)[\phi'''(0)] = 0. 
\end{equation}
Since the first two terms above are not in $\m{ran}D_1\mathcal{F}(0,\Omega_m)$, we obtain $\lambda''(0) = -(m-1)/2$, as desired.

With the aim of proving relations \eqref{Frechet2} and \eqref{Frechet3}, we take the second Fr\'echet derivative in \eqref{Frechet1}: 
\begin{equation} \label{Frechet2Calc}
    D^2_1\mathcal{F}(\phi, \Omega)[\xi, \xi] = \m{Im}\tcb{\tp{D^2\mathcal C(\Phi)[\xi, \xi] \Bar{\Phi}/2 +D\mathcal C(\Phi)[\xi]\overline \xi }w \Phi'  + \tp{2\Omega \overline \xi +  D\mathcal C (\Phi)[\xi]\overline \Phi + \mathcal C(\Phi) \overline \xi}w\xi'}
\end{equation}
where 
\begin{multline}
    D^2 \mathcal C(\Phi)[\xi,\xi]f(w) = \frac{2}{2\pi\ii} \int_{\S^1} \frac{(f(\tau)-f(w))(\xi(\tau)-\xi(w))^2}{(\Phi(\tau)- \Phi(w))^3} \Phi'(\tau) \;\m{d}\tau \\- \frac{2}{2\pi\ii} \int_{\S^1} \frac{(f(\tau)- f(w))(\xi(\tau) - \xi(w))}{(\Phi(\tau)-\Phi(w))^2} \xi'(\tau)\;\m{d}\tau.
\end{multline}
Restricting now to the case $\Phi(w) = w$ and $\xi(w) = \overline{w}^{m-1}$, we find that
\begin{multline}
    D^2 \mathcal C(w)[\overline w^{m-1}, \overline w^{m-1}]\overline w = \frac{2}{2\pi\ii} \int_{\S^1} \frac{(\overline \tau - \overline w)(\overline \tau^{m-1} - \overline w^{m-1})^2}{(\tau - w)^3} \;\m{d}\tau \\
     + \frac{2(m-1)}{2\pi\ii} \int_{\S^1} \frac{(\overline \tau - \overline w)(\overline \tau ^{m-1} - \overline w^{m-1})}{(\tau - w)^2} \frac
    {1}{\tau^{m}} \;\m{d}\tau \\ 
    = -\frac{2}{2\pi\ii} \int_{\S^1} \frac{\tp{\tau^{m-2}+...+w^{m-2}}^2}{w^{2m-1}\tau^{2m-1}}\;\m{d}\tau + \frac{2(m-1)}{2\pi\ii}\int_{\S^1} \frac{\tau^{m-2}+...+w^{m-2}}{w^{m}\tau^{2m}}\;\m{d}\tau = 0.
\end{multline}
On the other hand, 
\begin{multline}
    D \mathcal C(w)[\overline w^{m-1}]w^{m-1} = -\frac{m-1}{2\pi\ii} \int_{\S^1} \frac{\tau^{m-1} - w^{m-1}}{\tau -w} \frac{1}{\tau^{m}}\;\m{d}\tau - \frac{1}{2\pi\ii} \int_{\S^1} \frac{(\tau^{m-1} - w^{m-1})(\overline \tau^{m-1} - \overline w^{m-1})}{(\tau - w)^2}\;\m{d}\tau \\ 
    = -\frac{m-1}{2\pi\ii} \int_{\S^1} \frac{\tau^{m-2}+...+w^{m-2}}{\tau^{m}}\;\m{d}\tau + \frac{1}{2\pi\ii} \int_{\S^1} \frac{(\tau^{m-2}+...+w^{m-2})^2}{\tau^{m-1} w^{m-1}}\;\m{d}\tau = \frac{m-1}{w}.
\end{multline}
We have calculated, then, that 
\begin{equation}
    D_1^2 \mathcal{F}(0,\Omega_m)[\overline{w}^{m-1},\overline{w}^{m-1}] = \m{Im}\tcb{(m-1)(1-2\Omega_m)} = 0,
\end{equation}
and \eqref{Frechet2} is proven. To show the validity of \eqref{Frechet3}, we take the third Fr\'echet derivative in \eqref{Frechet2Calc}: 
\begin{multline} \label{Frechet3Calc}
    D_1^3 \mathcal{F}(\phi,\Omega)[\xi, \xi, \xi] = \m{Im}\{\tp{D^3 \mathcal C(\Phi)[\xi,\xi,\xi]\overline \Phi/2 + 3 D^2 \mathcal C(\Phi)[\xi, \xi]\overline \xi/2}w\Phi' \\ + \tp{3D^2\mathcal C(\Phi)[\xi, \xi]\overline \Phi/2 + 3D\mathcal C(\Phi)[\xi]\overline \xi}w\xi'\}, 
\end{multline}
where
\begin{multline}
    D^3\mathcal C(\Phi)[\xi,\xi,\xi] f(w) =\\ \frac{6}{2\pi\ii} \int_{\S^1} \frac{(f(\tau) - f(w))(\xi(\tau) - \xi(w))^2}{(\Phi(\tau) - \Phi(w))^3} \xi'(\tau)\;\m{d}\tau  
     - \frac{6}{2\pi\ii} \int_{\S^1} \frac{(f(\tau) - f(w))(\xi(\tau)-\xi (w))^3}{(\Phi(\tau)-\Phi(w))^4}\Phi'(\tau)\;\m{d}\tau.
\end{multline}
For $\Phi(w) = w$ and $\xi(w) = \overline w^{m-1}$, we have 
\begin{multline}
    D^3 \mathcal C(w)[\overline w^{m-1},\overline w^{m-1},\overline w^{m-1}]\overline w = \\-\frac{6(m-1)}{2\pi \ii} \int_{\S^1} \frac{(\overline \tau - \overline w)(\overline \tau^{m-1}-\overline w^{m-1})^2}{(\tau-w)^3}\frac{1}{\tau^m} \;\m{d}\tau  - \frac{6}{2\pi\ii} \int_{\S^1} \frac{(\overline \tau - \overline w)(\overline \tau^{m-1}-\overline w^{m-1})^3}{(\tau - w)^4}\;\m{d}\tau \\ 
    = \frac{6(m-1)}{2\pi\ii} \int_\T \frac{(\tau^{m-2}+...+w^{m-2})^2}{\tau^{3m-1}w^{2m-1}}\;\m{d}\tau - \frac{6}{2\pi\ii} \int_{\S^1} \frac{(\tau^{m-2}+...+w^{m-2})^3}{\tau^{3m-2}w^{3m-2}}\;\m{d}\tau = 0,
\end{multline}
and 
\begin{multline}
    D^2 \mathcal{C}(w)[\overline w^{m-1},\overline w^{m-1}] w^{m-1} = \\\frac{2}{2\pi\ii} \int_{\S^1} \frac{(\tau^{m-1} - w^{m-1})(\overline \tau^{m-1} - \overline w^{m-1})^2}{(\tau - w)^3}\;\m{d}\tau + \frac{2(m-1)}{2\pi\ii} \int_{\S^1} \frac{(\tau^{m-1} - w^{m-1})(\overline \tau ^{m-1} - \overline w^{m-1})}{(\tau -w)^2} \frac{1}{\tau^{m}}\;\m{d}\tau  \\
    = \frac{2}{2\pi\ii}\int_{\S^1} \frac{(\tau^{m-2}+...+w^{m-2})^3}{\tau^{2m-2}w^{2m-2}}\;\m{d}\tau - \frac{2(m-1)}{2\pi\ii}\int_{\S^1} \frac{(\tau^{m-2}+...+w^{m-2})^2}{\tau^{2m-1}w^{m-1}}\;\m{d}\tau \\ 
    = 2 \sum_{\substack{k_1,k_2,k_3 = 0 \\k_1 + k_2 + k_3 = 2m-3}}^{m-2} \frac{w^{3(m-2)-(k_1+k_2+k_3)}}{w^{2m-2}} = \frac{(m-1)(m-2)}{w^{m+1}}.
\end{multline}
By substituting what we have found in \eqref{Frechet3Calc}, we recover \eqref{Frechet3}.

It remains to argue that the kernel of $D_1 \mathcal F(\phi(s), \lambda(s))$ is trivial for $|s| > 0$ sufficiently small. Since $\lambda'(0) = 0$ and $\lambda''(0) \neq 0$, it is the case that $\lambda'(s) \neq 0$ for all $s \neq 0$ sufficiently close to zero. The conclusion follows by Theorem 1.17 in Crandall and Rabinowitz~\cite{MR288640}.
\end{proof}

Let $\Phi_m:\S^1 \rightarrow \C$ now be defined by $\Phi_m(w) = w + \phi(s)(w)$, with $\phi(s) \in X_m^{k,\alpha}$, for some $0<|s|<\tilde{\epsilon}$ such that $\m{ker}D_1 \mathcal{F}(\phi(s), \lambda(s)) = \tcb{0}$ and $\lambda(s) < \Omega_m < 1/2$. Since $\phi(s) \in U^{k,\alpha}$, we have that 
\begin{equation}
    \inf_{\tau \neq w} \babs{\frac{\Phi_m(\tau)-\Phi_m(w)}{\tau - w} }> 0,
\end{equation}
which implies the injectivity of $\Phi_m$. Consequently, $\Phi_m(\S^1)$ is a $C^{k,\alpha}$ simple closed curve bounding a domain $D_m$. Moreover, since $\phi(s) \in X_m^{k,\alpha}$, we have that 
\begin{equation*}
    \Phi_m(w) = w \bp{1+\sum_{n=1}^\infty a_n \overline w^{nm}},
\end{equation*}
for some sequence $\tcb{a_n}_{n=1}^\infty \subset \R$. Therefore, $\Phi_m$ satisfies the reflection symmetry $\overline {\Phi_m(w)}=\Phi_m(\overline w)$ and the discrete rotation symmetry
\begin{equation}
    \Phi_m(e^{2\pi\ii/m}w)=e^{2\pi\ii/m}\Phi_m(w).
\end{equation}
These symmetries translate to $e_2$-reflection and $m$-fold symmetry, respectively, for the domain $D_m$. Let 
\begin{equation} \label{Psim_def}
    \tilde \Psi_m(x) = \bf{N}\ast \mathds{1}_{D_m} - \lambda_m|x|^2/2\text{ for }x\in\R^2,
\end{equation}
where we denote $\lambda(s) = \lambda_m$. Then,~\eqref{compnaphi} holds with $\Psi = \tilde \Psi_m$, $ \Phi = \Phi_m$, and $\Omega = \lambda_m$, while $(\Phi_m, \lambda_m)$ solves~\eqref{Burbea1}. Consequently $\tilde \Psi_m(\Phi_m(w)) = c_m \in \R$. We then let $\Psi_m = \tilde \Psi_m - c_m$. The tuple $(\Psi_m, \lambda_m, c_m, D_m,m)$ satisfies the first, second, and third items of Definition~\ref{defn of admissible patches}. Indeed, the regularity of $\Psi_m$ follows from Proposition 1 in Bertozzi-Constantin~\cite{MR1207667}, the regularity of $\partial D_m$ is satisfied as long as we choose $k \ge 2$, the fact that $\Psi_m$ vanishes on $\partial D_m$ holds by construction as the third item, and the condition $\min_{\partial D_m}(\nu \cdot \nabla \Psi_m) > 0$ follows from the Hopf boundary point lemma (see, e.g., Lemma 3.4 in Gilbarg and Trudinger~\cite{MR1814364}), since $\Delta \Psi_m = 1 - 2\lambda_m > 0$ in $D_m$.

The remainder of this subsection is devoted to showing that the fourth item in Definition~\ref{defn of admissible patches} also holds and, thus, that these $m$-fold symmetric patches are admissible. We first prove a lemma which will aid us in mapping a non-trivial solution to the nondegeneracy condition to a non-trivial kernel of the linearization of the Burbea map.

\begin{lemC}\label{lem on riemann hilbert}
  Let $k \ge 2$ and $\alpha \in (0,1)$, $(\Psi_m, \lambda_m, c_m, D_m,m)$ be one of the tuples described in the preceding paragraphs corresponding to a solution $(\phi(s),\lambda(s)) \in U^{k,\alpha}_m \times \R$ given by Theorem~\ref{thm on bifurcations from the rankine vortex}, and $\varphi \in C_{\tp{m}}^0(\partial D_m)$ satisfying 
  \begin{equation} \label{AssumKer}
      \varphi(x) + \frac{1}{2\pi}\int_{\partial D_m} \log|x-y| \frac{\varphi(y)}{|\nabla \Psi_m(y)|}\;\m{d}\mathcal{H}^1(y) = 0,
  \end{equation}
  for all $x \in \partial D_m$. Then, there exists a solution $\delta\in C^{k-1,\alpha}(\S^1;\C)$ to the Riemann-Hilbert problem 
  \begin{equation}\label{i cannot believe this wasnt labeled earlier}
      \m{Re}\tcb{\delta(w) \overline {\nabla \Psi_m}(\Phi_m(w)) } = \varphi(\Phi_m(w))\text{ for all } w\in \S^1,
  \end{equation}
  such that
  \begin{equation} \label{deltaseries}
      \delta(w) = \sum_{\Z\ni n\le 1} a_n w^n,
  \end{equation}
  with $\{a_n\}_{\Z\ni n \le 1}\subset \R$.
\end{lemC}
\begin{proof}
We claim first that $\varphi(\Phi_m) \in C^{k,\alpha}(\S^1;\R)$. Since $(\Psi_m,\Phi_m)$ satisfy~\eqref{compnaphi}, the boundedness of the Cauchy integral operator $\mathcal C(\Phi_m) :C^{k,\alpha}(\S^1;\C)\rightarrow C^{k,\alpha}(\S^1;\C)$ implies that $\nabla \Psi_m(\Phi_m) \in C^{k,\alpha}(\S^1;\C)$. Moreover, we can express equation~\eqref{AssumKer} as 
\begin{equation}
    \varphi(\Phi_m(e^{\ii\theta})) + \frac{1}{2\pi}\int_{2\pi\T} \log\tabs{\Phi_m(e^{\ii\theta})-\Phi_m(e^{\ii\vartheta}) } \frac{\varphi(\Phi_m(e^{\ii\vartheta}))}{|\nabla \Psi_m(\Phi_m(e^{\ii\vartheta}))|} |\Phi_m'(e^{\ii\vartheta})|\;\m{d}\vartheta = 0,\quad\theta\in2\pi\T
\end{equation}
Taking a derivative along the unit circle, we obtain 
\begin{equation}
    \pd_{\theta}(\varphi \circ \Phi_m)(e^{\ii\theta}) + \m{Re}\bcb{\ii e^{\ii\theta}\Phi_m'(e^{\ii\theta})\frac{1}{2\pi} \m{p.v.} \int_{2\pi\T}\frac{1}{\Phi_m(e^{\ii\theta}) - \Phi_m(e^{\ii\vartheta})}\frac{\varphi(\Phi_m(e^{\ii\vartheta}))}{|\nabla \Psi_m(\Phi_m(e^{\ii\vartheta}))|} |\Phi_m'(e^{\ii\vartheta})|\;\m{d}\vartheta}=0,
\end{equation}
and, thus, with $w = e^{\ii\theta}$, we have 
\begin{equation}
    \pd_\theta(\varphi \circ \Phi_m)(w) + \m{Re} \bcb{w\Phi_m'(w) \frac{1}{2\pi} \m{p.v.}\int_{\S^1} \frac{g(\tau)}{\Phi_m(w) - \Phi_m(\tau)} \Phi_m'(\tau)\;\m{d}\tau} = 0,
\end{equation}
where 
\begin{equation} \label{def_g}
    g(\tau) = \frac{\varphi(\Phi_m(\tau))|\Phi_m'(\tau)|}{|\nabla \Psi_m(\Phi_m)|\tau\Phi_m'(\tau)}.
\end{equation}
Consequently, we can express the derivative of $\varphi \circ \Phi_m$ in terms of the Cauchy integral operator: 
\begin{equation} \label{BootStra}
    \pd_\theta(\varphi \circ \Phi_m)(w) + \m{Im}\tcb{w\Phi_m'(w)\mathcal C(\Phi_m)g} = 0.
\end{equation}
By the proof of Proposition~\ref{prop on uniform intrgrals estimates for the Newtonian potential}, $\varphi \in C^0(\pd D_m)$ satisfying~\eqref{AssumKer} implies $\varphi \in \m{LL}(\pd D_m) \subset C^\alpha(\pd D_m)$. On the other hand, if $\varphi \circ \Phi_m \in C^{j,\alpha}(\S^1;\R)$, then $g \in C^{\min\{j,k-1\},\alpha}(\S^1;\C)$, so~\eqref{BootStra} implies that $\varphi \circ \Phi_m \in C^{\min\{j,k-1\}+1,\alpha}(\S^1;\R)$. A finite induction argument proves our claim that $\varphi \circ \Phi_m \in C^{k,\alpha}(\S^1; \R)$. 

Since $\mathcal F(\Phi_m(w) - w, \lambda_m) = 0$, we have that $\m{Im}\tcb{\overline{\nabla \Psi_m}(\Phi_m(w))w\Phi_m'(w)} = 0$. Consequently, there exists a nowhere-vanishing function $\tilde F \in C^{k-1,\alpha}(\S^1; \R)$ such that $\overline{\nabla \Psi_m}(\Phi_m(w))w\Phi_m'(w) = \tilde F(w)$. Note that since $\Psi_m$ is reflection invariant, we have that $\overline{\nabla \Psi_m}(w) = \nabla \Psi_m(\overline w)$, and, therefore, $\tilde F(\overline w) = \tilde F(w)$. We write 
\begin{equation} \label{def_F}
    \overline{\nabla\Psi_m}(\Phi_m(w)) = F(w) \overline{\Phi'_m}(w) \overline w,
\end{equation}
with 
\begin{equation}\label{def__F}
    F(w) = \tilde F(w)/|\Phi_m'(w)|^2,
\end{equation}
and note that $F$ is, likewise, non-vanishing and reflection invariant. By the classical theory for the Riemann-Hilbert problem (see, e.g., Muskhelishvili~\cite{MR58845} or Theorem A.1. in~\cite{MR4156612} for a compact statement), since $\Phi'_m$ has vanishing winding number, there exists a holomorphic $f : \mathbb D \rightarrow \C$, which extends as a $C^{k-1,\alpha}$ mapping to the closure of the disk, is reflection invariant $\overline {f(w)} = f(\overline w)$, and solves 
\begin{equation}
    \m{Re}\{\Phi'_m(w)f(w)\}= \varphi(\Phi_m(w))/F(w)\text{ for all } w \in \S^1.
\end{equation}
We define, then, $\delta \in C^{k-1,\alpha}(\S^1;\C)$ by $\delta(w) = w \overline{f(w)}$. We have that 
\begin{equation}
    \m{Re}\tcb{\delta(w) \overline{\nabla \Psi_m}(\Phi_m\tp{w})} = F(w) \m{Re}\tcb{\Phi_m'(w) f(w)} = \varphi(\Phi_m(w)),
\end{equation}
as desired. Moreover, since $f$ is holomorphic in the disk and reflection invariant, there exists a sequence $\{b_n\}_{n \in \N}\subset \R$ such that $f(w) = \sum_{n=0}^\infty b_n w^n$. Identity~\eqref{deltaseries} is then satisfied provided that $a_n = b_{1-n}$.
\end{proof}

At last, we are now ready to prove the main result of this subsection. 

\begin{thmC}[Admissibility of Rankine vortex perturbations]\label{thm on admissibility of rankine vortex perturbations}
  Let $\N \ni k \ge 2$ and $\al \in (0,1)$. There exists $0 < \epsilon' \le \tilde{\epsilon}$, with $\tilde \epsilon>0$ given by Theorem~\ref{thm on bifurcations from the rankine vortex}, such that the tuple $(\Psi_m, \lambda_m, c_m, D_m,m)$ corresponding, via Theorem~\ref{thm on bifurcations from the rankine vortex} and the subsequent discussion, to a solution $(\phi(s), \lambda(s))\in U_m^{k,\alpha} \times \R$, $0<|s|\le \epsilon'$ is admissible in the sense of Definition~\ref{defn of admissible patches}.
\end{thmC}
\begin{proof}
  As already remarked in the discussion following the proof of Theorem~\ref{thm on bifurcations from the rankine vortex}, the first, second, and third items in Definition~\ref{defn of admissible patches} are easily seen to hold for the tuple $(\Psi_m,\lambda_m,c_m,D_m,m)$. For the purpose of obtaining a contradiction, assume that item four fails to hold; namely, that there exists a nontrivial solution $\varphi \in C_{\tp{m}}^0(\partial D_m)\setminus\tcb{0}$ satisfying equation~\eqref{AssumKer}. Then, Lemma~\ref{lem on riemann hilbert} applies and yields a function $\delta \in C^{k-1,\alpha}(\S^1;\C)$ which satisfies~\eqref{i cannot believe this wasnt labeled earlier} and, by the assumed non-triviality of $\varphi$, does not vanish identically. 

  Note that, while we have defined $\mathcal F$ on the domains $U^{k,\alpha}$ for the purpose of applying the Crandall-Rabinowitz Theorem, the expression~\eqref{BurbeaAroundRankine} is also well-defined on the Banach space 
  \begin{equation}
      X_{\m{ext}}^{k-1,\alpha} = \bcb{\phi \in C^{k-1,\alpha}(\S^1,\C)\;:\;\phi(w) = \sum_{n=-1}^\infty a_n \overline w^{n}; a_n \in \R},
  \end{equation}
  of which $\delta$ is an element. The Fr\'echet derivative $D_1 \mathcal F(\Phi_m -w,\lambda_m)$ exists, then, on $X_{\m{ext}}^{k,\alpha}$ and its expression coincides with the one in \eqref{Frechet1}. We claim that $D_1 \mathcal F(\Phi_m -w, \lambda_m)[\delta] = 0$. To this end, we compute 
  \begin{multline} \label{der_of_Re}
      \partial_\theta \m{Re}\tcb{\delta(e^{\ii \theta}) \overline{\nabla \Psi_m }(\Phi_m(e^{\ii\theta}))} =  
      \m{Im}\{e^{\ii \theta} \delta'(e^{\ii \theta})\tp{\lambda_m \overline{\Phi_m}(e^{\ii \theta}) + \mathcal{C}(\Phi_m)\overline{\Phi_m}(e^{\ii \theta})/2} \\
     - \ii  \delta(e^{\ii \theta})\tp{\lambda_m \overline{\partial_\theta \Phi_m}(e^{i\theta}) + \partial_\theta\tp{\mathcal C(\Phi_m)\overline \Phi_m}(e^{\ii \theta})/2}\}.
  \end{multline}
  For the tangential derivative of the Cauchy integral operator, we have 
  \begin{multline}
      \partial_\theta \mathcal C(\Phi_m)\overline \Phi_m(e^{\ii \theta}) = \frac{1}{2\pi \ii} \m{p.v.} \int_{2\pi \T} \frac{\overline{\partial_\theta \Phi_m}(e^{\ii \theta})}{\Phi_m(e^{\ii \theta})-\Phi_m(e^{\ii \vartheta})}\ii e^{\ii \vartheta} \Phi_m'(e^{\ii \vartheta})\;\m{d}\vartheta \\  + \frac{\partial_\theta \Phi_m(e^{\ii \theta})}{2\pi \ii} \m{p.v.} \int_{2 \pi \T} \tp{\overline \Phi_m(e^{\ii \theta}) - \overline \Phi_m(e^{\ii \vartheta})}\partial_\vartheta \frac{1}{\Phi_m(e^{\ii \vartheta}) - \Phi_m(e^{\ii \theta})}\;\m{d}\theta \\ 
      = - \frac12 \overline{\partial_\theta \Phi_m}(e^{\ii \theta}) + \frac{\partial_\theta \Phi_m(e^{\ii \theta})}{2\pi \ii} \m{p.v.} \int_{2 \pi \T} \frac{\overline{\partial_\vartheta \Phi_m}(e^{\ii \vartheta })}{\Phi_m(e^{\ii \vartheta})-\Phi_m(e^{\ii \theta})}\;\m{d}\vartheta  
      = \partial_\theta \Phi_m(e^{\ii \theta}) \mathcal C(\Phi_m) \bsb{\frac{\overline{\partial_\theta \Phi_m}}{\partial_\theta \Phi_m}}(e^{\ii \theta}).
  \end{multline}
  Using the notation $w = e^{\ii \theta}$, we can express \eqref{der_of_Re} as 
  \begin{multline} \label{specific_tan_der}
      \partial_\theta \m{Re}\{\delta \overline{\nabla \Psi_m}(\Phi_m)  \}(w) = \m{Im} \{\tp{\lambda_m \overline {\Phi_m}(w) + \mathcal C(\Phi_m)\overline{\Phi_m}(w)/2 }w \delta'(w) \\ 
      + \tp{\lambda_m \overline \delta(w) + \delta(w) \mathcal{C}(\Phi_m)[\overline{\partial_\theta \Phi_m} / \partial_\theta \Phi_m] /2 }w \Phi_m'(w) \}.
  \end{multline}
  On the other hand, with $\varphi\circ \Phi_m = \m{Re}\{\delta \overline {\nabla \Psi_m}(\Phi_m)\}$, equation \eqref{def_g} becomes 
  \begin{multline} \label{specific_g}
      g(w) = \frac{1}{w \Phi_m'(w)} \m{Re}\bcb{\delta(w) \frac{\overline{\nabla \Psi_m}(\Phi_m(w))}{|\nabla \Psi_m(\Phi_m(w))|}|\Phi'_m(w)| }\\
      = \frac{1}{w \Phi_m'(w)} \m{sgn}(F) \m{Re}\tcb{\delta(w) \overline{\Phi_m'(w)}\overline w}= \frac{1}{2}\bp{\overline \delta(w) - \delta(w)\frac{\overline{\partial_\theta \Phi_m}}{\partial_\theta \Phi_m}(w)},
  \end{multline}
  where, for the second equality, we have recalled that $F$ is defined in~\eqref{def__F} and used \eqref{def_F}. For the third equality, the fact that $\m{sgn}(F) = 1$ was invoked. The latter can be deduced as follows: from \eqref{def_F}, it follows that, as long as $\|\Phi_m - w\|_{C^1(\S^1)}$ is sufficiently small,
  \begin{equation}
      \m{sgn}(F) = \m{sgn}\tp{\m{Re}\tcb{\overline{\nabla \Psi_m}w\Phi_m'}} = \m{sgn}\tp{\m{Re}\tcb{\overline{\nabla \Psi_m}w}} = \m{sgn}(\pd_{\m{r}}\Psi_m)=1,
  \end{equation}
  with $\pd_{\m{r}}$ denoting the derivative in the radial direction. Consequently, by plugging \eqref{specific_tan_der} and \eqref{specific_g} into \eqref{BootStra}, we obtain
  \begin{multline} \label{PreKer_delta}
      \m{Im} \{\tp{\lambda_m \overline {\Phi_m}(w) + \mathcal C(\Phi_m)\overline{\Phi_m}(w)/2 }w \delta'(w)  
      + \tp{\lambda_m \overline \delta(w) + \mathcal C(\Phi_m)\overline \delta(w)/2-\mathcal C(\Phi_m)[ \delta \overline{\partial_\theta \Phi_m}/\partial_\theta \Phi_m]/2 \\+ \delta(w) \mathcal{C}(\Phi_m)[\overline{\partial_\theta \Phi_m} / \partial_\theta \Phi_m] /2 }w \Phi_m'(w) \} = 0.
  \end{multline}
  Finally, we note that, by integrating the second term by parts in \eqref{Cauchy_der_1} with $\xi = \delta$ and $f = \overline {\Phi_m}$, we obtain
  \begin{multline} \label{Cauchy_comm}
      D\mathcal C(\Phi_m)[\delta]\overline{\Phi_m}(w) =\\ - \frac{1}{2\pi\ii}\int_{2\pi\T}\frac{\overline{\partial_\vartheta \Phi_m}(e^{\ii \vartheta})(\delta(e^{\ii \vartheta})-\delta(e^{\ii \theta}))}{\Phi_m(e^{\ii \vartheta})-\Phi_m(e^{\ii \theta})}\;\m{d}\vartheta = - \frac{1}{2\pi\ii}\int_{\S^1}\frac{\overline{\partial_\theta \Phi_m}}{\partial_\theta \Phi_m}(\tau)\frac{(\delta(\tau)-\delta(w))}{\Phi_m(\tau)-\Phi_m(w)} \Phi_m'(\tau)\;\m{d}\tau\\ = \delta(w) \mathcal{C}(\Phi_m)[\overline{\partial_\theta \Phi_m} / \partial_\theta \Phi_m](w) -  \mathcal{C}(\Phi_m)[\delta\overline{\partial_\theta \Phi_m} / \partial_\theta \Phi_m](w).
  \end{multline}
  In view of \eqref{Frechet1}, equations \eqref{PreKer_delta} and \eqref{Cauchy_comm} prove our claim that $D_1 \mathcal F(\Phi_m -w, \lambda_m)[\delta] = 0$.

  If the coefficient $a_1 \in \R$ in the series expansion \eqref{deltaseries} vanishes, then $\delta $ belongs to the smaller space $X^{k-1,\alpha}$ and hence is a non-trivial element in the kernel of $D_1 \mathcal{F}(\Phi_m-w,\lambda_m)$, contradicting item three of Theorem \ref{thm on bifurcations from the rankine vortex}. Assume, then, $a_1 \neq 0$. First, we have that, for all $t \in \R$, it holds that
  \begin{equation}
      \mathcal F(\Phi_m-w+t\Phi_m, \lambda_m) = (1+t)^2 \mathcal F(\Phi_m -w,\lambda_m) = 0,
  \end{equation}
  and, thus, $\Phi_m \in \m{ker}D_1\mathcal F(\Phi_m-w,\lambda_m)$. Consequently, $\delta - a_1\Phi_m \in X^{k-1,\alpha}\cap \m{ker}D_1 \mathcal F(\Phi_m-w,\lambda_m)$. We will obtain the desired contradiction once we show that $\delta \neq a_1 \Phi_m$. 
  
  Under the assumption that it is, in fact, the case that $\delta = a_1 \Phi_m$, it follows that $\varphi(\Phi_m) = \m{Re}\tcb{\Phi_m \overline {\nabla \Psi_m}}$ yields a solution to \eqref{AssumKer}. Switching to a real-variables formulation, this implies that
  \begin{equation}\label{thing we will later say absurd}
      \varphi(x) = x \cdot \grad \Psi_m(x)\text{ for all }x\in\pd D_m.
  \end{equation}
  Consequently,
\begin{multline*}
    \varphi(x) = - \frac{1}{2\pi} \int_{\partial D_m} \log |x-y| y \cdot \frac{\nabla \Psi_m(y)}{|\nabla \Psi_m(y)|}\;\m{d}\mathcal H^1(y) 
    \\= - \frac{1}{2\pi} \int_{D_m} \grad_y\cdot\tp{\log|x-y| y }\;\m{d}y 
    = -\frac{1}{2\pi} \int_{D_m} \frac{y-x}{|y-x|^2} \cdot (y-x+x)\;\m{d}y - \frac{1}{\pi} \int_{D_m} \log|x-y|\;\m{d}y 
    \\= -\frac{1}{2\pi}\mathcal{H}^2(D_m) - \frac{1}{2\pi} x \cdot \int_{D_m} \nabla_y\log|x-y|\;\m{d}y - \frac{1}{\pi} \int_{D_m} \log|x-y|\;\m{d}y,
\end{multline*}
where we have applied the divergence theorem, keeping in mind that $\nabla \Psi_m / |\nabla \Psi_m|$ is the outward-pointing unit normal on $\partial D_m$, and used the notation $\mathcal H^2$ to denote the $2$-dimensional Hausdorff measure on $\R^2$. By the definition \eqref{Psim_def} of $\Psi_m$,
\begin{equation}
    \nabla \Psi_m(x) + \lambda_m x = \frac{1}{2\pi}\int_{D_m} \nabla_x \log|x-y|\;\m{d}y = -\frac{1}{2\pi}\int_{D_m} \nabla_y \log|x-y|\;\m{d}y,
\end{equation}
which implies 
\begin{equation}
    \frac{1}{2\pi} x \cdot \int_{D_m} \nabla_y\log|x-y|\;\m{d}y = - x\cdot \nabla \Psi_m(x) - \lambda_m |x|^2.
\end{equation}
It follows that 
\begin{equation}
    \varphi(x) = -\frac{1}{2\pi}\mathcal H^2(D_m) + x \cdot \nabla \Psi_m(x) + 2 \bp{\lambda_m \frac{|x|^2}{2} - \frac
    {1}{2\pi} \int_{D_m} \log|x-y|\;\m{d}y}.
\end{equation}
On the other hand, since
\begin{equation}
    \Psi_m(x) + \lambda_m \frac{|x|^2}{2} - \frac{1}{2\pi}\int_{D_m}\log|x-y|\;\m{d}y + c_m =0,
\end{equation}
we conclude that 
\begin{equation}
    \varphi(x) = - \mathcal H^2(D_m)/2\pi + x \cdot \nabla \Psi_m(x) - 2\Psi_m(x) - 2c_m.
\end{equation}
Restricting to $x \in \partial D_m$, the latter reads
\begin{equation}
    \varphi(x) = x \cdot \nabla \Psi_m(x) - \tp{\mathcal H^2(D_m) + 4\pi c_m}/2\pi,
\end{equation}
which, by equality~\eqref{thing we will later say absurd}, is absurd whenever 
\begin{equation} \label{contracond}
    \mathcal H^2(D_m) + 4 \pi c_m \neq 0.
\end{equation}

We show now by a continuity argument that \eqref{contracond} holds for all $m$-fold symmetric patches near the Rankine vortex. The stream function of the latter circular vortex has the following explicit form 
\begin{equation}
    \Psi_\Omega = \frac{1}{4} \begin{cases}
        (1-2\Omega)(|x|^2 - 1) &\text{for } x \in \mathbb D, \\ 
        2 \log|x| - 2\Omega(|x|^2-1) &\text{for } x \in  \R^2\setminus\mathbb D,
    \end{cases}
\end{equation}
in terms of the angular velocity $\Omega \in \R$ and, therefore, $c = \mathfrak{c}_\Omega = - \Omega/2$ as in~\eqref{the stream function form of the equation}. Since $\mathcal H^2(\mathbb{D}) = \pi$, we obtain that \eqref{contracond} holds for all $\Omega \neq 1/2$; in particular, it holds for all bifurcation angular velocities $\Omega_m$ given in Theorem \ref{thm on bifurcations from the rankine vortex}. On the other hand, we have that 
\begin{equation}
    B(0,1-\|\Phi_m-w\|_{C^0(\S^1)}) \subset D_m \subset B(0,1+\|\Phi_m-w\|_{C^0(\S^1)}),
\end{equation}
Then, 
\begin{equation}
    \mathcal H^2(D_m) \ge \mathcal H^2\tp{B(0,1-\|\Phi_m-w\|_{C^0(\S^1)})} = \mathcal H^2(\mathbb D)\tp{1-\|\Phi_m - w\|_{C^0(\mathbb S^1)}}^2 \ge \mathcal H^2(\mathbb D) / 2,
\end{equation}
as long as $\|\Phi_m - w\|_{C^0(\mathbb S^1)}$ is sufficiently small. Moreover, for any $x \in \partial D_m$,
\begin{equation}
    c_m = -\lambda_m |x|^2/2 + \frac{1}{2\pi} \int_{D_m} \log|x-y|\;\m{d}y.
\end{equation}
The latter is clearly a continuous mapping of $(\Phi_m,\lambda_m)\in  C^0(\S^1)\times \R$ and, thus, $c_m \ge\mathfrak{c}_{\Omega_m}/2$ provided $\|\Phi_m -w\|_{C^0(\S^1)}$ and $|\lambda_m - \Omega_m|$ are chosen sufficiently small. We obtain, then, that 
\begin{equation}
    \mathcal H^2(D_m) + 4 \pi c_m \ge \tp{\mathcal H^2(\mathbb D) + 4\pi c_{\Omega_m} }/2 > 0,
\end{equation}
and, therefore~\eqref{contracond} holds, thus showing the non-triviality of $\delta - a_1 \Phi_m \in X^{k,\alpha}\cap\ker D_1 \mathcal F(\Phi_m-w,\lambda_m)$. This was the desired contradiction in the case $a_1 \neq 0$ and we conclude that there is no non-trivial solution $\varphi \in C_{\tp{m}}^0(\partial D_m)$ to \eqref{AssumKer}. The admissibility of the tuple $(\Psi_m, \lambda_m, c_m, D_m,m)$ is, thus, verified. 
\end{proof}

\section{Analysis near admissible states}\label{section on analysis near admissible states}

Let us instantiate, for the entirety of Section~\ref{section on analysis near admissible states}, a fixed but arbitrary $m\in\N^+$ and tuple $\Psi\in C^{1,1}_{\tp{m}}\tp{\R^2}$, $\Omega,c\in\R$, and $\es\neq\tilde{U}\subset\R^2$ satisfying the admissibility conditions of Definition~\ref{defn of admissible patches}. We know that such a tuple exists thanks to the results of Section~\ref{section on existence of admissible states}. Recall that $\Sigma = \pd\tilde{U}$ denotes the patch boundary and also has $m$-fold dihedral symmetry.

We shall also fix the following other objects. Let $\gam\in C^\infty\tp{\R}$ be any smooth function that satisfies $\gam(t) = 0$ for $t\le - 1$ and $\gam(t) = 1$ for $t\ge 1$. Also let $M\in\N$ and $\tcb{\sig_k}_{k=-M}^M\subset\R$ satisfy
\begin{equation}\label{normalization of the splitting}
  \sum_{k=-M}^M\sigma_k = 1.
\end{equation}

For $\Qoppa\in[0,1]$, the Newtonian potential or the Poisson kernel is a function
\begin{equation}
    \bf{N}_\Qoppa:\tcb{\tp{x,y}\in\tp{\R^2}^2\;:\;x\neq y\text{ and }x\neq y/\Qoppa|y|^2}\to\R
\end{equation}
defined via
\begin{equation}\label{generalized Newtonian potentials}
  \bf{N}_\Qoppa(x,y) = \f{1}{4\pi}\log\tabs{x - y}^2 - \f{1}{4\pi}\log\tp{1 - 2\Qoppa x\cdot y + \Qoppa^2|x|^2|y|^2}.
\end{equation}
Perhaps unsurprisingly, the kernel $\bf{N}_\Qoppa$~\eqref{generalized Newtonian potentials} is related to the Biot-Savart kernels of equation~\eqref{bio savart in a domain} via 
\begin{equation}
    \grad_1\bf{N}_{\Qoppa}\tp{x,y} = \bf{K}_{\Qoppa^{-1/2}}\tp{x,y}.
\end{equation}

The operator given by integration against the kernels~\eqref{generalized Newtonian potentials} is denoted, by an abuse of notation, $\bf{N}_\Qoppa:C^0_{\m{c}}\tp{B(0,\Qoppa^{-1/2})}\to C^0\tp{B(0,\Qoppa^{-1/2})}$ and has the action
\begin{equation}\label{generalized poisson solver}
  \bf{N}_{\Qoppa}\tsb{f}(x) = \int_{\R^2}\bf{N}_{\Qoppa}\tp{x,y}f(y)\;\m{d}y,\text{ for }f\in C^0_{\m{c}}\tp{B(0,\Qoppa^{-1/2})}\text{ and }x\in\R^2.
\end{equation}
This operator has the property that, for all $f\in C^0_{\m{c}}\tp{B(0,\Qoppa^{-1/2})}$, in the sense of weak solutions the following equations are satisfied
\begin{equation}
  \Delta\tp{\bf{N}_{\Qoppa}[f]} = f \text{ in }B(0,\Qoppa^{-1/2})\text{ and }\bf{N}_{\Qoppa}[f] = -\f{\log\Qoppa}{4\pi}\int_{\R^2}f(y)\;\m{d}y\text{ on }\pd B(0,\Qoppa^{-1/2}).
\end{equation}
Moreover, the operators $\bf{N}_{\Qoppa}$ preserve $m$-fold dihedral symmetry.

\subsection{Preliminary estimates and identities}\label{subsection on preliminary estimates and identities}

Before we can get to the heart of our desingularization, splitting, and trapping constructions, we need to first set the stage. This, admittedly lengthy and technical, subsection paves the way by setting up the functional framework and operators along with recording and justifying a whirlwind of estimates.

Our first result of the subsection unpacks some important immediate consequences of Definition~\ref{defn of admissible patches}.

\begin{lemC}[Simple consequences of admissibility]\label{lem  on simple consequences of admissibility}
  There exists an open, bounded, and connected set $U\subset\R^2$ that has $m$-fold dihedral symmetry and a class $C^{1,1}$ boundary such that the following hold.
  \begin{enumerate}
    \item We have the inclusion $\Sigma\subset U$.
    \item We have the equality $\Sigma = \tcb{x\in U\;:\;\Psi(x) = 0}$.
    \item We have the estimate $2\min_{\Bar{U}}\tabs{\grad\Psi}\ge\min_{\Sigma}\tp{\nu\cdot\grad\Psi}$.
    \item The complement $\R^2\setminus\Bar{U}$ has two connected components $U^{\m{in}},U^{\m{out}}\subset\R^2$; these satisfy
    \begin{equation}
        U^{\m{in}}\subset\tilde{U},\;\m{dist}\tp{U^{\m{in}},U^{\m{out}}}>0,\text{ and }U^{\m{out}}\cap\tilde{U} = \es.
    \end{equation}
    See Figure~\ref{fig on the Uberhoods} for a depiction of $U$, $U^{\m{in}}$, and $U^{\m{out}}$.
  \end{enumerate}
\end{lemC}
\begin{proof}
    The $U$ shall be taken as an appropriate topological tubular neighborhood of $\Sigma$. By a simple compactness argument using the properties from the second item of Definition~\ref{defn of admissible patches}, there exists $0<\uprho\le 1$ such that 
    \begin{equation}
        \hat{U} = \tcb{x\in\R^2\;:\;\m{dist}\tp{x,\Sigma}<\uprho}\text{ satisfies }\inf\tcb{\tabs{\grad\Psi(x)}\;:\;x\in\hat{U}}\ge\min_{\Sigma}\tp{\nu\cdot\grad\Psi}/2.
    \end{equation}
    Now, there exists $\sampi_1,\sampi_2\in(0,1]$ sufficiently small so that
    \begin{equation}
        \tcb{x\in\hat{U}\;:\;|\Psi(x)|\le\sampi_1}\subset\tcb{x\in\hat{U}\;:\;\m{dist}\tp{x,\pd\hat{U}}>\sampi_2}.
    \end{equation}
    We therefore may take $U = \tcb{x\in\hat{U}\;:\;|\Psi(x)|<\sampi_1}$. The items of the lemma are now simple consequences of this definition.
\end{proof}

\begin{figure}[!h]
    \centering
    \includegraphics[width=0.5\linewidth]{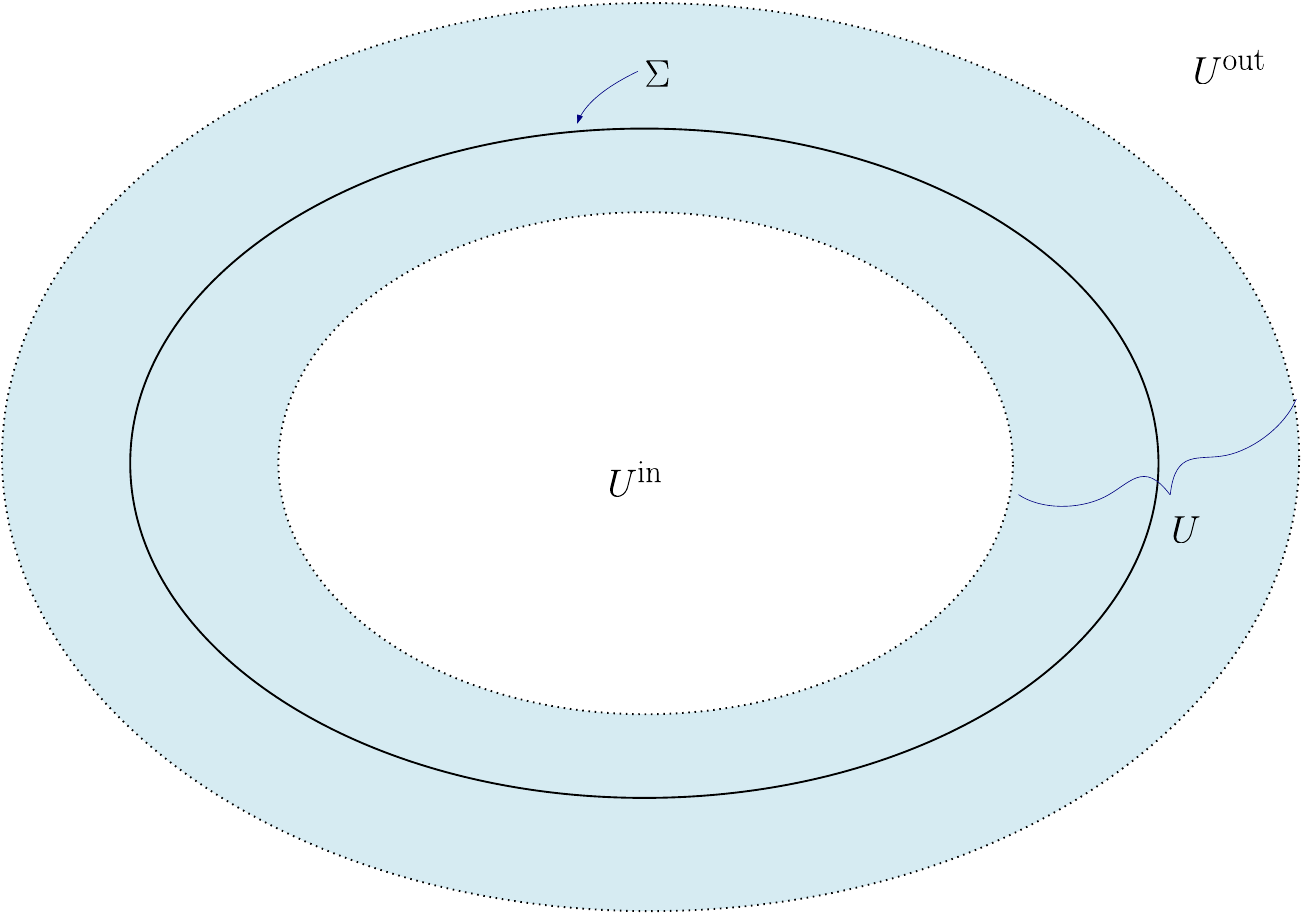}
    \caption{Shown here are the sets $U$, $U^{\m{in}}$, and $U^{\m{out}}$ produced by Lemma~\ref{lem  on simple consequences of admissibility} and their relation to $\Sigma$ (the curve in solid black). The region in blue is the topological tubular neighborhood $U$. The set $U^{\m{in}}$ is the bounded connected component of the white region. The set $U^{\m{out}}$ is the complement of $\Bar{U\cup U^{\m{in}}}$. The support of the vorticity $\tilde{U}$ is not emphasized in the diagram, but it is the bounded simply connected region enclosed by $\Sigma$.}
    \label{fig on the Uberhoods}
\end{figure}

In the next result, we introduce a family of spatially localized regularizations to the Heaviside function and identify a few basic properties.

\begin{lemC}[Regularized Heaviside operators]\label{lem on regularized hevi}
  There exists $\ep_0,r_0,\varrho_0,\Qoppa_0\in(0,1]$ such that the following hold for all $0<\ep\le\ep_0$.
  \begin{enumerate}
    \item For all $\psi\in C^1_{\tp{m}}\tp{\Bar{U}}$ satisfying $\tnorm{\psi}_{C^1\tp{\Bar{U}}}\le \tp{M+1}r_0$ we have 
    \begin{equation}
      \m{dist}\tp{\pd U,\tcb{x\in U\;:\;|\Psi(x) + \psi(x)|\le\ep}}\ge\varrho_0\text{ and }4\min_{\Bar{U}}|\grad\tp{\Psi + \psi}|\ge\min_{\Sigma}\tp{\nu\cdot\grad\Psi}.
    \end{equation}
    \item The regularized Heaviside operator $\Gamma^\Psi_\ep:\Bar{B(0,(M+1)r_0)}\subset C^1_{\tp{m}}\tp{\Bar{U}}\to C^1_{\tp{m},\m{c}}\tp{\R^2}$ with action given via
    \begin{equation}\label{heaviside}
      \tsb{\Gamma^\Psi_\ep\tp{\psi}}\tp{x} = \begin{cases}
        \gam\sp{-\f{\Psi(x) + \psi(x)}{\ep}}&\text{if }x\in\Bar{U},\\
        1&\text{if }x\in U^{\m{in}},\\
        0&\text{if }x\in U^{\m{out}},
      \end{cases}
    \end{equation}
    is well defined; moreover its restrictions satisfy $C_{\tp{m}}^\ell\tp{\Bar{U}}\to C^\ell_{\tp{m},\m{c}}\tp{\R^2}$ for every $\ell\in\N^+$.
    \item The derivative regularized Heaviside operator $\tp{\Gamma_\ep^\Psi}':\Bar{B(0,(M+1)r_0)}\subset C^1_{\tp{m}}\tp{\Bar{U}}\to C^1_{\tp{m},\m{c}}\tp{\R^2}$ with action given by
    \begin{equation}\label{derivative heaviside}
      \tsb{\tp{\Gamma^\Psi_\ep}'\tp{\psi}}\tp{x} = \begin{cases}
        -\f{1}{\ep}\gamma'\sp{-\f{\Psi(x) + \psi(x)}{\ep}}&\text{if }x\in\Bar{U},\\
        0&\text{if }x\in\R^2\setminus\Bar{U},
      \end{cases}
    \end{equation}
    is well defined; moreover its restrictions satisfy $C_{\tp{m}}^\ell\tp{\Bar{U}}\to C^\ell_{\tp{m},\m{c}}\tp{\R^2}$ for every $\ell\in\N^+$.
    \item For all $\psi\in C^1_{\tp{m}}\tp{\Bar{U}}$ we have that $\m{supp}\tp{\Gamma^\Psi_\ep\tp{\psi}}\subset B(0,\Qoppa_0^{-1/2}/2)$.
    \item For all $\psi\in C^1_{\tp{m}}\tp{\Bar{U}}$ we have that $\m{supp}\tp{\tp{\Gamma^\Psi_\ep}'\tp{\psi}}\subset U$.
  \end{enumerate}
\end{lemC}
\begin{proof}
  The first item follows simply by selecting sufficiently small parameters. The second item follows from the first in that: For $x\in\Bar{U}$ in an open neighborhood of the boundary component $\pd U\cap \pd U^{\m{in}}$ (resp. $\pd U\cap \pd U^{\m{out}}$) we have that the composition $\gam\tp{-\tp{\Psi(x) + \psi(x)}/\ep}$ is identically $1$ (resp. $0$). Thus the pointwise definition~\eqref{heaviside} is smoothly matching. The third item follows by similar observations. The fourth and fifth items are immediate.
\end{proof}

We are now in a position to introduce the main objects of study throughout Section~\ref{section on analysis near admissible states}. In particular, the third item of the following result introduces the \emph{perturbed} patch equation; the overarching goal of Section~\ref{section on analysis near admissible states} is to construct solutions to the operator equation posed in~\eqref{equation - perturbation form}.

\begin{lemC}[Principal actors and their basic properties]\label{lem on principal actors and their basic properties}
  Let $\ep_0$, $r_0$, and $\Qoppa_0$ be as in Lemma~\ref{lem on regularized hevi}. For $0<\ep\le\ep_0$, $0\le\rho\le r_0$, and $0\le\Qoppa<\Qoppa_0$ the following hold.
  \begin{enumerate}
    \item The operators $\pmb{K}_{\ep,\rho,\Qoppa},\pmb{F}_{\ep,\rho,\Qoppa}:\Bar{B(0,r_0)}\subset C^1_{\tp{m}}\tp{\Bar{U}}\to C^1_{\tp{m}}\tp{\Bar{U}}$ with actions given via
    \begin{equation}\label{the nonlienar operators}
      \pmb{K}_{\ep,\rho,\Qoppa}\tp{\psi} = -\sum_{k=-M}^M\sig_k\bf{N}_{\Qoppa}[\Gamma^\Psi_\ep\tp{\psi + k\rho} - \Gamma^\Psi_\ep\tp{k\rho}]\text{ and }\pmb{F}_{\ep,\rho,\Qoppa}\tp{\psi} = \psi + \pmb{K}_{\ep,\rho,\Qoppa}\tp{\psi}
    \end{equation}
    are well-defined and, on the interior of their domain, smooth in the Fr\'{e}chet sense. Here $\bf{N}_{\Qoppa}$ and $\Gamma^\Psi_\ep$ are as defined in equations~\eqref{generalized Newtonian potentials} and~\eqref{heaviside}, respectively.
    \item For any $\psi\in B(0,r_0)\subset C^1_{\tp{m}}\tp{\Bar{U}}$ and $\phi\in C^1_{\tp{m}}\tp{\Bar{U}}$ the derivative of $\pmb{F}_{\ep,\rho,\Qoppa}$ obeys the formula 
    \begin{equation}\label{the linearization identity}
        D\pmb{F}_{\ep,\rho,\Qoppa}\tp{\psi}\phi = \phi + \pmb{\kappa}_{\ep,\rho,\Qoppa}\tp{\psi}\phi
    \end{equation}
    with
    \begin{equation}\label{the linearized operator compact part}
      \pmb{\kappa}_{\ep,\rho,\Qoppa}\tp{\psi}\phi = -\sum_{k=-M}^M\sig_k\bf{N}_{\Qoppa}\tsb{\tp{\Gamma^\Psi_\ep}'\tp{\psi + k\rho}\phi}
    \end{equation}
    \item With the functions $f_{\ep,\rho,\Qoppa}\in C^1_{\tp{m}}\tp{\Bar{U}}$ defined via
    \begin{equation}\label{the source term target}
      f_{\ep,\rho,\Qoppa} = \sum_{k=-M}^M\sig_k\tp{\bf{N}_{\Qoppa}\tsb{\Gamma_\ep^\Psi\tp{k\rho}} - \bf{N}\ast\mathds{1}_{\tilde{U}}}
    \end{equation} 
    the following two items are equivalent for $\psi\in\Bar{B(0,r_0)}\subset C^1_{\tp{m}}\tp{\Bar{U}}$:
    \begin{enumerate}
      \item We have satisfied 
      \begin{equation}\label{equation - perturbation form}
        \pmb{F}_{\ep,\rho,\Qoppa}\tp{\psi} = f_{\ep,\rho,\Qoppa}.
      \end{equation}
      \item We have satisfied
      \begin{equation}\label{equation - natural form}
        \Psi + \psi = \sum_{k = -M}^M\sig_k\bf{N}_{\Qoppa}\tsb{\Gamma_\ep^\Psi\tp{\psi + k\rho}} - \Omega\tabs{\cdot}^2/2 - c.
      \end{equation}
    \end{enumerate}
    \end{enumerate}
\end{lemC}
\begin{proof}
  The second and third items are direct computations, as such we omit their simple proofs.

  Let us turn our attention to the first item. Since the difference between $\pmb{F}$ and $\pmb{K}$ is merely the identity map, it suffices to observe that $\pmb{K}$ is smooth between the stated spaces. To see this, we note that for $\psi\in B(0,r_0)\subset C^1_{\tp{m}}\tp{\Bar{U}}$, $\phi\in C^1_{\tp{m}}\tp{\Bar{U}}$, and $x\in U$ the following formula holds:
  \begin{equation}
      \tsb{\pmb{K}_{\ep,\rho,\Qoppa}\tp{\psi}}\tp{x} = -\sum_{k=-M}^M\sig_k\int_{U}\bf{N}_{\Qoppa}\tp{x,y}\bp{\gam\bp{-\f{\Psi(y) + \psi(y) + k\rho}{\ep}} - \gam\bp{-\f{\Psi(y) + k\rho}{\ep}}}\;\m{d}y.
  \end{equation}
  It is now apparent that $\pmb{K}_{\ep,\rho,\Qoppa}$ is a sum of terms having the form of a bounded linear operator on $C^1_{\tp{m}}\tp{\Bar{U}}$ acting on outer composition with the smooth function $\gam$. Fr\'{e}chet smoothness is now a routine consequence of the converse to Taylor's theorem (see, e.g., Section 2.4B in Abraham, Marsden, and Ratiu~\cite{MR960687}).
\end{proof}

We now further examine the linear map from equation~\eqref{the linearized operator compact part} by proving a series of simple bounds. Note that in what follows the dependence on $\ep$ in the inequalities is likely far from optimal. Indeed, the forthcoming proof invokes truly the most brutal estimates available. It is only the qualitative assertions of this result that are important for the overarching strategy.

\begin{lemC}[Preliminary mapping estimates, I]\label{lem on prelim estimates 1}
  Let $\ep_0$, $r_0$, and $\Qoppa_0$ be the positive values granted by Lemma~\ref{lem on regularized hevi}. There exists $C\in\R^+$ with the property that for all $0<\ep\le\ep_0$, $0\le\rho\le r_0$, $0\le\Qoppa\le\Qoppa_0$, $\psi\in\Bar{B(0,r_0)}\subset C^1_{\tp{m}}\tp{\Bar{U}}$, and $\phi\in C^0_{\tp{m}}\tp{\Bar{U}}$ we have that $\pmb{\kappa}_{\ep,\rho,\Qoppa}\tp{\psi}\phi\in C^1_{\tp{m}}\tp{\Bar{U}}$ with the estimate
  \begin{equation}\label{linear remainder c1}
    \tnorm{\pmb{\kappa}_{\ep,\rho,\Qoppa}\tp{\psi}\phi}_{C^1\tp{\Bar{U}}}\le\tp{C/\ep}\tnorm{\phi}_{C^0\tp{\Bar{U}}}.
  \end{equation}
  If we suppose additionally that $\phi\in C^1_{\tp{m}}\tp{\Bar{U}}$, then $\pmb{\kappa}_{\ep,\rho,\Qoppa}\tp{\psi}\phi\in C^2_{\tp{m}}\tp{\Bar{U}}$ with the estimate
  \begin{equation}\label{linear remainder c2}
    \tnorm{\pmb{\kappa}_{\ep,\rho,\Qoppa}\tp{\psi}\phi}_{C^2\tp{\Bar{U}}}\le\tp{C/\ep^2}\tnorm{\phi}_{C^1\tp{\Bar{U}}}.
  \end{equation}
  If we suppose further that $\tilde{\psi}\in\Bar{B(0,r_0)}\subset C^1_{\tp{m}}\tp{\Bar{U}}$, then we have the estimate
  \begin{equation}\label{linear remainder lipschitz}
    \tnorm{\pmb{\kappa}_{\ep,\rho,\Qoppa}\tp{\psi}\phi - \pmb{\kappa}_{\ep,\rho,\Qoppa}\tp{\tilde{\psi}}\phi}_{C^1\tp{\Bar{U}}}\le\tp{C/\ep^2}\tnorm{\psi - \tilde{\psi}}_{C^0\tp{\Bar{U}}}\tnorm{\phi}_{C^0\tp{\Bar{U}}}.
  \end{equation}
\end{lemC}
\begin{proof}
    We take $C$ to be a finite and positive constant that is allowed to change line by line and depend on $\ep_0$, $r_0$, $\Qoppa_0$, $U$, $U^{\m{in}}$, $U^{\m{out}}$, $\gam$, $M$, $\tcb{\sig_k}_{k=-M}^M$, and $\Psi$.
    
    From simple bounds on $\bf{N}_{\Qoppa}$ and is derivative, we derive the pointwise estimates for $x\in U$:
    \begin{multline}\label{__EE__}
        \tabs{\tsb{\pmb{\kappa}_{\ep,\rho,\Qoppa}\tp{\psi}\phi}\tp{x}}\le \f{C}{\ep}\tnorm{\phi}_{C^0\tp{\Bar{U}}}\int_{U}\tp{\tabs{\log\tabs{x - y}} + 1}\;\m{d}y\\
        \text{ and }\tabs{\grad\tsb{\pmb{\kappa}_{\ep,\rho,\Qoppa}\tp{\psi}\phi}\tp{x}}\le\f{C}{\ep}\tnorm{\phi}_{C^0\tp{\Bar{U}}}\int_{U}\bp{\f{1}{\tabs{x - y}} + 1}\;\m{d}y.
    \end{multline}
    Estimate~\eqref{linear remainder c1} now follows from~\eqref{__EE__}.

    To prove~\eqref{linear remainder c2}, we now integrate by parts to derive the pointwise estimate:
    \begin{equation}
        \tabs{\grad^2\tsb{\pmb{\kappa}_{\ep,\rho,\Qoppa}\tp{\psi}\phi}\tp{x}}\le \f{C}{\ep^2}\tnorm{\phi}_{C^1\tp{\Bar{U}}}\int_{U}\bp{\f{1}{\tabs{x - y}} + 1}\;\m{d}y.
    \end{equation}

    Similarly, inequality~\eqref{linear remainder lipschitz} follows from the fundamental theorem of calculus and the pointwise bounds:
    \begin{multline}
        \tabs{\tsb{\pmb{\kappa}_{\ep,\rho,\Qoppa}\tp{\psi}\phi - \pmb{\kappa}_{\ep,\rho,\Qoppa}\tp{\tilde{\psi}}\phi}\tp{x}}\le \f{C}{\ep^2}\tnorm{\phi}_{C^0\tp{\Bar{U}}}\tnorm{\psi - \tilde{\psi}}_{C^0\tp{\Bar{U}}}\int_{U}\tp{\tabs{\log\tabs{x - y}} + 1}\;\m{d}y\\
        \text{ and }\tabs{\grad\tsb{\pmb{\kappa}_{\ep,\rho,\Qoppa}\tp{\psi}\phi - \pmb{\kappa}_{\ep,\rho,\Qoppa}\tp{\tilde{\psi}}\phi}\tp{x}}\le \f{C}{\ep^2}\tnorm{\phi}_{C^0\tp{\Bar{U}}}\tnorm{\psi - \tilde{\psi}}_{C^0\tp{\Bar{U}}}\int_{U}\bp{\f{1}{\tabs{x - y}} + 1}\;\m{d}y.
    \end{multline}
\end{proof}

To establish more subtle estimates than those of Lemma~\ref{lem on prelim estimates 1}, we require the indispensable tool of uniform local level set coordinates. The construction of these coordinates is digressive from the main narrative line of this section. As such, the proof of the following result is postponed to Appendix~\ref{appendix on uniform local level set coordinates}. In what follows we write for $\ell\in\R^+$ the set $\Bar{\ell Q} = [-\ell,\ell]^2\subset\R^2$.

\begin{lemC}[Uniform local level set coordinates]\label{lem on uniform local level set coordinates}
  There exists $0<\tilde{\ep}_0\le\ep_0$, $0<\tilde{r}_0\le r_0$, $\del_-,\del_+,\ell,R\in\R^+$, $N\in\N^+$, $\tcb{x_i}_{i=1}^N\subset\Sigma$ such that the following hold for all $i\in\tcb{1,\dots,N}$, $\psi\in C^1\tp{\Bar{U}}$ with $\tnorm{\psi}_{C^1\tp{\Bar{U}}}\le\tp{M+1}\tilde{r}_0$.
  \begin{enumerate}
      \item We have $\del_-<\del_+$ and, for all $0<\ep\le\tilde{\ep}_0$, the inclusions
      \begin{equation}\label{p_0}
          B(x_i,2\del_+)\subset U\text{ and }\tcb{x\in U\;:\;\tabs{\Psi(x) + \psi(x)}\le\ep}\subset\bigcup_{i=1}^NB(x_i,\del_-).
      \end{equation}
      \item There exists a map $G^\psi_i\in C^1\tp{\Bar{\ell Q};\R^2}$ which is a diffeomorphism onto its image, satisfies the inclusions
      \begin{equation}\label{p_1}
          B(x_i,2\del_-)\subset G^\psi_i\tp{\Bar{\ell Q}}\subset B(x_i,\del_+),
      \end{equation}
      and for all $y\in\Bar{\ell Q}$ we have the identities 
      \begin{equation}\label{p_2}
          \tp{\Psi + \psi}\circ G^\psi_i(y) = y_2\text{ and }\tabs{\pd_1 G^\psi_i\tp{y}} = \tabs{\grad\tp{\Psi + \psi}\circ G^\psi_i\tp{y}}\det\grad G^\psi_i\tp{y}
      \end{equation}
      along with the estimate
      \begin{equation}\label{p_3}
          \tabs{\grad G^\psi_i\tp{y}} + \tabs{\grad G^\psi_i\tp{y}^{-1}}\le R.
      \end{equation}
      \item The mapping
      \begin{equation}\label{p_4}
          C^1\tp{\Bar{U}}\supset\Bar{B(0,\tilde{r}_0)}\ni\psi\mapsto G^\psi_i\in C^1\tp{\Bar{\ell Q};\R^2}
      \end{equation}
      is continuous moreover, we have the estimate
      \begin{equation}\label{p_5}
          \tabs{\grad G^\psi_i\tp{y} - \grad G^0_i\tp{y}}\le R\tnorm{\psi}_{C^1\tp{\Bar{U}}}.
      \end{equation}
  \end{enumerate}
\end{lemC}

Lemma~\ref{lem on uniform local level set coordinates}'s uniform local level set coordinates permit us to obtain sharp density estimates on level sets.

\begin{coroC}[Refined estimates on level sets]\label{coro on refined estimates on level sets}
    Let $\del_-$, $\tilde{\ep}_0$, and $\tilde{r}_0$ be the parameters from Lemma~\ref{lem on uniform local level set coordinates}. There exists a constant $C\in\R^+$ such that for all $|t|\le\tilde{\ep}_0$, $0<r\le\del_-$, and $\psi\in C^1_{\tp{m}}\tp{\Bar{U}}$ with $\tnorm{\psi}_{C^1\tp{\Bar{U}}}\le\tp{M+1}\tilde{r}_0$ obeying the following properties, in which we use $\mathcal{H}^1$ to denote the one-dimensional Hausdorff measure.
    \begin{enumerate}
        \item We have the estimates
        \begin{equation}\label{hausdorff_1}
            \sup_{x\in\Bar{U}}\mathcal{H}^1\tp{U\cap\tcb{\Psi + \psi = t}\cap B(x,r)}\le Cr\text{ and }\mathcal{H}^1\tp{U\cap\tcb{\Psi + \psi = t}}\le C.
        \end{equation}
        \item We have the estimate 
        \begin{equation}\label{hausdorff_2}
            \sup_{x\in\Bar{U}}\int_{U\cap\tcb{\Psi + \psi = t}\cap B(x,r)}\tabs{\log\tabs{x - y}}\;\m{d}\mathcal{H}^1\tp{y}\le Cr\tp{1 + \log\tp{1/r}}.
        \end{equation}
        \item We have the estimates
        \begin{equation}\label{hausdorff_3}
            \sup_{x\in\Bar{U}}\int_{\tp{U\setminus B(x,r)}\cap\tcb{\Psi + \psi = t}}\f{1}{\tabs{x - y}}\;\m{d}\mathcal{H}^1\tp{y}\le C\tp{1 + \log\tp{1/r}}.
        \end{equation}
        \item We have the estimates
        \begin{equation}\label{the thing we want to estimate the level set continuty}
            \tabs{\mathcal{H}^1\tp{U\cap\tcb{\Psi + \psi = t}} - \mathcal{H}^1\tp{U\cap\tcb{\Psi = 0}}}\le C\tp{t + \tnorm{\psi}_{C^1\tp{\Bar{U}}}}.
        \end{equation}
    \end{enumerate}
\end{coroC}
\begin{proof}
     We begin by proving the first item. Let $x\in\Bar{U}$. From the first item of Lemma~\ref{lem on uniform local level set coordinates} we have
        \begin{equation}\label{reduction__0}
            \mathcal{H}^1\tp{U\cap\tcb{\Psi + \psi = t}\cap B(x,r)}\le\sum_{i=1}^N\mathcal{H}^1\tp{B(x_i,\del_-)\cap\tcb{\Psi + \psi = t}\cap B(x,r)}
        \end{equation}
        so it suffices to estimate the $i^{\m{th}}$-term in the sum above. For this we use the local parametrization of the $t$-level set determined by $G^\psi_i\tp{\cdot,t}$ where the latter is from the collection of coordinate functions produced by Lemma~\ref{lem on uniform local level set coordinates}. We get
        \begin{equation}\label{reduction__1}
            \mathcal{H}^1\tp{B(x_i,\del_-)\cap\tcb{\Psi + \psi = t}\cap B(x,r)}\le\int_{-\ell}^\ell\mathds{1}_{G^\psi_i\tp{\cdot,t}\in B(x,r)}\tp{y_1} \tabs{\pd_1G^\psi_i\tp{y_1,t}}\;\m{d}y_1.
        \end{equation}
        From estimate~\eqref{p_3} we get $\tabs{\pd_1 G^\psi_i\tp{y_1,t}}\le R$. On the other hand, by inspection of identity~\eqref{the diffeo defn} from the proof of Lemma~\ref{lem on uniform local level set coordinates}, we deduce
        \begin{equation}
            \forall\;y_1,\tilde{y}_1\in[-\ell,\ell],\;\tabs{G^\psi_i\tp{y_1,t} - G^\psi_i\tp{\tilde{y}_1,t}}\ge\tabs{y_1 - \tilde{y}_1}
        \end{equation}
        and hence, by bounding the one-dimensional Hausdorff measure above by the length,
        \begin{equation}\label{reduction__2}
            \mathcal{H}^1\tp{\tcb{y_1\in[-\ell,\ell]\;:\;G^\psi_i\tp{y_1,t}\in B(x,r)}}\le2r
        \end{equation}
        Synthesizing the estimates~\eqref{reduction__0}, \eqref{reduction__1}, and~\eqref{reduction__2} yields
        \begin{equation}
            \mathcal{H}^1\tp{U\cap\tcb{\Psi + \psi = t}\cap B(x,r)}\le\tp{2RM}r
        \end{equation}
        which is the bound claimed on the left hand side of~\eqref{hausdorff_1}. The right hand bound follows from the left by setting $r = \del_-$, and summing over $x\in\tcb{x_1,\dots,x_N}$.

        We turn our attention to the second item, so again let $x\in\Bar{U}$. We break up the region of integration according to dyadic annuli:
        \begin{equation}\label{___)))___}
            \int_{U\cap\tcb{\Psi + \psi = t}\cap B(x,r)}\tabs{\log\tabs{x - y}}\;\m{d}\mathcal{H}^1\tp{y}=\sum_{n=0}^\infty\int_{U\cap\tcb{\Psi + \psi = t}\cap A(x,r/2^n)}\tabs{\log\tabs{x - y}}\;\m{d}\mathcal{H}^1\tp{y}
        \end{equation}
        with $A\tp{x,r/2^n} = B(x,r/2^n)\setminus B(x,r/2^{n+1})$. On the $n^{\m{th}}$ annulus we have the (crude) inequality $\tabs{\log|x - y|}\le2\tp{n + 1}\log\tp{1/r}$ and so~\eqref{___)))___} lends to us the bound:
        \begin{equation}
            \int_{U\cap\tcb{\Psi + \psi = t}\cap B(x,r)}\tabs{\log\tabs{x - y}}\;\m{d}\mathcal{H}^1\tp{y}\le2\log\tp{1/r}\sum_{n=0}^\infty\tp{n+1}\mathcal{H}^1\tp{U\cap\tcb{\Psi + \psi = t}\cap B(x,r/2^n)}.
        \end{equation}
        The right hand side sum's $n^{\m{th}}$ Hausdorff measure can then be estimated by $Cr/2^n$, thanks to the first item. Thus we get the desired estimate~\eqref{hausdorff_2} with a constant $2C\sum_{n=0}^\infty\tp{n + 1}2^{-n}<\infty$.

        We prove the third item via a similar strategy as above. For $x\in\Bar{U}$ we dyadically decompose in a similar fashion, thus obtaining:
        \begin{equation}
            \int_{\tp{U\setminus B(x,r)}\cap\tcb{\Psi + \psi = t}}\f{1}{\tabs{x - y}}\;\m{d}\mathcal{H}^1\tp{y} \le \f{1}{\del_-}\mathcal{H}^1\tp{\tcb{\Psi + \psi = t}} + \sum_{n=0}^{N}\f{2^{n+1}}{\del_-}\mathcal{H}^1\tp{\tcb{\Psi + \psi = t}\cap B(x,\del_-/2^n)}.
        \end{equation}
        where $N\in\N^+$ is the smallest positive integer for which $\del_-/2^N<r$. We then use the first item again to bound the $n^{\m{th}}$ term's Hausdorff measure by $C\del_-/2^n$. This gives us
        \begin{equation}
            \int_{\tp{U\setminus B(x,r)}\cap\tcb{\Psi + \psi = t}}\f{1}{\tabs{x - y}}\;\m{d}\mathcal{H}^1\tp{y}\le\f{C}{\del_-} + 2C\tp{N+1}.
        \end{equation}
        We conclude by noting that $N\le 1+\log\tp{\del_-/r}$.

        Finally, we consider the fourth item. We begin by using the triangle inequality to write
        \begin{equation}
        \tabs{\mathcal{H}^1\tp{\tcb{\Psi + \psi = t}} - \mathcal{H}^{1}\tp{\tcb{\Psi = 0}}}\le\bf{I}_1 + \bf{I}_2
        \end{equation}
        with
    \begin{equation}
        \bf{I}_1 = \tabs{\mathcal{H}^1\tp{\tcb{\Psi + \psi = t}} - \mathcal{H}^1\tp{\tcb{\Psi = t}}}\text{ and }\bf{I}_2 = \tabs{\mathcal{H}^1\tp{\tcb{\Psi = t}} - \mathcal{H}^1\tp{\tcb{\Psi = 0}}}.
    \end{equation}
    Estimating $\bf{I}_2$ is simpler. Consider the case that $t>0$ (we omit the proof of the case $t<0$, as it is similar) We deduce from Definition~\ref{defn of admissible patches} that $\grad\Psi/\tabs{\grad\Psi}\in C^{0,1}\tp{\Bar{U}}$ and so we can use the divergence theorem to derive:
    \begin{equation}\label{_aa__}
        \bf{I}_2 = \babs{\int_{U\cap\tcb{0<\Psi<t}}\grad\cdot\tp{\grad\Psi/\tabs{\grad\Psi}}}\le C\mathcal{H}^2\tp{U\cap\tcb{0<\Psi<t}}.
    \end{equation}
    Then, the coarea formula (see, e.g., Theorem 3.2.11 in Federer~\cite{MR257325} or Theorem 3.13 in Evans and Gariepy~\cite{MR3409135}) in conjunction with the first item in turn gives us
    \begin{equation}\label{_bb__}
        \mathcal{H}^2\tp{U\cap\tcb{0<\Psi<t}} = \int_0^t\int_{\tcb{\Psi = \tau}}\tabs{\grad\Psi(y)}^{-1}\;\m{d}\mathcal{H}^1\tp{y}\;\m{d}\tau\le C t.
    \end{equation}
    Synthesizing~\eqref{_aa__} and~\eqref{_bb__} shows $\bf{I}_2\le Ct$.
    
    Estimating $\bf{I}_1$ is more involved. Let $\tcb{\chi_i}_{i=1}^N\subset C^\infty_{\m{c}}\tp{U}$ be any fixed sequence of functions that satisfy
    \begin{equation}\label{introduction of the POU}
        \m{supp}\chi_i\subset B(x_i,2\del_-)\text{ for every }i\in\tcb{1,\dots,N}\text{ and }\sum_{i=1}^N\chi_i = 1\text{ on the set }\bigcup_{i=1}^NB(x_i,\del_-).
    \end{equation} 
    We use the uniform local level set coordinates constructed by Lemma~\ref{lem on uniform local level set coordinates} to locally parametrize the $t$-level sets; in doing so we get
    \begin{equation}
        \mathcal{H}^1\tp{\tcb{\Psi + \psi = t}} - \mathcal{H}^1\tp{\tcb{\Psi = t}}=\sum_{i=1}^N\int_{0}^\ell\tp{\chi\tp{G^\psi_i\tp{y_1,t}}\tabs{\pd_1G^\psi_i\tp{y_1,t}} - \chi_i\tp{G^0_i\tp{y_1,t}}\tabs{\pd_1 G^0_i\tp{y_1,t}}}\;\m{d}y_1.
    \end{equation}
    Then, using the special continuity estimate~\eqref{p_5} from Lemma~\ref{lem on uniform local level set coordinates}, we arrive at the bound
    \begin{equation}
        \bf{I}_1\le C\tnorm{\psi}_{C^1\tp{\Bar{U}}}.
    \end{equation}
    
    Synthesizing the estimates on $\bf{I}_1$ and $\bf{I}_2$ completes the proof of estimate~\eqref{the thing we want to estimate the level set continuty}.
\end{proof}

The level set coordinate functions constructed by Lemma~\ref{lem on uniform local level set coordinates} combine with the density estimates of the first three items of Corollary~\ref{coro on refined estimates on level sets} to allow us to make the subsequent important boundedness and regularity estimates.

\begin{propC}[Uniform integral estimates for the Newtonian potential over thin domains]\label{prop on uniform intrgrals estimates for the Newtonian potential}
    Let $\tilde{\ep}_0$ and $\tilde{r}_0$ be the positive values granted by Lemma~\ref{lem on uniform local level set coordinates}. There exists $C\in\R^+$ such that for all $\psi\in C^1_{\tp{m}}\tp{\Bar{U}}$ with $\tnorm{\psi}_{C^1\tp{\Bar{U}}}\le(M+1)\tilde{r}_0$ and all $0<\ep\le\tilde{\ep}_0$ we have the estimates
    \begin{equation}\label{first log estimate}
        \sup_{x\in\Bar{U}}\f{1}{\ep}\int_{\tcb{\tabs{\Psi + \psi}\le\ep}}\tp{\tabs{\log\tabs{x - y}} + 1}\;\m{d}y\le C
    \end{equation}
    and
    \begin{equation}\label{second log estimate}
        \sup_{\substack{x,z\in\Bar{U}\\x\neq z}}\f{1}{\ep}\int_{\tcb{\tabs{\Psi + \psi}\le\ep}}\bp{\f{\tabs{\log\tabs{x - y} - \log\tabs{z - y}}}{\tabs{x - z}\tp{1 + \tabs{\log\tabs{x - z}}}} + 1}\;\m{d}y\le C.
    \end{equation}
\end{propC}
\begin{proof}
    We claim first, with $\mathcal{H}^2$ denoting the two-dimensional Hausdorff measure, that the area estimate
    \begin{equation}\label{The area estimates}
        \mathcal{H}^2\tp{\tcb{\tabs{\Psi + \psi}\le\ep}}\le C\ep
    \end{equation}
    holds. By the coarea formula (e.g. Theorem 3.2.11 in~\cite{MR257325} or Theorem 3.13 in~\cite{MR3409135}) we have 
    \begin{equation}\label{ive got you under my skin}
        \mathcal{H}^2\tp{\tcb{\tabs{\Psi + \psi}\le\ep}}\le2\ep\cdot\sup_{|t|\le\ep}\int_{\tcb{\Psi + \psi = t}}\tabs{\grad\tp{\Psi + \psi}\tp{y}}^{-1}\;\m{d}\mathcal{H}^1\tp{y}
    \end{equation}
    The definition of $\tilde{r_0}$ ensures us that $\min_{\Bar{U}}\tabs{\grad\tp{\Psi + \psi}}\ge\min_{\Bar{U}}\tabs{\grad\Psi}/2$ and so it follows from~\eqref{ive got you under my skin} that
    \begin{equation}\label{fly me to the moon}
        \mathcal{H}^2\tp{\tcb{\tabs{\Psi + \psi}\le\ep}}\le C\ep\cdot\sup_{|t|\le\ep}\mathcal{H}^1\tp{\tcb{\Psi + \psi = t}}.
    \end{equation}
    The bound~\eqref{The area estimates} now follows by combining~\eqref{fly me to the moon} with the first item of Corollary~\ref{coro on refined estimates on level sets}.

    We continue with the proof of estimate~\eqref{first log estimate}. The coarea formula again lends us the initial estimate for any $x\in\Bar{U}$
    \begin{equation}
        \f{1}{\ep}\int_{\tcb{\tabs{\Psi + \psi}\le\ep}}\tabs{\log\tabs{x - y}}\;\m{d}y\le2\sup_{|t|\le\ep}\int_{\tcb{\Psi + \psi = t}}\f{\tabs{\log\tabs{x - y}}}{\tabs{\grad\tp{\Psi + \psi}\tp{y}}}\;\m{d}\mathcal{H}^1\tp{y}.
    \end{equation}
    We can, for $|t|\le\ep$ fixed, split $\tcb{\Psi + \psi = t}$ relative to the ball $B(x,\del_-)$ and use the first item of Corollary~\ref{coro on refined estimates on level sets} to bound
    \begin{equation}\label{will be referenced later}
        \int_{\tcb{\Psi + \psi = t}}\tabs{\log\tabs{x - y}}\;\m{d}\mathcal{H}^1\tp{y}\le C\tabs{\log\del_-} + \int_{\tcb{\Psi + \psi = t}\cap B(x,\del_-)}\tabs{\log\tabs{x - y}}\;\m{d}\mathcal{H}^1\tp{y}.
    \end{equation}
    The final term above is at most $C\del_-\tabs{\log\del_-}$ (by the second item of Corollary~\ref{coro on refined estimates on level sets}) and so we conclude that estimate~\eqref{first log estimate} holds.

    We turn our attention to the proof of~\eqref{second log estimate}. Let $x,z\in\Bar{U}$ with $x\neq z$. If $\tabs{x - z}\ge\del_-/2$, then we do not exploit the difference between the logarithms. The estimate
    \begin{equation}
        \f{1}{\ep}\int_{\tcb{\tabs{\Psi + \psi}\le\ep}}\f{\tabs{\log\tabs{x - y} - \log\tabs{z - y}}}{\tabs{x - z}\tp{1 + \tabs{\log\tabs{x - z}}}}\m{d}y\le\f{4C}{\del_-\log\tp{2/\del_-}}
    \end{equation}
    simply follows from~\eqref{first log estimate}. Thus we need only consider the case that $0<\tabs{x - z}<\del_-/2$.

    By the coarea formula again, we have
    \begin{equation}
        \f{1}{\ep}\int_{\tcb{\tabs{\Psi + \psi}\le\ep}}\f{\tabs{\log\tabs{x - y} - \log\tabs{z - y}}}{\tabs{x - z}\tp{1 + \tabs{\log\tabs{x - z}}}}\m{d}y\le\f{2\sup_{|t|\le\ep}}{\min_{\Bar{U}}\tabs{\grad\Psi}}\int_{\tcb{\Psi + \psi = t}}\f{\tabs{\log\tabs{x - y} - \log\tabs{z - y}}}{\tabs{x - z}\tp{1 + \tabs{\log\tabs{x - z}}}}\;\m{d}\mathcal{H}^1\tp{y}.
    \end{equation}
    Fix $|t|\le\ep$. We make the region splitting
    \begin{equation}\label{region splitting}
        \tcb{\Psi + \psi = t} = \tsb{\tcb{\Psi + \psi = t}\cap B((x+z)/2,2|x - z|)} \cup \tsb{\tcb{\Psi + \psi = t}\setminus B((x+z)/2,2|x - z|)}.
    \end{equation}
    In the first case that $y\in B((x+z)/2,2|x - z|)$, we again do not exploit the difference and instead use the second item of Corollary~\ref{coro on refined estimates on level sets}, giving us
    \begin{equation}
        \int_{\tcb{\Psi + \psi = t}\cap B((x+z)/2,2|x - z|)}\tabs{\log\tabs{x - y} - \log\tabs{z - y}}\;\m{d}\mathcal{H}^1\tp{y}\le 4C|x - z|\tp{1 + \tabs{\log\tabs{2|x - z|}}}.
    \end{equation}
    
    On the other hand, in the second case of~\eqref{region splitting} we have the fundamental theorem of calculus estimate:
    \begin{equation}\label{the FTC estimate}
        \tabs{\log\tabs{x - y} - \log\tabs{z - y}}\le\int_{0}^1\f{\tabs{x - z}}{\tabs{\tau x + (1-\tau)z - y}}\;\m{d}\tau\le\f{4\tabs{x - z}}{3\tabs{y - (x+z)/2}}
    \end{equation}
    for $y\in\tcb{\Psi + \psi = t}\setminus B((x + z)/2,2|x - z|)$. We combine~\eqref{the FTC estimate} with the third item of Corollary~\ref{coro on refined estimates on level sets} to deduce
    \begin{equation}
        \int_{\tcb{\Psi + \psi = t}\setminus B((x + z)/2,2|x - z|)}\tabs{\log\tabs{x - y} - \log\tabs{z - y}}\;\m{d}\mathcal{H}^1\tp{y}\le\f{4C}{3}|x - z|\tp{1 + \tabs{\log\tabs{2|x - z|}}}.
    \end{equation}
    Synthesizing the above casework gives the desired bound~\eqref{second log estimate}.
\end{proof}

The level set coordinates of Lemma~\ref{lem on uniform local level set coordinates} also find an application in the identification of certain limits.

\begin{propC}[Limit identification]\label{prop on limit identification}
    Assume that $g\in C^0\tp{\Bar{U}}$, $\tcb{\ep_n}_{n\in\N}\subset(0,\tilde{\ep}_0]$, $\psi,\tcb{\psi_n}_{n\in\N}\subset\Bar{B(0,(M+1)\tilde{r}_0)}\subset C_{\tp{m}}^1\tp{\Bar{U}}$, and $\phi,\tcb{\phi_n}_{n\in\N}\subset C_{\tp{m}}^0\tp{\Bar{U}}$ satisfy
    \begin{equation}\label{convergence hypotheses}
        \ep_n\to0,\;\tnorm{\psi_n - \psi}_{C^1\tp{\Bar{U}}}\to0,\;\text{and}\;\tnorm{\phi_n-\phi}_{C^0\tp{\Bar{U}}}\to 0\text{ as }n\to\infty.
    \end{equation}
    Then
    \begin{equation}\label{conv_1}
        -\int_{U}g(x)\tsb{\tp{\Gamma_{\ep_n}^\Psi}'\tp{\psi_n}}\tp{x}\phi_n(x)\;\m{d}x \to \int_{U\cap \tcb{\Psi + \psi = 0}}\f{g(y)\phi(y)}{\tabs{\grad\tp{\Psi + \psi}\tp{y}}}\;\m{d}\mathcal{H}^1\tp{y}\text{ as }n\to\infty
    \end{equation}
    and 
    \begin{equation}\label{conv_2}
        \int_{U}g(x)\tsb{\Gamma_{\ep_n}^\Psi\tp{\psi_n}}\tp{x}\phi_n(x)\;\m{d}x\to\int_{U\cap\tcb{\Psi + \psi<0}}g(x)\phi(x)\;\m{d}x\text{ and }n\to\infty.
    \end{equation}
\end{propC}
\begin{proof}
    We only prove~\eqref{conv_1} in the special case that $\psi = 0$. The remaining claims in the general case follow by similar arguments. We shall use the uniform local level set coordinates constructed by Lemma~\ref{lem on uniform local level set coordinates} and the partition of unity~\eqref{introduction of the POU}.
    
    Lemmas~\ref{lem on regularized hevi} and~\ref{lem on uniform local level set coordinates} imply that $\tp{\Gamma^\Psi_{\ep_n}}'\tp{\psi_n}$ is supported in $\bigcup_{i=1}^NB(x_i,\del_-)$ for every $n\in\N$. In turn we can decompose into a sum of $N$ terms and in the $i^{\m{th}}$-term make a change of variables with $G^{\psi_n}_i$:
    \begin{multline}
        \int_{U}g(x)\tsb{\tp{\Gamma_{\ep_n}^\Psi}'\tp{\psi_n}}\tp{x}\phi_n(x)\;\m{d}x = \sum_{i=1}^N I^i_n,\\ \text{where } I^i_n=\f{1}{\ep_n}\int_{-\ell}^\ell\int_{-\ep_n}^{\ep_n}\chi_i\tp{G^{\psi_n}_i\tp{s}}g(G^{\psi_n}_i\tp{s})\gamma'\tp{-s_2/\ep_n}\phi_n(G^{\psi_n}_i\tp{s})\det\grad G^{\psi_n}_i\tp{s}\;\m{d}s_2\;\m{d}s_1.
    \end{multline}
    Now, it is suggested to us by the structure above that we ought to make a further change of variables by replacing $s_2/\ep_n$ by $s_2$. Doing so gives:
    \begin{equation}
        I^i_n = \int_{-\ell}^\ell\int_{-1}^1\gam'(-s_2)\chi_i\tp{G^{\psi_n}_i\tp{s_1,\ep_ns_2}}g(G^{\psi_n}_i\tp{s_1,\ep_n s_2})\phi_n(G^{\psi_n}_i\tp{s_1,\ep_ns_2})\det\grad G^{\psi_n}_i\tp{s_1,\ep_ns_2}\;\m{d}s_2\;\m{d}s_1.
    \end{equation}
    Next, we use the dominated convergence theorem in conjunction with the hypotheses~\eqref{convergence hypotheses} and the third item of Lemma~\ref{lem on uniform local level set coordinates} to deduce that the sequence $\tcb{I^i_n}_{n\in\N}\subset\R$ converges to the limit $I^i\in\R$ given by
    \begin{equation}
        \lim_{n\to\infty}I^i_n = I^i = \int_{-\ell}^{\ell}\chi_i\tp{G^0_i(s_1,0)}g(G^0_i(s_1,0))\phi(G^0_i(s_1,0))\det\grad G^0_i(s_1,0)\;\m{d}s_1.
    \end{equation}
    We next use the right hand equality of~\eqref{p_2} to introduce a line element to the above integral and rewrite it in the coordinate-free fashion:
    \begin{equation}
        I^i = \int_{\Sigma}\chi_i(y)\f{ g(y) \phi(y)}{\tabs{\grad\Psi\tp{y}}}\;\m{d}\mathcal{H}^1\tp{y}.
    \end{equation}
    The proof is then completed upon summing over $i\in\tcb{1,\dots,N}$.
\end{proof}

This subsection's final result provides essential logarithmic-Lipschitz estimates on the operators $\pmb{K}_{\ep,\rho,\Qoppa}$ that we recall are defined in equation~\eqref{the nonlienar operators}.

\begin{propC}[Preliminary mapping estimates, II]\label{prop on preliminary mapping estimates II}
    With $\tilde{\ep}_0$ and $\tilde{r}_0$ as in Proposition~\ref{prop on uniform intrgrals estimates for the Newtonian potential} and $\Qoppa_0$ as in Lemma~\ref{lem on regularized hevi}, there exists $C\in\R^+$ with the property that for all $0<\ep\le\tilde{\ep}_0$, $0\le\rho\le\tilde{r}_0$, $0\le\Qoppa\le\Qoppa_0$, and $\psi,\tilde{\psi}\in\Bar{B(0,\tilde{r}_0)}\subset C^1_{\tp{m}}\tp{\Bar{U}}$ we have the continuity estimate
    \begin{equation}\label{holder bound K dude}
        \tnorm{\pmb{K}_{\ep,\rho,\Qoppa}(\psi) - \pmb{K}_{\ep,\rho,\Qoppa}\tp{\tilde{\psi}}}_{C^1\tp{\Bar{U}}}\le C\tnorm{\psi - \tilde{\psi}}_{C^0\tp{\Bar{U}}}\tp{1 + \tabs{\log\tabs{\tnorm{\psi - \tilde{\psi}}_{C^0\tp{\Bar{U}}}}}}.
    \end{equation}
\end{propC}
\begin{proof}
    The first part of the proof is to establish two auxiliary claims. The first claim is: there exists $C\in\R^+$ such that for all $0<\ep\le\tilde{\ep}_0$, $0\le\rho\le\tilde{r}_0$, $0\le\Qoppa\le\Qoppa_0$, and $\psi\in\Bar{B(0,\tilde{r}_0)}\subset C^1_{\tp{m}}\tp{\Bar{U}}$ we have the estimate
    \begin{equation}\label{ee_0}
        \tnorm{\pmb{K}_{\ep,\rho,\Qoppa}\tp{\psi}}_{\m{LL}^1\tp{\Bar{U}}}\le C.
    \end{equation}
    To prove the claim~\eqref{ee_0}, we begin by noting that the pointwise estimates
    \begin{equation}\label{ee_1}
        \tabs{\tsb{\pmb{K}_{\ep,\rho,\Qoppa}\tp{\psi}}\tp{x}}\le C\int_{U}\tp{\tabs{\log\tabs{x - y}} + 1}\;\m{d}y,\;\tabs{\grad\tsb{\pmb{K}_{\ep,\rho,\Qoppa}\tp{\psi}}\tp{x}}\le C\int_{U}\bp{\f{1}{\tabs{x - y}} + 1}\;\m{d}y
    \end{equation}
    hold for $x\in U$. If $\tilde{x}\in U\setminus\tcb{x}$ we also have 
    \begin{equation}\label{ee_2}
        \tabs{\grad\tsb{\pmb{K}_{\ep,\rho,\Qoppa}\tp{\psi}}\tp{x} - \grad\tsb{\pmb{K}_{\ep,\rho,\Qoppa}\tp{\psi}}\tp{\tilde{x}}}\le C\bp{\int_{U}\babs{\f{x - y}{\tabs{x - y}^2} - \f{\tilde{x} - y}{\tabs{\tilde{x} - y}^2}}\;\m{d}y + \tabs{x - \tilde{x}}}.
    \end{equation}
    
    Estimates~\eqref{ee_1} evidently establish the variant of~\eqref{ee_0} obtained by replacing $\tnorm{\cdot}_{\m{LL}^1\tp{\Bar{U}}}$ by $\tnorm{\cdot}_{C^1\tp{\Bar{U}}}$. To obtain the remaining $\m{LL}\tp{\Bar{U}}$-norm on the derivative as a consequence of~\eqref{ee_2}, further analysis is required. We shall split the integral on the right of~\eqref{ee_2} over the following two regions:
    \begin{equation}
        E_{x,\tilde{x}} = \tcb{y\in U\;:\;|y - (x + \tilde{x})/2|\ge 2\tabs{x - \tilde{x}}},\;F_{x,\tilde{x}} = \tcb{y\in U\;:\;\tabs{y - \tp{x + \tilde{x}}/2}<2|x - \tilde{x}|}.
    \end{equation}
    If $y\in E_{x,\tilde{x}}$ then $2\min_{0\le\tau\le 1}\tabs{y - \tau x - (1-\tau)\tilde{x}}\ge\tabs{y - (x + \tilde{x})/2} + \tabs{x - \tilde{x}}$ and so the fundamental theorem of calculus leads to the estimate
    \begin{equation}
        \int_{E_{x,\tilde{x}}}\babs{\f{x - y}{\tabs{x - y}^2} - \f{\tilde{x} - y}{\tabs{\tilde{x} - y}^2}}\;\m{d}y\le C\tabs{x - \tilde{x}}\int_{E_{x,\tilde{x}}}\f{1}{\tabs{x - \tilde{x}}^2 + \tabs{y - (x + \tilde{x})/2}^2}\;\m{d}y\le C|x - \tilde{x}|\tp{1 + \tabs{\log\tabs{x - \tilde{x}}}}.
    \end{equation}
    where $C$ is a constant depending only on $U$. On the other hand, for the region $F_{x,\tilde{x}}$ there is no need to exploit the difference, instead we simply bound
    \begin{equation}
        \int_{F_{x,\tilde{x}}}\babs{\f{x - y}{\tabs{x - y}^2} - \f{\tilde{x} - y}{\tabs{\tilde{x} - y}^2}}\;\m{d}y\le 2\int_{B(0,4|x - \tilde{x}|)}\f{1}{\tabs{y}}\;\m{d}y\le C\tabs{x - \tilde{x}}.
    \end{equation}
    Synthesizing the above gives the first auxiliary claim~\eqref{ee_0}.

    Let us move on to the second auxiliary claim: there exists $C\in\R^+$ such that for all $0<\ep\le\tilde{\ep}_0$, $0\le\rho\le\tilde{r}_0$, $0\le\Qoppa\le\Qoppa_0$, and $\psi,\tilde{\psi}\in\Bar{B(0,\tilde{r}_0)}\subset C^1_{\tp{m}}\tp{\Bar{U}}$ we have the estimate
    \begin{equation}\label{aux_claim_taxed}
        \tnorm{\pmb{K}_{\ep,\rho,\Qoppa}\tp{\psi} - \pmb{K}_{\ep,\rho,\Qoppa}\tp{\tilde{\psi}}}_{\m{LL}\tp{\Bar{U}}}\le C\tnorm{\psi - \tilde{\psi}}_{C^0\tp{\Bar{U}}}.
    \end{equation}
    Thanks to the fundamental theorem of calculus, we have the following pointwise expression for $x\in U$:
    \begin{multline}\label{alternative pointwise expression}
        \tsb{\pmb{K}_{\ep,\rho,\Qoppa}\tp{\psi} - \pmb{K}_{\ep,\rho,\Qoppa}\tp{\tilde{\psi}}}\tp{x} \\= \sum_{k=-M}^M\f{\sig_k}{\ep}\int_0^1\int_{U}\bf{N}_{\Qoppa}\tp{x,y}\gamma'\bp{-\f{\Psi(y) + k\rho + \tau\psi(y) + (1-\tau)\tilde{\psi}(y)}{\ep}}\tp{\psi - \tilde{\psi}}\tp{y}\;\m{d}y\;\m{d}\tau.
    \end{multline}
    Therefore, if we define the regions indexed by $k\in\tcb{-M,\dots,M}$
    \begin{equation}
        R_{k}(\tau) = \tcb{y\in U\;:\;\tabs{\Psi(y) + \tau\psi(y) + \tp{1 - \tau}\tilde{\psi}\tp{y} + k\rho}\le\ep},
    \end{equation}
    we see that~\eqref{alternative pointwise expression} lends us the estimates
    \begin{equation}\label{l_1}
        \tabs{\tsb{\pmb{K}_{\ep,\rho,\Qoppa}\tp{\psi} - \pmb{K}_{\ep,\rho,\Qoppa}\tp{\tilde{\psi}}}\tp{x}}\le\f{C}{\ep}\tnorm{\psi - \tilde{\psi}}_{C^0\tp{\Bar{U}}}\sum_{k=-M}^M\int_{0}^1\int_{U}\mathds{1}_{R_{k}\tp{\tau}}\tp{s}\tp{\tabs{\log\tabs{x - y}} + 1}\;\m{d}y
    \end{equation}
    and
    \begin{multline}\label{l_2}
        \f{\tabs{\tsb{\pmb{K}_{\ep,\rho,\Qoppa}\tp{\psi} - \pmb{K}_{\ep,\rho,\Qoppa}\tp{\tilde{\psi}}}(x) - \tsb{\pmb{K}_{\ep,\rho,\Qoppa}\tp{\psi} - \pmb{K}_{\ep,\rho,\Qoppa}\tp{\tilde{\psi}}}\tp{\tilde{x}}}}{\tabs{x - \tilde{x}}\tp{1 + \tabs{\log\tabs{x - \tilde{x}}}}}\\
        \le\f{C}{\ep}\tnorm{\psi - \tilde{\psi}}_{C^0\tp{\Bar{U}}}\sum_{k=-M}^M\int_0^1\int_U\mathds{1}_{R_{k}\tp{\tau}}\tp{s}\f{\tp{\tabs{\log\tabs{x - y} - \log\tabs{\tilde{x} - y}} + |x -\tilde{x}|}}{\tabs{x - \tilde{x}}\tp{1 + \tabs{\log\tabs{x - \tilde{x}}}}}\;\m{d}y\;\m{d}\tau
    \end{multline}
    for $x,\tilde{x}\in U$ with $x\neq\tilde{x}$. Proposition~\ref{prop on uniform intrgrals estimates for the Newtonian potential} can now be applied to estimate the right hand sides of~\eqref{l_1} and~\eqref{l_2}; the result is that the second auxiliary claim~\eqref{aux_claim_taxed} holds. 

    The final thrust of the proof is to combine the above auxiliary claims to deduce~\eqref{holder bound K dude}. First we note that thanks to the interpolation estimates of Appendix~\ref{appendix on an interpolation inequality}, we have
    \begin{multline}
        \tnorm{\pmb{K}_{\ep,\rho,\Qoppa}\tp{\psi} - \pmb{K}_{\ep,\rho,\Qoppa}\tp{\tilde{\psi}}}_{C^1\tp{\Bar{U}}}\\\le C \tnorm{\pmb{K}_{\ep,\rho,\Qoppa}\tp{\psi} - \pmb{K}_{\ep,\rho,\Qoppa}\tp{\tilde{\psi}}}_{\m{LL}\tp{\Bar{U}}}\bp{1 + \babs{\log\babs{\f{c\tnorm{\pmb{K}_{\ep,\rho,\Qoppa}(\psi) - \pmb{K}_{\ep,\rho,\Qoppa}\tp{\tilde{\psi}}}_{\m{LL}\tp{\Bar{U}}}}{\tnorm{\pmb{K}_{\ep,\rho,\Qoppa}\tp{\psi} - \pmb{K}_{\ep,\rho,\Qoppa}\tp{\tilde{\psi}}}_{\m{LL}^1\tp{\Bar{U}}}}}}}.
    \end{multline}
    Now for the $\tnorm{\cdot}_{\m{LL}\tp{\Bar{U}}}$ and $\tnorm{\cdot}_{\m{LL}^1\tp{\Bar{U}}}$ factors on the right hand side we employ the bounds from the second and first auxiliary claims, respectively.
\end{proof}

\subsection{Key consequences of nondegeneracy}\label{subsection on key consequences of nondegeneracy}

In this subsection the nondegeneracy condition~\eqref{the nondegeneracy condition} from Definition~\ref{defn of admissible patches} will now be exploited, unveiling the heart of our construction. Our first result combines nondegeneracy with certain compactness properties of the operators $\pmb{\kappa}$ from~\eqref{the linearized operator compact part} to establish uniformly coercive operator bounds in the supremum norm. We recall that the parameters $\tilde{\ep_0}$ and $\tilde{r}_0$ are given by Lemma~\ref{lem on uniform local level set coordinates} and $\Qoppa_0$ is from Lemma~\ref{lem on regularized hevi}.

\begin{propC}[Quantitative closed range estimates]\label{prop on quantitative closed range estimate}
    There exists $C\in\R^+$, $0<\ep_1\le\tilde{\ep}_0$, $0<\rho_1,r_1< \tilde{r}_0$, and $0<\Qoppa_1\le\Qoppa_0$ with the property that for all $0<\ep\le\ep_1$, $0\le\rho\le\rho_1$, $0\le\Qoppa\le\Qoppa_1$, $\psi^0,\psi^1\in\Bar{B(0,r_1)}\subset C^1_{\tp{m}}\tp{\Bar{U}}$, and $\phi\in C^0_{\tp{m}}\tp{\Bar{U}}$ we have the estimate
    \begin{equation}\label{closed range estimate in C0}
        \tnorm{\phi}_{C^0\tp{\Bar{U}}}\le C\bnorm{\phi + \bsb{\int_0^1\pmb{\kappa}_{\ep,\rho,\Qoppa}\tp{(1-\tau)\psi^0 + \tau\psi^1}\;\m{d}\tau}\phi}_{C^0\tp{\Bar{U}}}
    \end{equation}
\end{propC}
\begin{proof}
    We shall argue by way of contradiction, so let us suppose that the claim of the proposition is false. Then we would be able to find sequences 
    \begin{multline}
        \tcb{\ep_n}_{n\in\N}\subset(0,\tilde{\ep}_0],\;\tcb{\rho_n}_{n\in\N}\subset[0,\tilde{r}_0),\;\tcb{\Qoppa_n}_{n\in\N}\subset[0,\Qoppa_0],\\
        \tcb{\psi^0_n}_{n\in\N},\tcb{\psi^1_n}_{n\in\N}\subset\Bar{B(0,\tilde{r}_0)}\subset C^1_{\tp{m}}\tp{\Bar{U}},\;\tcb{\phi_n}_{n\in\N}\subset C_{\tp{m}}^0\tp{\Bar{U}}
    \end{multline}
     that satisfy $\tnorm{\phi_n}_{C^0\tp{\Bar{U}}}=1$ for all $n\in\N$ but
    \begin{equation}\label{absurd stuff}
        \lim_{n\to\infty}\tsb{|\ep_n| + |\rho_n| + |\Qoppa_n| + \tnorm{\psi^0_n}_{C^1(\Bar{U})} + \tnorm{\psi^1_n}_{C^1\tp{\Bar{U}}} + \tnorm{\phi_n - \varphi_n}_{C^0\tp{\Bar{U}}}} = 0
    \end{equation}
    where for $n\in\N$ we have defined
    \begin{equation}
        \varphi_n = - \bsb{\int_0^1\pmb{\kappa}_{\ep_n,\rho_n,\Qoppa_n}\tp{(1-\tau)\psi_n^0 + \tau\psi_n^1}\;\m{d}\tau}\phi_n\in C^0\tp{\Bar{U}}.
    \end{equation}

    We claim first that the sequence $\tcb{\varphi_n}_{n\in\N}\subset C_{\tp{m}}^0\tp{\Bar{U}}$ has compact closure. Thanks to the Arzel\`a-Ascoli theorem, it is sufficient to establish that the sequence $\tcb{\varphi_n}_{n\in\N}\subset\m{LL}\tp{\Bar{U}}$ is bounded. Due to $\tnorm{\phi_n}_{C^0\tp{\Bar{U}}} = 1$, support considerations, and Tonelli's theorem  we can directly make the estimates
    \begin{equation}\label{important_bound_1}
        \tabs{\varphi_n(x)} \le \f{C}{\ep_n}\sum_{k=-M}^M\int_0^1\int_{U}\mathds{1}_{E^k_n\tp{\tau}}\tp{y}\tp{\tabs{\log\tabs{x - y}} + 1}\;\m{d}y\;\m{d}\tau
    \end{equation}
    for any $x\in\Bar{U}$. For $\tau\in[0,1]$ and $k\in\tcb{-M,\dots,M}$ we have defined the sets
    \begin{equation}
        E_n^k\tp{\tau} = \tcb{x\in U\;:\;\tabs{\Psi(x) + (1-\tau)\psi^0_n(x) + \tau\psi^1_n(x) + k\rho_n}\le\ep_n}.
    \end{equation}
    Similarly, if $z\in \Bar{U}\setminus\tcb{x}$ we have the bounds
    \begin{equation}\label{important_bound_2}
        \f{\tabs{\varphi_n(x) - \varphi_n(z)}}{\tabs{x - z}\tp{1 + \tabs{\log\tabs{x - z}}}}\le\f{C}{\ep_n}\sum_{k=-M}^M\int_0^1\int_{U}\mathds{1}_{E^k_n\tp{\tau}}\tp{y}\bp{\f{\tabs{\log\tabs{x - y} - \log\tabs{z - y}}}{\tabs{x - z}\tp{1 + \tabs{\log\tabs{x - z}}}} + 1}\;\m{d}y\;\m{d}\tau.
    \end{equation}
    By using~\eqref{important_bound_1} and~\eqref{important_bound_2} in conjunction with Proposition~\ref{prop on uniform intrgrals estimates for the Newtonian potential} and equation~\eqref{absurd stuff} we conclude that $\tcb{\varphi_n}_{n\in\N}$ is indeed bounded in the space $\m{LL}\tp{\Bar{U}}$.

    Thanks to the aforementioned compactness, we are assured the existence of $\phi\in C_{\tp{m}}^0(\Bar{U})$ that we may suppose, after extraction of a subsequence and relabeling, satisfies $\varphi_n\to\phi$ as $n\to\infty$ in the $C^0\tp{\Bar{U}}$ norm. From the final limit of~\eqref{absurd stuff} we deduce further that $\phi_n\to\phi$ as $n\to\infty$ in the same norm. Since $\tnorm{\phi_n}_{C^0(\Bar{U})} = 1$ for all $n\in\N$ we must have $\tnorm{\phi}_{C^0\tp{\Bar{U}}}=1$ as well.

    We next claim that, in the sense of distributions on $U$, we have the convergence of $\varphi_n\to\phi_\star$ as $n\to\infty$ where $\phi_\star\in C^0_{\tp{m}}\tp{\Bar{U}}$ is determined via
    \begin{equation}\label{phi star}
        \phi_\star(x) = -\int_{\Sigma}\bf{N}\tp{x - y}\f{\phi(y)}{\nu\tp{y}\cdot \grad\Psi\tp{y}}\;\m{d}\mathcal{H}^1\tp{y},\;x\in U,
    \end{equation}
    where we recall that $\bf{N} = (2\pi)^{-1}\log\tabs{\cdot}$ is the $\R^2$-Newtonian potential. So let us fix an arbitrary test function $g\in C^\infty_{\m{c}}\tp{U}$ and consider the limit of the $L^2(U)$-pairings $\tcb{\tbr{\varphi_n,g}_{L^2\tp{U}}}_{n\in\N}\subset\R$. Thanks to Fubini's theorem and symmetry of the operators from equation~\eqref{generalized poisson solver}, we first have the identity
    \begin{equation}\label{identity as a distribution}
        \tbr{\varphi_n,g}_{L^2\tp{U}} = \sum_{k=-M}^M\sig_k\int_0^1[\tbr{\tp{\Gamma_{\ep_n}^\Psi}'\tp{(1-\tau)\psi^0_n + \tau\psi^1_n + k\rho_n}\phi_n,\bf{N}_{\Qoppa_n}\tsb{g}}_{L^2(\R^2)}\;\m{d}\tau
    \end{equation}
    for any $n\in\N$. The integrand above is a continuous function in $\tau\in[0,1]$. We deduce uniform boundedness in $n\in\N$ and convergence pointwise thanks to Propositions~\ref{prop on uniform intrgrals estimates for the Newtonian potential} and~\ref{prop on limit identification}. Thus, the dominated convergence theorem applies and we identify (tacitly using~\eqref{normalization of the splitting}) the limit of~\eqref{identity as a distribution} as $n\to\infty$ to be
    \begin{equation}\label{identification of the distributional limit}
        \tbr{\phi,g}_{L^2\tp{U}} = \lim_{n\to\infty}\tbr{\varphi_n,g}_{L^2\tp{U}} = - \int_{\Sigma}\f{\phi(y)}{\nu(y)\cdot\grad\Psi(y)}\tp{\bf{N}\ast g}\tp{y}\;\m{d}\mathcal{H}^1\tp{y} = \tbr{\phi_\star,g}_{L^2\tp{U}}.
    \end{equation}
    
    As~\eqref{identification of the distributional limit} holds for all $g\in C^\infty_{\m{c}}\tp{U}$ we deduce that $\phi = \phi_\star$ where the latter is defined in~\eqref{phi star}. This is a contradiction of the nondegeneracy condition from the fourth item of Definition~\ref{defn of admissible patches}. Therefore our contradiction hypothesis, that the statement of the proposition does not hold, must be false.
\end{proof}

Our first consequence of Proposition~\ref{prop on quantitative closed range estimate} allows us to deduce invertibility of the derivative of the map $\pmb{F}$ (defined in~\eqref{the nonlienar operators}) on an open subset of stream function perturbations. While the estimates~\eqref{important estimate 1} and~\eqref{important estiamte 2} that follow are manifestly \emph{not} uniform in $\ep$, the size of the neighborhood on which we have invertibility is uniform.

\begin{coroC}[Invertibility and estimates on the derivative]\label{coro on invertibility of the derivative}
     Let $\ep_1$, $\rho_1$, $r_1$, and $\Qoppa_1$ be the positive values granted by Proposition~\ref{prop on quantitative closed range estimate}. There exists $C\in\R^+$ such that for all $0<\ep\le\ep_1$, $0\le\rho\le\rho_1$, $0\le\Qoppa\le\Qoppa_1$, and $\psi,\tilde{\psi}\in\Bar{B(0,r_1)}\subset C^1_{\tp{m}}\tp{\Bar{U}}$ the derivative $D\pmb{F}_{\ep,\rho,\Qoppa}(\psi):C^1_{\tp{m}}\tp{\Bar{U}}\to C_{\tp{m}}^1\tp{\Bar{U}}$ is invertible and we have the estimates
    \begin{equation}\label{important estimate 1}
        \tnorm{\tsb{D\pmb{F}_{\ep,\rho,\Qoppa}(\psi)}^{-1}\phi}_{C^1\tp{\Bar{U}}} + \tnorm{\tsb{D\pmb{F}_{\ep,\rho,\Qoppa}\tp{\psi}}\phi}_{C^1\tp{\Bar{U}}}\le \tp{C/\ep}\tnorm{\phi}_{C^1\tp{\Bar{U}}}
    \end{equation}
    and
    \begin{equation}\label{important estiamte 2}
        \tnorm{\tsb{D\pmb{F}_{\ep,\rho,\Qoppa}\tp{\psi}}^{-1}\phi - \tsb{D\pmb{F}_{\ep,\rho,\Qoppa}\tp{\tilde{\psi}}}^{-1}\phi}_{C^1\tp{\Bar{U}}}\le\tp{C/\ep^3}\tnorm{\psi - \tilde{\psi}}_{C^1\tp{\Bar{U}}}\tnorm{\phi}_{C^0\tp{\Bar{U}}}
    \end{equation}
    for all $\phi\in C_{\tp{m}}^1\tp{\Bar{U}}$. 
\end{coroC}
\begin{proof}
    Identity~\eqref{the linearization identity} and estimate~\eqref{linear remainder c2} show that the map $D\pmb{F}_{\ep,\rho,\Qoppa}\tp{\psi}$ as a linear operator is a compact perturbation of the identity and thus has vanishing Fredholm index. Invertibility of $D\pmb{F}_{\ep,\rho,\Qoppa}\tp{\psi}$ follows as soon as we know that $\m{ker}D\pmb{F}_{\ep,\rho,\Qoppa}\tp{\psi} = \tcb{0}$. If $\phi\in C_{\tp{m}}^1\tp{\Bar{U}}$ belongs to this kernel, then we can use estimate~\eqref{closed range estimate in C0} from Proposition~\ref{prop on quantitative closed range estimate} (with $\psi^0 = \psi^1 = \psi$) to deduce that 
    \begin{equation}
        \tnorm{\phi}_{C^0\tp{\Bar{U}}}\le C\tnorm{\phi + \pmb{\kappa}_{\ep,\rho,\Qoppa}\tp{\psi}\phi}_{C^0\tp{\Bar{U}}} = C\tnorm{D\pmb{F}_{\ep,\rho,\Qoppa}\tp{\psi}\phi}_{C^0\tp{\Bar{U}}} = 0.
    \end{equation}
    The kernel is indeed trivial.

    We now prove estimate~\eqref{important estimate 1}. That the second summand on the left hand side is bounded by $\tp{C/\ep}\tnorm{\phi}_{C^1\tp{\Bar{U}}}$ is a consequence of estimate~\eqref{linear remainder c1} from Lemma~\ref{lem on prelim estimates 1}. We now focus on the first summand. Let $\varphi = \tsb{D\pmb{F}_{\ep,\rho,\Qoppa}\tp{\psi}}^{-1}\phi$. Thanks to Proposition~\ref{prop on quantitative closed range estimate} (again with $\psi^0 = \psi^1 = \psi$) we have initially that $\tnorm{\varphi}_{C^0\tp{\Bar{U}}}\le C\tnorm{\phi}_{C^0\tp{\Bar{U}}}$. On the other hand, we have $\varphi = \phi - \pmb{\kappa}_{\ep,\rho,\Qoppa}\tp{\psi}\varphi$ and so estimate~\eqref{linear remainder c1} from Lemma~\ref{lem on prelim estimates 1} provides the bound
    \begin{equation}
        \tnorm{\varphi}_{C^1\tp{\Bar{U}}}\le C\tnorm{\phi}_{C^1\tp{\Bar{U}}} + \tp{C/\ep}\tnorm{\varphi}_{C^0\tp{\Bar{U}}}\le\tp{C/\ep}\tnorm{\phi}_{C^1\tp{\Bar{U}}}.
    \end{equation}
    Estimate~\eqref{important estimate 1} is now established.

    We conclude by proving~\eqref{important estiamte 2}. Let $\varphi = \tsb{D\pmb{F}_{\ep,\rho,\Qoppa}\tp{\psi}}^{-1}\phi$ and $\tilde{\varphi} = \tsb{D\pmb{F}_{\ep,\rho,\Qoppa}\tp{\tilde{\psi}}}^{-1}\phi$. We compute first that
    \begin{equation}
        D\pmb{F}_{\ep,\rho,\Qoppa}\tp{\psi}\tp{\varphi - \tilde{\varphi}} = -\tp{\pmb{\kappa}_{\ep,\rho,\Qoppa}\tp{\psi} - \pmb{\kappa}_{\ep,\rho,\Qoppa}\tp{\tilde{\psi}}}\tilde{\varphi}.
    \end{equation}
    Now take the norm in $C^1\tp{\Bar{U}}$ on both sides of the above equation. On the left we use~\eqref{important estimate 1} while on the right we use estimate~\eqref{linear remainder lipschitz} from Lemma~\ref{lem on prelim estimates 1}. This leaves us with
    \begin{equation}
        \tnorm{\varphi - \tilde{\varphi}}_{C^1(\Bar{U})}\le\tp{C/\ep^3}\tnorm{\psi - \tilde{\psi}}_{C^1\tp{\Bar{U}}}\tnorm{\tilde{\varphi}}_{C^0(\Bar{U})}.
    \end{equation}
    Estimate~\eqref{important estiamte 2} follows by using Proposition~\ref{prop on quantitative closed range estimate} to bound $\tnorm{\tilde{\varphi}}_{C^0(\Bar{U})}\le C\tnorm{\phi}_{C^0\tp{\Bar{U}}}$.
\end{proof}

Our second consequence of Proposition~\ref{prop on quantitative closed range estimate} gives us a reverse uniform continuity estimate on $\pmb{F}$.

\begin{coroC}[Reverse logarithmic-Lipschitz estimate]\label{coro on reverse Holder estimate}
    Let $\ep_1$, $\rho_1$, $\Qoppa_1$, and $r_1$ be the positive values granted by Proposition~\ref{prop on quantitative closed range estimate}. There exists $C\in\R^+$ such that for all $0<\ep\le\ep_1$, $0\le\rho\le\rho_1$, $0\le\Qoppa\le\Qoppa_1$, and $\psi,\tilde{\psi}\in\Bar{B(0,r_1)}\subset C^1_{\tp{m}}\tp{\Bar{U}}$ we have the estimate
    \begin{multline}\label{the reverse Holder bound}
        \tnorm{\psi - \tilde{\psi}}_{C^1(\Bar{U})}\le C\tnorm{\pmb{F}_{\ep,\rho,\Qoppa}\tp{\psi} - \pmb{F}_{\ep,\rho,\Qoppa}\tp{\tilde{\psi}}}_{C^1\tp{\Bar{U}}} \\+ C\tnorm{\pmb{F}_{\ep,\rho,\Qoppa}\tp{\psi} - \pmb{F}_{\ep,\rho,\Qoppa}\tp{\tilde{\psi}}}_{C^0\tp{\Bar{U}}}\tp{1 + \tabs{\log\tnorm{\pmb{F}_{\ep,\rho,\Qoppa}\tp{\psi} - \pmb{F}_{\ep,\rho,\Qoppa}\tp{\tilde{\psi}}}_{C^0\tp{\Bar{U}}}}}.
    \end{multline}
\end{coroC}
\begin{proof}
    We begin by using Lemma~\ref{lem on principal actors and their basic properties} and the fundamental theorem of calculus to write
    \begin{equation}\label{the fundamental theorem of calculus identity}
        \pmb{F}_{\ep,\rho,\Qoppa}\tp{\psi} - \pmb{F}_{\ep,\rho,\Qoppa}\tp{\tilde{\psi}} = \tp{\psi - \tilde{\psi}} + \bsb{\int_0^1\pmb{\kappa}_{\ep,\rho,\Qoppa}\tp{(1-\tau)\tilde{\psi} + \tau\psi}\;\m{d}\tau}\tp{\psi - \tilde{\psi}}.
    \end{equation}
    We then take the norm in $C^0(\Bar{U})$ of the above expression and use Proposition~\ref{prop on quantitative closed range estimate} to deduce that
    \begin{equation}\label{the low regularity reverse lipschitz estimate}
        \tnorm{\psi - \tilde{\psi}}_{C^0\tp{\Bar{U}}}\le C\tnorm{\pmb{F}_{\ep,\rho,\Qoppa}\tp{\psi} - \pmb{F}_{\ep,\rho,\Qoppa}\tp{\tilde{\psi}}}_{C^0\tp{\Bar{U}}}.
    \end{equation}

    On the other hand from~\eqref{the nonlienar operators} we also have the estimate
    \begin{equation}
        \tnorm{\psi - \tilde{\psi}}_{C^1\tp{\Bar{U}}}\le\tnorm{\pmb{F}_{\ep,\rho,\Qoppa}\tp{\psi} - \pmb{F}_{\ep,\rho,\Qoppa}\tp{\tilde{\psi}}}_{C^1\tp{\Bar{U}}} + \tnorm{\pmb{K}_{\ep,\rho,\Qoppa}\tp{\psi} - \pmb{K}_{\ep,\rho,\Qoppa}\tp{\tilde{\psi}}}_{C^1\tp{\Bar{U}}}.
    \end{equation}
    Combining the above with Proposition~\ref{prop on preliminary mapping estimates II} and the bound~\eqref{the low regularity reverse lipschitz estimate} gives the claim~\eqref{the reverse Holder bound}.
\end{proof}

\subsection{Quantitative local invertibility}\label{subsection on quantitative local invertibility}

Somewhat remarkably, the conclusions of Corollaries~\ref{coro on invertibility of the derivative} and~\ref{coro on reverse Holder estimate} are all we need to deduce local invertibility near zero for the smooth maps $\pmb{F}_{\ep,\rho,\Qoppa}$. The properties established by these results ensure local surjectivity by guaranteeing the convergence of Newton's method, despite the proximity to singularities. We have the following abstract result that is inspired by Hamilton's version of the Nash-Moser inverse function theorem~\cite{MR656198}.

\begin{thmC}[Abstract quantitative local surjectivity]\label{thm on abstract quant loc surj}
    Let $X$ and $Y$ be Banach spaces, $r^+,r^-,\chi\in\R^+$ with $r^-<r^+$, and $F:B_X(0,r^+)\to Y$ be continuously differentiable and satisfy:
    \begin{enumerate}
        \item $F(0) = 0$ and if $x\in B_X(0,r^+)$ is such that $\tnorm{F(x)}_{Y}\le\chi$ then $\tnorm{x}_X\le r^-$.
        \item There exists $L_0\in\R^+$ such that for all $x\in B_X(0,r^+)$ we have the estimates
        \begin{equation}
            \tnorm{F(x)}_X + \tnorm{DF(x)}_{\mathcal{L}(X;Y)}\le L_0.
        \end{equation}
        \item There exists $L_1\in\R^+$ and a map $R:B_X(0,r^+)\to\mathcal{L}\tp{Y;X}$ such that for all $x,\tilde{x}\in B_X(0,r^+)$ with $x\neq\tilde{x}$ and all $y\in Y$ it holds that
        \begin{equation}\label{hypotheses on the right inverse}
            \tnorm{R(x)}_{\mathcal{L}(Y;X)} + \tnorm{R(x) - R(\tilde{x})}_{\mathcal{L}(Y;X)}/\tnorm{x - \tilde{x}}_{X}\le L_1\text{ and }DF(x)R(x)y = y.
        \end{equation}
    \end{enumerate}
    Then, for each $y\in Y$ with $\tnorm{y}_{Y}\le\chi$ there exists $x\in X$ with $\tnorm{x}_{X}\le r^-$ with $F(x) = y$.
\end{thmC}
\begin{proof}
    We begin by defining an auxiliary map. We set $r^0 = (r^- + r^+)/2$ and 
    \begin{equation}\label{the timestep}
        T = \min\tcb{\tp{r^+ - r^-}/2L_1\tp{\chi + L_0},1/2L_1\tp{\chi + 2L_0}}\in\R^+.
    \end{equation}
    Then, for any $T_0\in[0,\infty)$ let us define
    \begin{equation}\label{the aux map Phi}
        \Phi_{T_0}:\Bar{B_X(0,r^-)}\times\Bar{B_Y\tp{0,\chi}}\times C^0\tp{[T_0,T_0 + T];\Bar{B_X(0,r^0)}}\to C^0\tp{[T_0,T_0 + T];\Bar{B_X(0,r^0)}}
    \end{equation}
    to be the operator
    \begin{equation}\label{the defn of the aux map Phi}
        \Phi_{T_0}(x,y,\gam)(t) = x + \int_{T_0}^tR(\gam(s))\tp{y - F(\gam(s))}\;\m{d}s
    \end{equation}
    for $x\in\Bar{B_X(0,r^-)}$, $y\in\Bar{B_Y(0,\chi)}$, $\gam\in C^0\tp{[T_0,T_0 + T];\Bar{B_X(0,r^0)}}$, and $t\in[T_0,T_0 + T]$. Let us verify that $\Phi_{T_0}$ is well defined in the sense that it maps into the codomain stated in~\eqref{the aux map Phi}. By using the hypotheses of the proposition and the definition of $T$ in~\eqref{the timestep} we estimate for $x$, $y$, $\gam$, and $t$ belonging to the aforementioned sets
    \begin{equation}
        \tnorm{\Phi_{T_0}(x,y,\gam)(t)}_X\le r^- + TL_1\tp{\chi + L_0}\le r^0.
    \end{equation}
    So indeed $\Phi_{T_0}$ is well-defined. 

    We now claim that $\Phi_{T_0}$ is a contraction in its final argument. We estimate for any $x\in\Bar{B_X(0,r^-)}$, $y\in\Bar{B_Y(0,\chi)}$, $\gam,\tilde{\gam}\in C^0\tp{[T_0,T_0 + T],\Bar{B_X(0,r^0)}}$, and $t\in[T_0,T]$ that 
    \begin{equation}
        \tnorm{\Phi_{T_0}\tp{x,y,\gam}(t) - \Phi_{T_0}\tp{x,y,\tilde{\gam}}\tp{t}}_X\le TL_1\tp{\chi + 2L_0}\tnorm{\gam - \tilde{\gam}}_{C^0\tp{[T_0,T_0 + T];X}}.
    \end{equation}
    Thanks to~\eqref{the timestep}, we are assured that $TL_1\tp{\chi + 2L_0}\le1/2$ and thus the above estimate shows the desired contraction estimates.

    We are in a position to invoke the contraction mapping principle and deduce the existence of a function 
    \begin{equation}\label{the fixed point map}
        \varphi_{T_0}:\Bar{B_X(0,r^-)}\times\Bar{B_Y(0,\chi)}\to C^0\tp{[T_0,T_0 + T];\Bar{B_X(0,r^0)}}
    \end{equation}
    with the property that $\Phi_{T_0}\tp{x,y,\varphi_{T_0}(x,y)} = \varphi_{T_0}\tp{x,y}$ for all $x$ and $y$ belonging to the domain of~\eqref{the fixed point map}. Immediately from~\eqref{the defn of the aux map Phi} we deduce that $\varphi_{T_0}(x,y)\in C^1\tp{[T_0,T_0 + T];\Bar{B_X(0,r^0)}}$ and is a solution to the ODE
    \begin{equation}\label{ODE form of the fixed point}
        \varphi_{T_0}\tp{x,y}(T_0) = x\text{ and for all }t\in[T_0,T_0 + T]\;[\varphi_{T_0}\tp{x,y}]'(t) = R(\varphi_{T_0}\tp{x,y}(t))\tp{y - F\tp{\varphi_{T_0}\tp{x,y}\tp{t}}}.
    \end{equation}
    Therefore, for any $z\in Y^\star$ (the dual space of $Y$) we can use~\eqref{ODE form of the fixed point} and the right hand identity of~\eqref{hypotheses on the right inverse} to compute that
    \begin{equation}
        [\tbr{y - F(\varphi_{T_0}\tp{x,y}(\cdot)),z}_{Y,Y^\ast}]'(t) = -\tbr{y - F(\varphi_{T_0}\tp{x,y}(t)),z}_{Y,Y^\ast}
    \end{equation}
    for $t\in[T_0,T_0 + T]$. Hence, for such $t$ we have the identity
    \begin{equation}
        \tbr{y - F(\varphi_{T_0}\tp{x,y}(t)),z}_{Y,Y^\ast} = e^{-\tp{t - T_0}}\tbr{y - F(x),z}_{Y,Y^\ast}.
    \end{equation}
    Since $z\in Y^\ast$ was arbitrary, we deduce that for all $t\in[T_0,T_0 + T]$ 
    \begin{equation}\label{the important identity for newtons method}
        F(\varphi_{T_0}\tp{x,y}(t)) = e^{-\tp{t - T_0}}F(x) + (1 - e^{-(t - T_0)})y.
    \end{equation}

    Let us move on to the next stage of our construction. We fix $y\in\Bar{B_Y(0,\chi)}$ and define a sequence $\tcb{x_n}_{n\in\N}\subset\Bar{B_X(0,r^-)}$ satisfying $F(x_n) = (1-e^{-nT})y$, for all $n\in\N$, inductively as follows. We set $x_0 = 0$; since $F(0) = 0$ we have $F(x_0) = (1 - e^{-0})y$. Now if $n\in\N$ and $x_n\in\Bar{B_X(0,r^-)}$ has been defined and satisfies the correct image identity, we take
    \begin{equation}\label{inductive construction of the sequence}
        x_{n+1} = \varphi_{nT}\tp{x_n,y}((n+1)T)
    \end{equation}
    where $T$ is from~\eqref{the timestep} and $\varphi$ is the fixed point from~\eqref{the fixed point map}. Initially we only know that $x_{n+1}\in\Bar{B_X\tp{0,r^0}}$. However, thanks to~\eqref{the important identity for newtons method} and the induction hypothesis, we get
    \begin{equation}\label{propigation of the indunction hypothesis}
        F(x_{n+1}) = e^{-T}F(x_n) + (1-e^{-T})y = (1 - e^{-(n+1)T})y.
    \end{equation}
    Then from the first hypothesis we find that $\tnorm{x_{n+1}}_{X}\le r^-$ since~\eqref{propigation of the indunction hypothesis} implies that $\tnorm{F(x_{n+1})}_Y\le\chi$. This completes the inductive construction.

    Let us now show that the sequence $\tcb{x_n}_{n\in\N}\subset\Bar{B_X(0,r^-)}$ is Cauchy. To see this, we first use identity~\eqref{inductive construction of the sequence} and the fact that $\varphi$ is a fixed point of $\Phi$ from~\eqref{the defn of the aux map Phi} to obtain the difference identity
    \begin{equation}\label{the differnece identity}
        x_{n+1} - x_n = \int_{nT}^{(n+1)T}R(\varphi_{nT}\tp{x_n,y}\tp{s})\tp{y - F(\varphi_{nT}\tp{x_n,y}(s))}\;\m{d}s.
    \end{equation}
    Now from~\eqref{the important identity for newtons method} it follows that for $s\in[nT,(n+1)T]$
    \begin{equation}
        y - F(\varphi_{nT}\tp{x_n,y}(s)) = e^{-(s - nT)}\tp{y - F(x_n)} = e^{-s}y.
    \end{equation}
    Therefore~\eqref{the differnece identity} admits the upper bound
    \begin{equation}
        \tnorm{x_{n+1} - x_n}_{X} \le L_1\chi\tp{1 - e^{-T}}e^{-nT}\text{ and hence }\sum_{n\in\N}\tnorm{x_{n+1} - x_n}_{X}<\infty.
    \end{equation}
    We conclude that the sequence in question is indeed Cauchy and so there exists $x\in\Bar{B_X(0,r^-)}$ such that $x_n\to x$ as $n\to\infty$. From the identity $F(x_n) = (1-e^{-nT})y$, that holds for every $n\in\N$, we deduce that $F(x) = y$. This completes the proof.
\end{proof}

\begin{coroC}[Local invertibility]\label{coro on local invertibility}
    Let $\ep_1$, $\rho_1$, $\Qoppa_1$, and $r_1$ be the positive values granted by Proposition~\ref{prop on quantitative closed range estimate}. There exists $C,r_2\in\R^+$ with $r_2\le r_1$ such that for all $0<\ep\le\ep_1$, $0\le\rho\le\rho_1$, $0\le\Qoppa\le\Qoppa_1$ the following hold.
    \begin{enumerate}
        \item For each $f\in\Bar{B(0,r_2)}\subset C^1_{\tp{m}}\tp{\Bar{U}}$, there exists a unique $\psi\in\Bar{B(0,r_1)}\subset C^1_{\tp{m}}\tp{\Bar{U}}$ such that $\pmb{F}_{\ep,\rho,\Qoppa}\tp{\psi} = f$. This defines a left inverse $\tp{\pmb{F}_{\ep,\rho,\Qoppa}}^{-1}:\Bar{B(0,r_2)}\to\Bar{B(0,r_1)}$ mapping between subsets of $C^1_{\tp{m}}\tp{\Bar{U}}$.
        \item The left inverses from the previous item obey the logarithmic-Lipschitz estimate
        \begin{equation}
            \tnorm{\tp{\pmb{F}_{\ep,\rho,\Qoppa}}^{-1}\tp{f} - \tp{\pmb{F}_{\ep,\rho,\Qoppa}}^{-1}\tp{\tilde{f}}}_{C^1\tp{\Bar{U}}}\le C\tnorm{f -\tilde{f}}_{C^1\tp{\Bar{U}}} + C\tnorm{f -\tilde{f}}_{C^0\tp{\Bar{U}}}\tp{1 + \tabs{\log\tnorm{f - \tilde{f}}_{C^0\tp{\Bar{U}}}}}
        \end{equation}
        for all $f,\tilde{f}\in\Bar{B(0,r_2)}\subset C^1_{\tp{m}}\tp{\Bar{U}}$.
    \end{enumerate}
\end{coroC}
\begin{proof}
    We begin by proving the first item. The strategy is to show that the hypotheses of Theorem~\ref{thm on abstract quant loc surj} are satisfied for $F = \pmb{F}_{\ep,\rho,\Qoppa}$, $X = Y = C^1_{\tp{m}}\tp{\Bar{U}}$, $r^+ = r_1$, and $\chi$, $r^-$ to be determined. Thanks to Proposition~\ref{coro on reverse Holder estimate} in the special case that $\tilde{\psi} = 0$, we know that
    \begin{equation}
        \tnorm{\psi}_{C^1\tp{\Bar{U}}}\le C\tnorm{\pmb{F}_{\ep,\rho,\Qoppa}\tp{\psi}}_{C^1\tp{\Bar{U}}}\tp{1 + \tabs{\log\tnorm{\pmb{F}_{\ep,\rho,\Qoppa}\tp{\psi}}_{C^1\tp{\Bar{U}}}}}\le C\tnorm{\pmb{F}_{\ep,\rho,\Qoppa}\tp{\psi}}_{C^1\tp{\Bar{U}}}^{1/2}
    \end{equation}
    for all $\tnorm{\psi}_{C^1\tp{\Bar{U}}}\le r_1$. Therefore, we may take $r^- = r_1/2$ and $\chi = (r_1/2C)^2$ and conclude (by also implicitly using Corollary~\ref{coro on invertibility of the derivative}) that the hypotheses of Theorem~\ref{thm on abstract quant loc surj} are met. We set $r_2 = \chi$.

    We are thus granted for each $f\in\Bar{B(0,r_2)}\subset C_{\tp{m}}^1\tp{\Bar{U}}$ at least one $\psi\in\Bar{B(0,r_1)}\subset C^1_{\tp{m}}\tp{\Bar{U}}$ such that $
    \pmb{F}_{\ep,\rho,\Qoppa}\tp{\psi} = f$. That there is at most one such $\psi$ follows from Corollary~\ref{coro on reverse Holder estimate}'s establishment of local injectivity. This proves the first item. The second item is an immediate consequence of the aforementioned corollary.
\end{proof}

We now recall the definition of the functions $f_{\ep,\rho,\Qoppa}$ introduced in equation~\eqref{the source term target} of Lemma~\ref{lem on principal actors and their basic properties}. Our next result, in a certain sense, proves Theorem~\ref{thm_main_1}, \ref{thm_main_2}, and~\ref{thm_main_3} simultaneously.

\begin{coroC}[Existence and estimates for the perturbed equation]\label{coro on existence and estimates for the regularized equation}
        Let $\ep_1,\rho_1$, $\Qoppa_1$, and $r_1$ be the positive values granted by Proposition~\ref{prop on quantitative closed range estimate}, and let $0<r_2\le r_1$ be from Corollary~\ref{coro on local invertibility}. There exists $\ep_\star,\rho_\star,\Qoppa_\star,C\in\R^+$ with $\ep_\star\le\ep_1$, $\rho_\star\le\rho_1$, $\Qoppa_\star\le\Qoppa_1$ such that for all $0<\ep\le\ep_\star$, $0\le\rho\le\rho_\star$, $0\le\Qoppa\le\Qoppa_\star$ the following hold.
    \begin{enumerate}
        \item There exists $\psi_{\ep,\rho,\Qoppa}\in \Bar{B(0,r_1)}\subset C_{\tp{m}}^1\tp{\Bar{U}}$ satisfying $\pmb{F}_{\ep,\rho,\Qoppa}\tp{\psi_{\ep,\rho,\Qoppa}} = f_{\ep,\rho,\Qoppa}$ along with the estimates
        \begin{equation}\label{estimates on the solutions}
            \tnorm{\psi_{\ep,\rho,\Qoppa}}_{\m{LL}^1\tp{\Bar{U}}}\le C,\;\tnorm{\psi_{\ep,\rho,\Qoppa}}_{C^1\tp{\Bar{U}}}\le C\tp{\ep\log\tp{1/\ep} + \rho\log\tp{1/\rho} + \Qoppa\log\tp{1/\Qoppa}}.
        \end{equation}
    \item There exists $\Psi_{\ep,\rho,\Qoppa}\in C^\infty_{\tp{m}}\tp{\Bar{B\tp{0,\Qoppa^{-1/2}}}}$ and $\omega_{\ep,\rho,\Qoppa}\in C^\infty_{\tp{m},\m{c}}\tp{B(0,\Qoppa^{-1/2}/2)}$ satisfying
        \begin{equation}\label{what it means to solve the steady rotating euler equations}
            \Psi_{\ep,\rho,\Qoppa} = \bf{N}_\Qoppa\tsb{\omega_{\ep,\rho,\Qoppa}} - \f{\Omega}{2}\tabs{\cdot}^2 - c\text{ and }\grad^\perp\Psi_{\ep,\rho,\Qoppa}\cdot\grad\omega_{\ep,\rho,\Qoppa} = 0\text{ on }\Bar{B(0,\Qoppa^{-1/2})};
        \end{equation}
        moreover, recalling $\Gamma^\Psi_\ep$ from Lemma~\ref{lem on regularized hevi},
        \begin{equation}\label{unfortunately this needs to be cited}
            \Psi_{\ep,\rho,\Qoppa} = \Psi + \psi_{\ep,\rho,\Qoppa}\text{ and }\omega_{\ep,\rho,\Qoppa} = \sum_{k=-M}^M\sig_k\Gamma^\Psi_\ep\tp{\psi_{\ep,\rho,\Qoppa} + k \rho}
        \end{equation}
        on the set $\Bar{U}$. In the case that $\Qoppa>0$ we also have that $\Psi_{\ep,\rho,\Qoppa} + \Omega|\cdot|^2/2$ is constant on $\pd B(0,\Qoppa^{-1/2})$.
        \item  We have the estimates
        \begin{multline}\label{estimates on the stream functions and vorticity in BV}
            \tnorm{\omega_{\ep,\rho,\Qoppa} - \mathds{1}_{\tilde{U}}}_{L^1\tp{B(0,\Qoppa^{-1/2})}} + \tabs{\tnorm{\grad\omega_{\ep,\rho,\Qoppa}}_{\m{TV}\tp{B(0,\Qoppa^{-1/2})}} - \tnorm{\grad\mathds{1}_{\tilde{U}}}_{\m{TV}\tp{B(0,\Qoppa^{-1/2})}}} \\+ \tnorm{\grad\tp{\Psi_{\ep,\rho,\Qoppa} - \Psi}}_{L^\infty\tp{B(0,\Qoppa^{-1/2})}}\le C\tp{\ep\log\tp{1/\ep} + \rho\log\tp{1/\rho} + \Qoppa\log\tp{1/\Qoppa}},
        \end{multline}
        with $\m{TV}$ referring to the total variation norm -- see Section~\ref{section on notational conventions}.
    \end{enumerate}
\end{coroC}
\begin{proof}
    For the first part of the proof, we shall estimate the norm of the source terms $f_{\ep,\rho,\Qoppa}$, defined in~\eqref{the source term target}, in the space $C^1_{\tp{m}}\tp{\Bar{U}}$. We begin by considering the decomposition
    \begin{equation}\label{the summing up of the decomposition}
        f_{\ep,\rho,\Qoppa} = f^0_{\ep,\rho,\Qoppa} + f^1_{\ep,\rho} + f^2_{\ep} 
    \end{equation}
    with 
    \begin{equation}\label{the further decomposition of the source_0}
        f^0_{\ep,\rho,\Qoppa} = \sum_{k=-M}^M\sig_k\tp{\bf{N}_{\Qoppa} - \bf{N}_0}\tsb{\Gamma^\Psi_\ep\tp{k\rho}},
    \end{equation}
    \begin{equation}\label{the further decomposition of the source_1}
        f^1_{\ep,\rho} = \sum_{k=-M}^M\sig_k\bf{N}\ast\tp{\Gamma^\Psi_\ep\tp{k\rho} - \Gamma^\Psi_\ep\tp{0}},
    \end{equation}
    and
    \begin{equation}\label{the further decomposition of the source_2}
        f^2_\ep = \bf{N}\ast\tp{\Gamma_\ep^\Psi\tp{0} - \mathds{1}_{\tilde{U}}}.
    \end{equation}
    The claim is that there exists $C$ such that for all $0<\ep\le\ep_1$, $0\le\Qoppa\le\Qoppa_1$, and $0\le\rho\le\rho_1$ we have the estimates
    \begin{multline}\label{the claim on the size of the source terms}
        \tnorm{f^0_{\ep,\rho,\Qoppa}}_{C^0\tp{\Bar{U}}} + \tp{1/\log\tp{1/\Qoppa}}\tnorm{f^0_{\ep,\rho,\Qoppa}}_{C^1\tp{\Bar{U}}}\le C\Qoppa,\;\tnorm{f^1_{\ep,\rho}}_{C^0\tp{\Bar{U}}} + \tp{1/\log\tp{1/\rho}}\tnorm{f^1_{\ep,\rho}}_{C^1\tp{\Bar{U}}}\le C\rho,\\\text{and }\tnorm{f^2_\ep}_{C^0\tp{\Bar{U}}} + \tp{1/\log\tp{1/\ep}}\tnorm{f^2_\ep}_{C^1\tp{\Bar{U}}}\le C\ep.
    \end{multline}
    
    By arguing exactly as in the first part of the proof of Proposition~\ref{prop on preliminary mapping estimates II}, we find that 
    \begin{equation}\label{initial LL1 estimates on the source term}
        \tnorm{f^0_{\ep,\rho,\Qoppa}}_{\m{LL}^1\tp{\Bar{U}}} + \tnorm{f^1_{\ep,\rho}}_{\m{LL}^1\tp{\Bar{U}}} + \tnorm{f^2_{\ep}}_{\m{LL}^1\tp{\Bar{U}}}\le C.
    \end{equation}
    The interpolation results of Appendix~\ref{appendix on an interpolation inequality} combine with~\eqref{initial LL1 estimates on the source term} to reduce the task of establishing estimate~\eqref{the claim on the size of the source terms} to showing the following bounds:
    \begin{equation}\label{second LL estimates on the source terms}
        \tnorm{f^0_{\ep,\rho,\Qoppa}}_{\m{LL}\tp{\Bar{U}}}\le C\Qoppa,\;\tnorm{f^1_{\ep,\rho}}_{\m{LL}\tp{\Bar{U}}}\le C\rho,\text{ and }\tnorm{f^2_\ep}_{\m{LL}\tp{\Bar{U}}}\le C\ep. 
    \end{equation}

    The first bound in~\eqref{second LL estimates on the source terms} for the source term $f^0_{\ep,\rho,\Qoppa}$ holds thanks to the simple uniform kernel bounds
    \begin{equation}
        \tabs{\bf{N}_\Qoppa(x,y) - \bf{N}_0(x,y)} = \tabs{\log\tp{1 - 2\Qoppa x\cdot y + \Qoppa^2|x|^2|y|^2}}/4\pi\le C\Qoppa
    \end{equation}
    and
    \begin{equation}
        \tabs{\tp{\bf{N}_\Qoppa - \bf{N}_0}(x,y) - \tp{\bf{N}_\Qoppa - \bf{N}_0}(\tilde{x},y)}/\tabs{x - \tilde{x}}\le C\Qoppa,
    \end{equation}
    that are valid for all $x,\tilde{x},y\in\Bar{U\cup\tilde{U}}\subset B(0,\Qoppa_0^{-1/2}/2)$ with $x\neq\tilde{x}$.

   The establishment of the second bound in~\eqref{second LL estimates on the source terms} is our next task. For $x\in\Bar{U}$ the coarea formula (Theorem 3.2.11 in~\cite{MR257325} or Theorem 3.13 in~\cite{MR3409135}) and the fundamental theorem of calculus give us the identity
   \begin{equation}
       f^1_{\ep,\rho}\tp{x} = -\f{\rho}{2\pi}\sum_{k=-M}^Mk\sig_k\int_0^1\int_{-\ep-|k\rho|}^{\ep + |k\rho|}\f{1}{\ep}\gam'\bp{-\f{t + \tau k\rho}{\ep}}\int_{\tcb{\Psi = t}}\f{\log|x - y|}{|\grad\Psi(y)|}\;\m{d}\mathcal{H}^1\tp{y}\;\m{d}t\;\m{d}\tau.
   \end{equation}
   By taking the norm in $\m{LL}\tp{\Bar{U}}$, arguing as in the proof of Proposition~\ref{prop on uniform intrgrals estimates for the Newtonian potential} to bound the logarithmic terms over the level sets, and performing a simple change of variables, we arrive at the desired estimate:
   \begin{equation}
       \tnorm{f^1_{\ep,\rho}}_{\m{LL}\tp{\Bar{U}}}\le C\rho\sum_{k=-M}^M\int_0^1\int_{-\ep - |k\rho|}^{\ep + |k\rho|}\f{1}{\ep}\babs{\gam'\bp{-\f{t + \tau k\rho}{\ep}}}\;\m{d}t\;\m{d}\tau\le C\rho.
   \end{equation}

   Let us now discuss the proof of the third bound of~\eqref{second LL estimates on the source terms}. We again use the coarea formula to derive the pointwise equality
   \begin{equation}\label{just the way you look}
       f^2_\ep\tp{x} = \f{1}{2\pi}\int_{-\ep}^\ep\tp{\gam(-t/\ep) - \m{sgn}\tp{-t}}\int_{\tcb{\Psi = t}}\f{\log\tabs{x - y}}{\tabs{\grad\Psi(y)}}\;\m{d}\mathcal{H}^1\tp{y}\;\m{d}t
   \end{equation}
   for $x\in\Bar{U}$. Again we take the norm in $\m{LL}(\Bar{U})$ of~\eqref{just the way you look}, estimate the logarithmic contributions with the proof of Proposition~\ref{prop on uniform intrgrals estimates for the Newtonian potential}, and make a change of variables to get the sought after bound:
   \begin{equation}
       \tnorm{f^2_\ep}_{\m{LL}\tp{\Bar{U}}}\le C\int_{-\ep}^\ep\tabs{\gam(-t/\ep) - \m{sgn}\tp{-t}}\;\m{d}t\le C\ep.
   \end{equation}

    Estimate~\eqref{the claim on the size of the source terms} paired with the decomposition~\eqref{the summing up of the decomposition} implies the existence of $0<\ep_\star\le\ep_1$, $0<\rho_\star\le\rho_1$, and $0<\Qoppa_\star\le\Qoppa_1$ such that $\tnorm{f_{\ep,\rho,\Qoppa}}_{C^1\tp{\Bar{U}}}\le r_2$ whenever $0<\ep\le\ep_\star$, $0\le\rho\le\rho_\star$, and $0\le\Qoppa\le\Qoppa_\star$. The first item of Corollary~\ref{coro on local invertibility} then allows us to define $\psi_{\ep,\rho,\Qoppa} = \tp{\pmb{F}_{\ep,\rho,\Qoppa}}^{-1}\tp{f_{\ep,\rho,\Qoppa}}\in\Bar{B(0,r_1)}\subset C^1_{\tp{m}}\tp{\Bar{U}}$. In turn, the estimates~\eqref{the claim on the size of the source terms} in conjunction with the second item of the aforementioned corollary imply the second estimate in~\eqref{estimates on the solutions}. The first estimate of~\eqref{estimates on the solutions} follows from~\eqref{initial LL1 estimates on the source term} and~\eqref{ee_0}. The first item is now shown.

    We turn our attention to the second item. Define $\Psi_{\ep,\rho,\Qoppa}\in C^1_{\tp{m}}\tp{\Bar{U}}$ via $\Psi_{\ep,\rho,\Qoppa} = \Psi + \psi_{\ep,\rho,\Qoppa}$ with the latter summand being the function constructed by the previous item. Then Lemma~\ref{lem on regularized hevi} allows us to define $\omega_{\ep,\rho,\Qoppa}\in C^1_{\tp{m},\m{c}}\tp{B(0,\Qoppa^{-1/2}/2)}$ via 
    \begin{equation}\label{definition of omega}
        \omega_{\ep,\rho,\Qoppa} = \sum_{k=-M}^M\sig_k\Gamma^\Psi_\ep\tp{\psi_{\ep,\rho,\Qoppa} + k \rho}.    
    \end{equation}
    In fact, we have the inclusions
    \begin{equation}\label{support conditions on the vorticity}
        \m{supp}\tp{\omega_{\ep,\rho,\Qoppa}}\subset\Bar{U}\cup U^{\m{in}}\text{ and }\m{supp}\tp{\grad\omega_{\ep,\rho,\Qoppa}}\subset U.
    \end{equation}

    Next, we claim that the formula $\Psi_{\ep,\rho,\Qoppa} = \bf{N}_\Qoppa\tsb{\omega_{\ep,\rho,\Qoppa}} + \Omega\tabs{\cdot}^2/2 - c$ constitutes a smooth extension of $\Psi_{\ep,\rho,\Qoppa}$ to all of $\Bar{B(0,\Qoppa^{-1/2})}$. Since $\omega_{\ep,\rho,\Qoppa}\in C^1_{\tp{m},\m{c}}\tp{B(0,\Qoppa^{-1/2}/2)}$ and $\bf{N}_{\Qoppa}$, being the Poisson kernel, is an operator of order $-2$, the function
    \begin{equation}\label{the extended Psi identity}
        \tilde{\Psi}_{\ep,\rho,\Qoppa} = \bf{N}_\Qoppa\tsb{\omega_{\ep,\rho,\Qoppa}} - \Omega\tabs{\cdot}^2/2 - c
    \end{equation}
    is indeed in the regularity class $C^2_{\tp{m}}\tp{\Bar{B(0,\Qoppa^{-1/2})}}$. That we have the equality $\Psi_{\ep,\rho,\Qoppa} = \tilde{\Psi}_{\ep,\rho,\Qoppa}$ on the set $\Bar{U}$ is then a consequence of $\psi_{\ep,\rho,\Qoppa} = \tp{\pmb{F}_{\ep,\rho,\Qoppa}}^{-1}\tp{f_{\ep,\rho,\Qoppa}}$ and the equivalence in the third item of Lemma~\ref{lem on principal actors and their basic properties}.

    The second item is complete as soon as we verify that: 1) $\omega_{\ep,\rho,\Qoppa},\Psi_{\ep,\rho,\Qoppa}\in C^\infty\tp{\Bar{B(0,\Qoppa^{-1/2})}}$ and 2) the second equation of~\eqref{what it means to solve the steady rotating euler equations} is satisfied (the first equation is satisfied thanks to~\eqref{the extended Psi identity}). Item 1) follows from a simple induction argument, properties of the Poisson kernel, identities~\eqref{definition of omega} and~\eqref{the extended Psi identity}, and Lemma~\ref{lem on regularized hevi}. By the second inclusion of~\eqref{support conditions on the vorticity}, it suffices to verify that $\grad^\perp\Psi_{\ep,\rho,\Qoppa}\cdot\grad\omega_{\ep,\rho,\Qoppa} = 0$ on $U$. But $\omega_{\ep,\rho,\Qoppa} = \mathsf{v}_{\ep,\rho}\tp{\Psi_{\ep,\rho,\Qoppa}}$ where
    \begin{equation}
        \mathsf{v}_{\ep,\rho}\tp{t} = \sum_{k=-M}^M\sig_k\gam\bp{-\f{t + k\rho}{\ep}}\text{ for }t\in\R
    \end{equation}
    is a smooth function and so the equation is satisfied as a result of the chain rule and $\grad^\perp\Psi_{\ep,\rho,\Qoppa}\cdot\grad\Psi_{\ep,\rho,\Qoppa} = 0$.

    At last, we consider the third item. The claimed relative stream function estimate of~\eqref{estimates on the stream functions and vorticity in BV} in the region $\Bar{U}$, i.e. the bound $\tnorm{\grad\tp{\Psi_{\ep,\rho,\Qoppa} - \Psi}}_{L^\infty\tp{\Bar{U}}}\le C\tp{\ep\log\tp{1/\ep} + \rho\log\tp{1/\rho} + \Qoppa\log\tp{1/\Qoppa}}$, follows directly from~\eqref{estimates on the solutions}, so it suffices to study the behavior in the region $\Bar{B(0,\Qoppa^{-1/2})}\setminus\Bar{U}$. By using familiar coarea formula based arguments and that
    \begin{equation}
        \sup\tcb{\tabs{D_1\bf{N}_{\Qoppa}(x,y)}\;:\;x\in\Bar{B(0,\Qoppa^{-1/2})}\setminus\Bar{U},\;y\in\m{supp}\tp{\omega_{\ep,\rho} - \mathds{1}_{\tilde{U}}}}\le C,
    \end{equation}
    we get for any $x\in\Bar{B(0,\Qoppa^{-1/2})}\setminus\Bar{U}$ the pointwise bound
    \begin{multline}\label{also shows the L1 bounds}
        \tabs{\grad\tp{\Psi_{\ep,\rho,\Qoppa} - \Psi}\tp{x}}\le C\int_{U}\tabs{\omega_{\ep,\rho,\Qoppa}\tp{y} - \mathds{1}_{\tilde{U}}\tp{y}}\;\m{d}y + C\Qoppa
        \le C\sum_{k=-M}^M\int_{U}\babs{\gam\bp{-\f{\Psi_{\ep,\rho,\Qoppa} + k\rho}{\ep}} - \gam\bp{-\f{\Psi + k\rho}{\ep}}} \\+ C\int_U\babs{\sum_{k=-M}^M\sig_k\gam\bp{-\f{\Psi + k\rho}{\ep}} - \mathds{1}_{\tilde{U}}}
        + C\Qoppa\le C\tp{\ep\log{1/\ep} + \rho\log\tp{1/\rho} + \Qoppa\log\tp{1/\Qoppa}}.
    \end{multline}
    The desired inequality is now apparent.

    The last thing to prove are the $L^1\tp{B(0,\Qoppa^{-1/2})}$ and $\m{TV}\tp{B(0,\Qoppa^{-1/2})}$ estimates on the vorticity. The $L^1(B(0,\Qoppa^{-1/2}))$ bounds are established by a computation similar to~\eqref{also shows the L1 bounds}. On the other hand, for the $\m{TV}\tp{B(0,\Qoppa^{-1/2})}$-estimates, we use the coarea formula to identify
    \begin{equation}\label{_label_1}
        \tnorm{\grad\mathds{1}_{\tilde{U}}}_{\m{TV}\tp{B(0,\Qoppa^{-1/2})}} = \mathcal{H}^1\tp{\tcb{\Psi = 0}}\text{ and }\tnorm{\grad\omega_{\ep,\rho,\Qoppa}}_{\m{TV}\tp{B(0,\Qoppa^{-1/2})}} = \int_{-\ep-M|\rho|}^{\ep + M|\rho|}\mathsf{v}_{\ep,\rho}'\tp{t}\mathcal{H}^1\tp{\tcb{\Psi_{\ep,\rho,\Qoppa} = t}}\;\m{d}t.
    \end{equation}
    The above can be combined with the fourth item of Corollary~\ref{coro on refined estimates on level sets} to estimate
    \begin{multline}\label{_label_2}
        \tabs{\tnorm{\grad\omega_{\ep,\rho,\Qoppa}}_{\m{TV}\tp{B(0,\Qoppa^{-1/2})}} - \tnorm{\grad\mathds{1}_{\tilde{U}}}_{\m{TV}\tp{B(0,\Qoppa^{-1/2})}}}\le C\sum_{k=-M}^M\sup_{|t + k\rho|\le\ep}\tabs{\mathcal{H}^1\tp{\tcb{\Psi_{\ep,\rho,\Qoppa} = t}} - \mathcal{H}^1\tp{\tcb{\Psi = 0}}}\\\le C\tp{\ep\log\tp{1/\ep} + \rho\log\tp{1/\rho} + \Qoppa\log\tp{1/\Qoppa}}.
    \end{multline}
    Estimate~\eqref{estimates on the stream functions and vorticity in BV} is now shown.
\end{proof}

\subsection{Synthesis}\label{subsection on proofs of main theorems}

Our goal now is to combine the analyses of Sections~\ref{section on existence of admissible states} and~\ref{section on analysis near admissible states} to complete the proofs of the main theorems stated in Section~\ref{subsection on main results and discussion}.

\begin{proof}[Proof of Theorem~\ref{thm_main_1}]
    The desired result is a special consequence of Corollary~\ref{coro on existence and estimates for the regularized equation}. We take $\rho = 0$, $\Qoppa = 0$, and $0<\ep\le\ep_\star$ then define $\Psi_\ep = \Psi_{\ep,0,0}$ and $\omega_{\ep} = \omega_{\ep,0,0}$ where the latter are the smooth functions guaranteed to exist by the second item of the corollary. The neighborhood $U$ is taken as in Lemma~\ref{lem  on simple consequences of admissibility}. The validity of~\eqref{what it means to solve the steady rotating euler equations} immediately implies the first item of the theorem. Similarly, the second item of the theorem is a direct consequence of the estimate~\eqref{estimates on the stream functions and vorticity in BV}. The theorem's third item follows from~\eqref{unfortunately this needs to be cited} and Lemma~\ref{lem on regularized hevi} while the fourth item is a restatement of $\Psi_\ep\in C^\infty_{\tp{m}}\tp{\R^2}$ and $\omega_\ep\in C^\infty_{\tp{m},\m{c}}\tp{\R^2}$. Lastly, the limits~\eqref{THM_1_LIMS} hold by interpolation, equation~\eqref{THM_I_deviation}, and the uniform bounds $\sup_{0<\ep\le\ep_\star}\tnorm{\grad\Psi_\ep}_{\m{LL}\tp{\R^2}}<\infty$.
\end{proof}

\begin{proof}[Proof of Theorem~\ref{thm_main_2}]
    The idea is to pass to the limit $\ep\to0$ in Corollary~\ref{coro on existence and estimates for the regularized equation}. Let $\ep_\star$, $\rho_\star$, and $\Qoppa_\star$ be the small parameters granted by this result. Fix $\rho\in[0,\rho_\star]$ and $\Qoppa\in[0,\Qoppa_\star]$ and define the sequences $\tcb{\Psi^n_{\rho,\Qoppa}}\subset C_{\tp{m}}^\infty\tp{\Bar{B(0,\Qoppa^{-1/2})}}$ and $\tcb{\omega^n_{\rho,\Qoppa}}\subset C^\infty_{\tp{m},\m{c}}\tp{B(0,\Qoppa^{-1/2}/2)}$ via 
    \begin{equation}
        \Psi^n_{\rho,\Qoppa} = \Psi_{\ep_n,\rho,\Qoppa}\text{ and }\omega^n_{\rho,\Qoppa} = \omega_{\ep_n,\rho,\Qoppa}\text{ where }\ep_n = \ep_\star/2^n
    \end{equation}
    and $\Psi_{\cdot,\cdot,\cdot}$ and $\omega_{\cdot,\cdot,\cdot}$ are from the second item of Corollary~\ref{coro on existence and estimates for the regularized equation}.

    The uniform estimates of equation~\eqref{estimates on the solutions} paired with the far-field strategy used in and around equation~\eqref{also shows the L1 bounds} in fact allows one to deduce the uniform bounds
    \begin{equation}
        \sup_{n\in\N}\tnorm{\Psi^n_{\rho,\Qoppa}}_{\m{LL}^1\tp{\Bar{U}}}<\infty\text{ and }\sup_{n\in\N}\tnorm{\grad\Psi^n_{\rho,\Qoppa}}_{\m{LL}\tp{\Bar{B(0,\Qoppa^{-1/2})}}}<\infty.
    \end{equation}
    On the other hand, from~\eqref{estimates on the stream functions and vorticity in BV} and~\eqref{support conditions on the vorticity} we have
    \begin{equation}
        \sup_{n\in\N}\tnorm{\omega^n_{\rho,\Qoppa}}_{\tp{L^\infty\cap\m{BV}}\tp{B(0,\Qoppa^{-1/2}/2)}}<\infty\text{ and }\supp\tp{\omega^n_{\rho,\Qoppa}}\subset\Bar{U}\cup U^{\m{in}},
    \end{equation}
    where we recall that $U^{\m{in}}$ is from the fourth item of Lemma~\ref{lem  on simple consequences of admissibility}.

    By the Arzel\`{a}-Ascoli theorem, compactness properties of functions of bounded variation (e.g. Theorem 3.23 in Ambrosio, Fusco, and Pallara~\cite{MR1857292}), and extraction and relabeling of a diagonal subsequence, we are bestowed with 
    \begin{equation}\label{the result of the compactness argument}
        \Psi_{\rho,\Qoppa}\in C^1_{\tp{m}}\tp{\Bar{B(0,\Qoppa^{-1/2})}}\text{ satisfying }\grad\Psi_{\rho,\Qoppa}\in\m{LL}\tp{\Bar{B(0,\Qoppa^{-1/2})}}\text{ and }\omega_{\rho,\Qoppa}\in\tp{L^\infty\cap\m{BV}}_{\tp{m},\m{c}}\tp{\Bar{B(0,\Qoppa^{-1/2})}}
    \end{equation}
    and may suppose the satisfaction of the following strong limits (for all finite $1\le R\le \Qoppa^{-1/2}$):
    \begin{equation}\label{strong convergences}
        \tnorm{\omega^n_{\rho,\Qoppa} - \omega_{\rho,\Qoppa}}_{L^1\tp{B(0,\Qoppa^{-1/2}/2)}}\to 0\text{ and }\tnorm{\Psi^n_{\rho,\Qoppa} - \Psi_{\rho,\Qoppa}}_{C^1\tp{\Bar{B(0,R)}}}\to 0\text{ as }n\to\infty,
    \end{equation}
    as well as the following weak-$\ast$ convergences:
    \begin{equation}
        \grad\omega^n_{\rho,\Qoppa}\overset{\ast}{\rightharpoonup}\grad\omega_{\rho,\Qoppa}\text{ (in the space of Radon measures) and }\omega^n_{\rho,\Qoppa}\overset{\ast}{\rightharpoonup}\omega_{\rho,\Qoppa}\text{ (in }L^\infty\tp{B(0,\Qoppa^{-1/2})}\text{) as }n\to\infty.
    \end{equation}

    One can directly pass to the limit $n\to\infty$ in the first identity of~\eqref{what it means to solve the steady rotating euler equations} to deduce that 
    \begin{equation}\label{stream function vorticity limit 1}
        \Psi_{\rho,\Qoppa} = \bf{N}_{\Qoppa}\tsb{\omega_{\rho,\Qoppa}} - \Omega\tabs{\cdot}^2/2 - c\text{ in }\Bar{B(0,\Qoppa^{-1/2})}
    \end{equation}
    and hence, upon differentiation of~\eqref{stream function vorticity limit 1}, for almost every $x\in \Bar{B(0,\Qoppa^{-1/2})}$
    \begin{equation}\label{stream function vorticity limit 2}
        \grad\Psi_{\rho,\Qoppa}\tp{x} = -\Omega x - \f{1}{2\pi}\int_{B(0,\Qoppa^{-1/2})}\bf{K}_{\Qoppa^{-1/2}}\tp{x,y}^\perp\omega_{\rho,\Qoppa}(y)\;\m{d}y
    \end{equation}
    where $\bf{K}$ is the Biot-Savart kernel defined in~\eqref{bio savart in a domain}. On the other hand, for all $f\in C^\infty_{\m{c}}\tp{\Bar{B(0,\Qoppa^{-1/2})}}$ we may use the strong convergences in equation~\eqref{strong convergences} to justify passing to the limit $n\to\infty$ in the weak form of the second identity in equation~\eqref{what it means to solve the steady rotating euler equations}. Executing this procedure gives 
    \begin{equation}\label{stream function vorticity limit 3}
        \lim_{n\to\infty}\int_{B(0,\Qoppa^{-1/2})}\omega^n_{\rho,\Qoppa}\grad^\perp\Psi^n_{\rho,\Qoppa}\cdot\grad f = \int_{B(0,\Qoppa^{-1/2})}\omega_{\rho,\Qoppa}\grad^\perp\Psi_{\rho,\Qoppa}\cdot\grad f = 0.
    \end{equation}

    By synthesizing~\eqref{stream function vorticity limit 1}, \eqref{stream function vorticity limit 2}, and~\eqref{stream function vorticity limit 3} we see that the triple $\tp{\Omega,\Psi_{\rho,\Qoppa},\omega_{\rho,\Qoppa}}$ form a steady rotating weak solution to the Euler system, in the sense of Definition~\ref{defn od steady rotating weak solutions to the planar Euler system}, in the domain $\Bar{B(0,\Qoppa^{-1/2})}$. Therefore the first and fifth items of the theorem follows by taking $\Qoppa = 0$, $\omega_{\rho} = \omega_{\rho,0}$, and $\Psi_{\rho} = \Psi_{\rho,0}$. 

    We may also use~\eqref{conv_2} from Proposition~\ref{prop on limit identification} in conjunction with the convergences~\eqref{strong convergences} and the vorticity identity~\eqref{unfortunately this needs to be cited} to pass to the limit in the space of tempered distributions and deduce
    \begin{equation}\label{limiting vorticity function}
        \omega_{\rho,\Qoppa} = \sum_{k=-M}^M\sig_k\mathds{1}_{(0,\infty)}\tp{-\Psi_{\rho,\Qoppa} - k\rho}\text{ a.e. in }U,\;\omega_{\rho,\Qoppa} = 1\text{ a.e. in }\Bar{U^{\m{in}}},\;\omega_{\rho,\Qoppa} = 0\text{ a.e. in }\Bar{U^{\m{out}}}.
    \end{equation}

    We now turn our attention to the proof of the second item. Define the sequence of open sets $\tcb{W^k_\rho}_{k=-M}^M$ via
    \begin{equation}\label{definition of the W sets}
        W^k_\rho = \tcb{x\in U\;:\;\Psi_\rho(x) + k\rho<0}\cup\Bar{U^{\m{in}}}\text{ for all }k\in\tcb{-M,\dots,M}.
    \end{equation}
    The monotone inclusion relations~\eqref{the monotone inclusion relations} as well as the vorticity identities~\eqref{the vorticity cake} and~\eqref{other vorticity cake} are now immediate from~\eqref{limiting vorticity function} and~\eqref{definition of the W sets}. As $\Psi_\rho$ is class $\bigcap_{0<\al<1}C^{1,\al}$ with $\min_{\Bar{U}}\tabs{\grad\Psi_\rho}>0$ (see the first item of Lemma~\ref{lem on regularized hevi} and~\eqref{strong convergences}) we deduce from the identity
    \begin{equation}
        \pd W^k_\rho = \tcb{x\in U\;:\;\Psi_{\rho}(x) + k\rho = 0}
    \end{equation}
    and the implicit function theorem that $\pd W^k_\rho$ is class $\bigcap_{0<\al<1}C^{1,\al}$ as well.

    The distance estimate~\eqref{the distance estimates for the set boundaries} follows from the following observation. Fix $k\in\tcb{-M,\dots,M-1}$ and let $x\in\pd W^k_\rho$, $y\in\pd W^{k+1}_\rho$ satisfy
    \begin{equation}
        |x - y| = \m{dist}\tp{\pd W^{k}_{\rho},\pd W^{k+1}_\rho}.
    \end{equation}
    By the fundamental theorem of calculus, the lower bound
    \begin{equation}
        \rho = \Psi_\rho(x) - \Psi_{\rho}(y)\le\tnorm{\grad\Psi_\rho}_{L^\infty\tp{\R^2}}\tabs{x - y}\le C\cdot \m{dist}\tp{\pd W^{k}_{\rho},\pd W^{k+1}_\rho}
    \end{equation}
    holds.

    With the first, second, fourth, and fifth items of the theorem established, we complete the proof by justifying the third item and the limit~\eqref{THM_2_LIMS}. The first and third estimates of~\eqref{estimates of the fourth item} are immediate from the strong convergences~\eqref{strong convergences} and estimate~\eqref{estimates on the stream functions and vorticity in BV}. The claimed total variation estimate
    \begin{equation}\label{they_charged_him}
        \tabs{\tnorm{\grad\omega_\rho}_{\m{TV}\tp{\R^2}} - \tnorm{\grad\mathds{1}_{\tilde{U}}}_{\m{TV}\tp{\R^2}}}\le C\rho\log\tp{1/\rho}
    \end{equation}
    follows by an argument similar to that used in equations~\eqref{_label_1} and~\eqref{_label_2} in the proof of Corollary~\ref{coro on existence and estimates for the regularized equation}. The limit~\eqref{THM_2_LIMS} is again a consequence of~\eqref{they_charged_him}, interpolation, and a uniform bound on $\grad\Psi_\rho$ in $\m{LL}\tp{\R^2}$. We omit the repetitive details.
\end{proof}

\begin{proof}[Proof of Theorem~\ref{thm_main_3}]
    Most of the work was already carried out in the first step of the proof of Theorem~\ref{thm_main_2}. Recall that we took the limit $\ep\to0$ in Corollary~\ref{coro on existence and estimates for the regularized equation} and obtained the functions $\Psi_{\rho,\Qoppa}$, $\omega_{\rho,\Qoppa}$ of equation~\eqref{the result of the compactness argument}, which are defined for $0\le\rho\le\rho_\star$ and $0\le\Qoppa\le\Qoppa_\star$. We take $R_\star = \Qoppa_\star^{-1/2}$ and for $R_\star\le R<\infty$ define
    \begin{equation}
        \omega_R = \omega_{0,R^{-2}}\in\tp{L^\infty\cap\m{BV}}_{\tp{m},\m{c}}\tp{B(0,R/2)}\text{ and }\Psi_R = \Psi_{0,R^{-2}}\in\m{LL}_{\tp{m}}^1\tp{\Bar{B(0,R)}}.
    \end{equation}

    The satisfaction of the first and fourth items of the theorem are an immediate consequences of the proof of Theorem~\ref{thm_main_2}, specifically the paragraph following~\eqref{stream function vorticity limit 3}, since it was shown that $\tp{\Omega,\Psi_{\rho,\Qoppa},\omega_{\rho,\Qoppa}}$ is a steady rotating weak solution in the domain $\Bar{B(0,\Qoppa^{-1/2})}$. We require here only the special case of $\rho = 0$ and $\Qoppa = R^{-2}$.

    The proof of the second item is similar to that for the corresponding item of Theorem~\ref{thm_main_2}. We define
    \begin{equation}
        W_R = \tcb{x\in U\;:\;\Psi_R\tp{x}<0}\cup\Bar{U^{\m{in}}}.
    \end{equation}
    The proof of Lemma~\ref{lem on regularized hevi} guarantees the existence of $C_1\le R_{\star}$ such that $W_R\subset B(0,C_1)$. As in the proof of Theorem~\ref{thm_main_2} we observe also that
    \begin{equation}
        \pd W_R = \tcb{x\in U\;:\;\Psi_R(x) = 0}\text{ and }\min_{\Bar{U}}\tabs{\grad\Psi_R}>0
    \end{equation}
    from which it follows that $\pd W_R$ is class $\bigcap_{0<\al<1}C^{1,\al}$.

    The proofs of the estimate~\eqref{the convergence rate} and the limit~\eqref{THM_3_LIMS} are extremely similar to that of~\eqref{estimates of the fourth item} and~\eqref{THM_2_LIMS} in the proof of Theorem~\ref{thm_main_2}. We again omit the repetitive details.
\end{proof}

\begin{proof}[Proof of Theorem~\ref{thm_main_4}]
    The first item is Theorem~\ref{thm on admissibility of Kirchhoff}. The second item is Theorem~\ref{thm on admissibility of rankine vortex perturbations}.
\end{proof}

\appendix

\section{Tools from analysis}\label{appendix on tools from analysis}

\subsection{Uniform local level set coordinates}\label{appendix on uniform local level set coordinates}

In this appendix we give the proof of Lemma~\ref{lem on uniform local level set coordinates} and construct the uniform local level set coordinates. Throughout the proof we shall employ the following bracket notation. For any $d\in\N^+$ the function $\tbr{\cdot}:\C^d\to\R^+$ has the action
\begin{equation}\label{Japanese_Bracket}
    \tbr{z_1,\dots,z_d} = \tp{1 + \tabs{z_1}^2 + \dots + \tabs{z_d}^2}^{1/2}.
\end{equation}

\begin{proof}[Proof of Lemma~\ref{lem on uniform local level set coordinates}]
        We define the following positive values
    \begin{multline}\label{the small positive values}
        \del_+ = \m{dist}\tp{\Sigma,\pd U}/4,\;\upmu = \min_{\Sigma}\tp{\nu\cdot\grad\Psi},\;\upchi = \exp\tp{\max_{x\in\Bar{U}\setminus\Sigma}\tabs{\log\tabs{\Psi(x)/\m{dist}\tp{x,\Sigma}}}},\\\uplambda = \tp{\min\tcb{1,\upmu,\delta_+}/20\tbr{\tnorm{\Psi}_{C^{1,1/2}\tp{\Bar{U}}}}},\;L = \uplambda^2,\;\ell = \uplambda^3,\text{ and }\tilde{r}_0 = \tp{1/2\upchi}\uplambda^5.
    \end{multline}
    These determine the following complete metric spaces
    \begin{equation}
        \mathcal{X} = \tcb{g\in C^0\tp{\Bar{\ell Q}}\;:\;\tnorm{g}_{C^0\tp{\Bar{\ell Q}}}\le L}\text{ and }\mathcal{B} = \tcb{\psi\in C^1\tp{\Bar{U}}\;:\;\tnorm{\psi}_{C^1\tp{\Bar{U}}}\le2\tilde{r}_0}.
    \end{equation}

    Now let $x\in\Sigma$. It must be the case that either $\tabs{\pd_1\Psi(x)}\ge\tabs{\grad\Psi(x)}/2$ or $\tabs{\pd_2\Psi(x)}\ge\tabs{\grad\Psi(x)}/2$. Let us suppose that the latter inequality is satisfied. The construction that follows can be modified in the obvious way for points $x$ obeying the alternative inequality.

    Let us define the mapping $\aleph_x:\mathcal{B}\times\mathcal{X}\to\mathcal{X}$ via
    \begin{equation}\label{definition of the map we want to contract}
        \tsb{\aleph_x(\psi,g)}(y) = -\tp{1/\pd_2\Psi(x)}\tp{\Psi(x + \tp{y_1,g(y)}) - \Psi(x) - \grad\Psi(x)\cdot\tp{y_1,g(y)} + \psi(x + \tp{y_1,g(y)}) + y_1\pd_1\Psi(x) - y_2}
    \end{equation}
    for all $\tp{\psi,g}\in\mathcal{B}\times\mathcal{X}$ and $y\in\Bar{\ell Q}$. The values $\ell$, $L$, and $\tilde{r}_0$ from equation~\eqref{the small positive values} are chosen much more conservatively than what is required to ensure that $\aleph_x$ is well defined in the sense that the compositions in its definition are permissible and it maps into the stated codomain in the sense that $\tnorm{\aleph_x\tp{\psi,g}}_{C^0\tp{\Bar{\ell Q}}}\le L$. Furthermore, this selection of parameters also permit the following contraction and Lipschitz estimates:
    \begin{equation}
        \tnorm{\aleph_x\tp{\psi,g} - \aleph_x\tp{\psi,\tilde{g}}}_{C^0\tp{\Bar{\ell Q}}}\le\tp{2/3}\tnorm{g - \tilde{g}}_{C^0\tp{\Bar{\ell Q}}},\;\tnorm{\aleph_x\tp{\psi,g} - \aleph_x\tp{\tilde{\psi},g}}_{C^{0}\tp{\Bar{\ell Q}}}\le\tp{2/\upmu}\tnorm{\psi - \tilde{\psi}}_{C^0\tp{\Bar{U}}}.
    \end{equation}
    for all $\psi,\tilde{\psi}\in\mathcal{B}$ and all $g,\tilde{g}\in\mathcal{X}$. The contraction mapping principle with parameter thus applies (see, for instance, Theorem C.7 in Irwin~\cite{MR1867353}) and we are assured the existence of a Lipschitz continuous mapping $\bf{g}_x:\mathcal{B}\to\mathcal{X}$ that satisfies
    \begin{equation}\label{the fixed point identity}
        \bf{g}_x(\psi) = \aleph_x\tp{\psi,\bf{g}_x\tp{\psi}},\;\text{for all }\psi\in\mathcal{B}.
    \end{equation}
    By rearrangement of identity~\eqref{definition of the map we want to contract}, we compute that the above fixed point~\eqref{the fixed point identity} satisfies the level set parametrization identity
    \begin{equation}\label{the level set parametrization identity}
        \tp{\Psi + \psi}\tp{x + \tp{y_1,\tsb{\bf{g}_x\tp{\psi}}\tp{y}}} = y_2,
    \end{equation}
    whenever $\psi\in\mathcal{B}$ and $y\in\Bar{\ell Q}$. By using~\eqref{the level set parametrization identity} we can perform a standard argument via difference quotients to deduce that $\bf{g}_x\tp{\psi}\in C^1\tp{\Bar{\ell Q}}$ and its partial derivatives satisfy
    \begin{equation}\label{partial derivative identities}
        \pd_1\tsb{\bf{g}_x\tp{\psi}}\tp{y} = -\f{\pd_1\tp{\Psi + \psi}\tp{x + \tp{y_1,\tsb{\bf{g}_x\tp{\psi}}\tp{y}}}}{\pd_2\tp{\Psi + \psi}\tp{x + \tp{y_1,\tsb{\bf{g}_x\tp{\psi}}\tp{y}}}}\text{ and }\pd_2\tsb{\bf{g}_x\tp{\psi}}\tp{y} = \f{1}{\pd_2\tp{\Psi + \psi}\tp{x + \tp{y_1,\tsb{\bf{g}_x\tp{\psi}}\tp{y}}}}
    \end{equation}
    for all $y\in\Bar{\ell Q}$. Identities~\eqref{partial derivative identities} readily imply the estimates
    \begin{equation}\label{estimates on the derivative}
        \tnorm{\pd_1\tsb{\bf{g}_x\tp{\psi}}}_{C^0\tp{\Bar{\ell Q}}}\le\f{5}{2\upmu}\tnorm{\grad\Psi}_{C^0\tp{\Bar{U}}},\;\tnorm{\tp{\pd_2\tsb{\bf{g}_x\tp{\psi}}}^{-1}}_{C^0\tp{\Bar{U}}}\le\f{10}{9}\tnorm{\grad\Psi}_{C^0\tp{\Bar{U}}},\;\tnorm{\pd_2\tsb{\bf{g}_x\tp{\psi}}}_{C^0\tp{\Bar{U}}}\le\f{9}{4\upmu};
    \end{equation}
    That the map $\mathcal{B}\ni\psi\mapsto\bf{g}_x\tp{\psi}\in C^1\tp{\Bar{\ell Q}}$ is continuous can also be deduced from~\eqref{partial derivative identities}.

    From the fixed point functions above we now construct a continuous family of diffeomorphisms. Define $\bf{G}_x:\mathcal{B}\to C^1\tp{\Bar{\ell Q};\R^2}$ via
    \begin{equation}\label{the diffeo defn}
        \tsb{\bf{G}_x\tp{\psi}}(y) = x + \tp{y_1,\tsb{\bf{g}_x\tp{\psi}}\tp{y}}\text{ for }y\in\Bar{\ell Q},\;\psi\in\mathcal{B}.
    \end{equation}
    An elementary calculation using~\eqref{partial derivative identities} and the graphical structure of $\bf{G}_x\tp{\psi}$ shows that the map $\bf{G}_x\tp{\psi}:\Bar{\ell Q}\to\R^2$ is a diffeomorphism onto its image; moreover, $\det\grad\tsb{\bf{G}_x\tp{\psi}} = \pd_2\tsb{\bf{g}_x\tp{\psi}}>0$ and we have the derivative estimates for every $y\in\Bar{\ell Q}$
    \begin{equation}
        \tabs{\grad\tsb{\bf{G}_x\tp{\psi}}\tp{y}} + \tabs{\tp{\grad\tsb{\bf{G}_x\tp{\psi}}}^{-1}\tp{y}}\le R\text{ for }R = \tp{10/\min\tcb{1,\upmu}}\tbr{\tnorm{\grad\Psi}_{C^0\tp{\Bar{U}}}},
    \end{equation}
    where the $-1$ exponent above refers to matrix inversion. We also compute from the formulae~\eqref{partial derivative identities} that
    \begin{equation}
        \tabs{\pd_1\tsb{\bf{G}_x\tp{\psi}}\tp{y}} = \tp{1 + \tabs{\pd_1\tsb{\bf{g}_x\tp{\psi}}\tp{y}}^2}^{1/2} = \tabs{\grad\tp{\Psi + \psi}\tp{\tsb{\bf{G}_x\tp{\psi}}\tp{y}}}\det\grad\tsb{\bf{G}_x\tp{\psi}}\tp{y}.
    \end{equation}

    We now need to find a universal radius for a ball centered at $x$ which remains a subset of the images of the family of diffeomorphisms constructed above. We claim that we can take
    \begin{equation}\label{the ball man}
        \del_- = \ell\min\tcb{1,\upmu}/40\tbr{\tnorm{\Psi}_{C^{1,1/2}\tp{\Bar{U}}}}^2\text{ and satisfy }B(x,2\del_-)\subset\tsb{\bf{G}_x\tp{\psi}}\tp{\Bar{\ell Q}}
    \end{equation}
    for all $\psi\in\mathcal{B}$. The first thing we need to estimate is the points $\tsb{\bf{g}_x\tp{\psi}}\tp{0}\in\R$. From the fixed point identity~\eqref{the fixed point identity} and the definition~\eqref{definition of the map we want to contract} we find that 
    \begin{equation}\label{for what is a man}
        \tabs{\tsb{\bf{g}_x\tp{\psi}}\tp{0}}\le\tp{2/\upmu}\tp{\tnorm{\grad\Psi}_{C^{1/2}\tp{\Bar{U}}}\tabs{\tsb{\bf{g}_x\tp{\psi}}\tp{0}}^{3/2} + 2\tilde{r}_0}.
    \end{equation}
    On the other hand we know that $\tabs{\tsb{\bf{g}_x\tp{\psi}}(0)}^{1/2}\le L^{1/2} = \uplambda$. So inserting this bound into the right hand side of the above~\eqref{for what is a man} and solving for $\tabs{\tsb{\bf{g}_x\tp{\psi}}(0)}$ gives:
    \begin{equation}\label{location of the zero}
        \tabs{\tsb{\bf{g}_x\tp{\psi}}(0)}\le\f{(4/\upmu)\tilde{r}_0}{1 - (2/\upmu)\tnorm{\grad\Psi}_{C^{1/2}\tp{\Bar{U}}} L^{1/2}}\le\f{20}{9\upmu\upchi}\uplambda^5.
    \end{equation}

    The next ingredient is a deviation estimate. We use, for each fixed $y_1\in[-\ell,\ell]$, the fundamental theorem of calculus applied to the map $[-\ell,\ell]\ni y_2\mapsto\tsb{\bf{g}_x\tp{\psi}}\tp{y_1,y_2} - \tsb{\bf{g}_x\tp{\psi}}\tp{y_1,0}\in\R$ in conjunction with the positive lower bound on $\pd_2\tsb{\bf{g}_x\tp{\psi}}$ from~\eqref{estimates on the derivative}; doing this gives the inclusion
    \begin{equation}\label{intermediate inclusion}
        \tsb{\bf{G}_x\tp{\psi}}\tp{\Bar{\ell Q}} - x \supseteq\tcb{\tp{y_1,h + \tsb{\bf{g}_x\tp{\psi}}\tp{y_1,0}}\;:\;|y_1|\le\ell,\;|h|\le 9\ell/10\tnorm{\grad\Psi}_{C^0\tp{\Bar{U}}}}.
    \end{equation}
    The set on the right hand side of~\eqref{intermediate inclusion} is the region between the graphs of Lipschitz functions. Therefore we can deduce, for the values $h_\star = 9\ell/10\tnorm{\grad\Psi}_{C^0\tp{\Bar{U}}}$ and $s_\star = 5\tnorm{\grad\Psi}_{C^0\tp{\Bar{U}}}/2\upmu$, the following polygonal region inclusion:
    \begin{equation}\label{intermediate inclusion 2}
        \tsb{\bf{G}_x\tp{\psi}}\tp{\Bar{\ell Q}} - x - \tp{0,\tsb{\bf{g}_x\tp{\psi}}\tp{0}} \supseteq\tcb{y\in\R^2\;:\;|y_1|\le\ell,|y_2|\le h_\star - s_\star|y_1|}.
    \end{equation}
    Elementary planar geometry now permits us to conclude that the right hand side polygonal region above contains the ball centered at the origin of radius
    \begin{equation}\label{updelta lower bound}
        \updelta = \min\tcb{\ell,h_\star/\tbr{s_\star}} \ge 9\ell\min\tcb{1,\upmu}/50\tbr{\tnorm{\Psi}_{C^{1,1/2}\tp{\Bar{U}}}}^2.
    \end{equation}
    By synthesizing equations~\eqref{location of the zero}, \eqref{intermediate inclusion 2}, and~\eqref{updelta lower bound} with the estimate
    \begin{equation}
        \updelta - \tabs{[\bf{g}_x\tp{\psi}]\tp{0}}\ge\tp{157/900}\tp{\ell\min\tcb{1,\upmu}/\tbr{\tnorm{\Psi}_{C^{1,1/2}\tp{\Bar{U}}}}^2}\ge 2\del_-,
    \end{equation}
    we conclude that the inclusion asserted in~\eqref{the ball man} indeed holds.

    We are now ready to complete the construction of the uniform local level set coordinates. The set $\tilde{\Sigma} = \tcb{x\in\Bar{U}\;:\;\m{dist}(x,\Sigma)\le\del_-}$ is compact and is covered by the collection of open balls $\tcb{B(x,2\del_-)}_{x\in\Sigma}$ and so there exists $N\in\N$ and a finite collection of points $\tcb{x_i}_{i=1}^N\subset\Sigma$ with $\tilde{\Sigma}\subset\bigcup_{i=1}^NB(x_i,2\del_-)$. 
    
    Define $\tilde{\ep}_0 = \min\tcb{\ep_0,\tilde{r}_0}$. If $0<\ep\le\tilde{\ep}_0$, $\psi\in\mathcal{B}$, and $x\in\Bar{U}$ is such that $\tabs{\Psi(x) + \psi(x)}\le\ep$, then
    \begin{equation}
        \tabs{\Psi(x)}\le3\tilde{r}_0 = \f{3\ell}{2\upchi}\uplambda^2<\f{\del_-}{\upchi}\imp\m{dist}\tp{x,\Sigma}<\delta_-\imp x\in\bigcup_{i=1}^NB(x_i,2\del_-).
    \end{equation}
    The first item is now shown.

    For the second item we simply define for $\psi\in\mathcal{B}$ and $i\in\tcb{1,\dots,N}$ the map $G^\psi_i\in C^1\tp{\Bar{\ell Q};\R^2}$ via $G^\psi_i = \bf{G}_{x_i}\tp{\psi}$ where the latter map is defined in~\eqref{the diffeo defn}. The claimed properties~\eqref{p_1}, \eqref{p_2}, \eqref{p_3}, \eqref{p_4}, and~\eqref{p_5} along with the third item follow from the previous construction and elementary calculations.
\end{proof}

\subsection{A logarithmic interpolation inequality}\label{appendix on an interpolation inequality}

\begin{lemC}\label{lem on interpolation}
    Let $U\subset\R^d$ be a smooth bounded domain. There exists constants $C,c\in\R^+$ such that the following hold for all $f\in\m{LL}^1\tp{\Bar{U}}$:
    \begin{enumerate}
        \item We have $\tnorm{f}_{\m{LL}\tp{\Bar{U}}}\le c\tnorm{f}_{\m{LL}^{1}\tp{\Bar{U}}}$.
        \item We have
        \begin{equation}\label{log-lip interpolation estimates__}
            \tnorm{f}_{C^1\tp{\Bar{U}}}\le C\tnorm{f}_{\m{LL}\tp{\Bar{U}}}\tp{1 + \log\tp{c\tnorm{f}_{\m{LL}^1\tp{\Bar{U}}}/\tnorm{f}_{\m{LL}\tp{\Bar{U}}}}}.
        \end{equation}
    \end{enumerate}
\end{lemC}
\begin{proof}
    Since $U$ is a domain with $C^{1,1}$ boundary, the extension operator $\mathfrak{E}$ constructed for Theorem 5 of Chapter 6 in Stein~\cite{MR290095} can be shown to map boundedly
    \begin{equation}\label{the stein extension operator}
        \mathfrak{E}:X\tp{\Bar{U}}\to X\tp{\R^2}\cap\tcb{f\;:\;\m{supp}\tp{f}\subset B(0,R)}\text{ for }X = \tcb{\m{LL},C^1,\m{LL}^1},
    \end{equation}
    for some fixed $R<\infty$ chosen sufficiently large depending on $U$. By using $\mathfrak{E}$ appropriately, we can reduce to proving the first and second items for the case $U = \R^2$; in this case the first item is trivial.

    To prove the second item (when $U = \R^2$) and $f\in\m{LL}^1\tp{\Bar{U}}\setminus\tcb{0}$, we let $0<r<1/e$ and begin with the following pointwise formula for $x\in\R^2$:
    \begin{equation}
        \grad f(x) = -\f{1}{\pi r^2}\int_{\pd B(x,r)}\tp{f(y) - f(x)}\f{y - x}{|y - x|}\;\m{d}y + \f{1}{\pi r^2}\int_{B(x,r)}\tp{\grad f(x) - \grad f(y)}\;\m{d}y.
    \end{equation}
    In turn, we arrive at the pointwise estimate
    \begin{equation}\label{intermediate bound}
        \tabs{\grad f(x)}\le2\log\tp{1/r}\tnorm{f}_{\m{LL}\tp{\R^2}} + r\log(1/r)\tnorm{f}_{\m{LL}^1\tp{\R^2}}.
    \end{equation}
    The estimate~\eqref{log-lip interpolation estimates__} follows from~\eqref{intermediate bound} by selecting
    \begin{equation}
        r = \min\tcb{1/e,\tnorm{f}_{\m{LL}\tp{\R^2}}/c\tnorm{f}_{\m{LL}^1\tp{\R^2}}}.
    \end{equation}
\end{proof}


\section*{Acknowledgments}

The authors would like to express their gratitude to Peter Constantin, Javier G\'{o}mez-Serrano, and Miles Wheeler, for helpful discussions, comments, and suggestions.




\bibliographystyle{abbrv}
\bibliography{bib.bib}
\end{document}